\documentclass[aop,MSNbibl,seceqn,citesort,dvips]{arximspdf}
\usepackage{mathrsfs}

\doi{10.1214/14-AOP907}
\volume{43}
\issue{1}
\pubyear{2015}
\firstpage{339}
\lastpage{404}

\newcommand{\rrVert}{\Vert}
\newcommand{\rrvert}{\vert}
\newcommand{\llVert}{\Vert}
\newcommand{\llvert}{\vert}
\newcommand{\C}{\mathbb{C}}
\newcommand{\M}{\mathbb{M}}
\newcommand{\N}{\mathbb{N}}
\newcommand{\Q}{\mathbb{Q}}
\newcommand{\R}{\mathbb{R}}
\newcommand{\cA}{{{\mathcal A}}}
\newcommand{\cB}{{{\mathcal B}}}
\newcommand{\cE}{{{\mathcal E}}}
\newcommand{\cG}{{{\mathcal G}}}
\newcommand{\cL}{{{\mathcal L}}}
\newcommand{\cN}{{{\mathcal N}}}
\newcommand{\cV}{{{\mathcal V}}}
\newcommand{\ff}{\mathbf{f}}
\newcommand{\ii}{\mathbf{i}}
\newcommand{\mm}{\mathbf{m}}
\newcommand{\nn}{\mathbf{n}}
\newcommand{\rr}{\mathbf{r}}
\newcommand{\ggamma}{\boldsymbol{\gamma}}
\newcommand{\mmu}{\boldsymbol{\mu}}
\newcommand{\ppi}{\boldsymbol{\pi}}
\newcommand{\sfg}{{\mathsf g}}
\newcommand{\sfA}{{\mathsf A}}
\newcommand{\sfB}{{\mathsf B}}
\newcommand{\sfC}{{\mathsf C}}
\newcommand{\sfD}{{\mathsf D}}
\newcommand{\sfF}{{\mathsf F}}
\newcommand{\sfH}{{\mathsf H}}
\newcommand{\sfP}{{\mathsf P}}
\newcommand{\frh}{{\mathfrak h}}
\newcommand{\rme}{{\mathrm e}}
\newcommand{\rmB}{{\mathrm B}}
\newcommand{\rmC}{{\mathrm C}}
\newcommand{\rmD}{{\mathrm D}}
\newcommand{\rmI}{{\mathrm I}}
\newcommand{\rmR}{{\mathrm R}}
\newcommand{\rmS}{{\mathrm S}}
\newcommand{\rmT}{{\mathrm T}}
\newcommand{\supp}{\operatorname{supp}} 
\newcommand{\Lip}{\operatorname{Lip}} 
\renewcommand{\d}{{\mathrm d}}
\newcommand{\dt}{{\d t}}
\newcommand{\Leb}[1]{{\mathscr L}^{#1}}
\newcommand{\down}{\downarrow} 
\newcommand{\up}{\uparrow}
\newcommand{\weakto}{\rightharpoonup}

\newcommand{\eps}{\varepsilon}
\newcommand{\nchi}{\chi}
\newcommand{\AC}[3]{\mathrm{AC}^{#1}(#2;#3)}
\newcommand{\entv}{\mathrm{Ent}_{\mm}} 
\renewcommand{\mm}{\mathfrak m}
\renewcommand{\nn}{\mathfrak n}
\newcommand{\Probabilities}[1]{\mathscr P(#1)} 
\newcommand{\prob}{\Probabilities}
\newcommand{\ProbabilitiesTwo}[1]{\mathscr P_2(#1)}
\newcommand{\probt}{\ProbabilitiesTwo}
\newcommand{\res}{\mathop{\hbox{\vrule height 7pt width .5pt depth 0pt
\vrule height .5pt width 6pt depth 0pt}}\nolimits}

\renewcommand{\C}{\mathsf{Ch}}
\newcommand{\EVI}{{\mathrm{EVI}}}
\newcommand{\RCD}[2]{\mathrm{RCD}(#1,#2)}
\newcommand{\CD}[2]{\mathrm{CD}(#1,#2)}
\newcommand{\BE}[2]{\mathrm{BE}(#1,#2)} 

\newcommand{\DeltaE}{\Delta_\cE}
\renewcommand{\C}{\mathsf{Ch}}
\newcommand{\heat}[2]{\sP{#1}#2}

\newcommand{\weaksto}{\stackrel{\star}{\rightharpoonup}}
\newcommand{\sfz}{\mathsf z}
\newcommand{\cZ}{\mathcal Z}
\renewcommand{\cV}{\mathbb{V}}
\renewcommand{\cG}{\mathbb{G}}
\newcommand{\cGz}{\cG_{\infty}}

\newcommand{\cVu}{\cV^1_{\infty}}
\newcommand{\cVd}{\cV^2_{\infty}}
\newcommand{\sfdZ}{\mathsf{d}}
\newcommand{\mmZ}{\mm}
\newcommand{\cEZ}{\cE}
\newcommand{\cVZ}{\cV}

\newcommand{\sP}[1]{\sfP_{#1}}
\newcommand{\sH}[1]{\sfH_{#1}}
\newcommand{\tsP}[1]{\tilde{\sfP}_{#1}}
\newcommand{\MD}{\textup{MD}}
\newcommand{\ED}{\textup{ED}}

\newcommand{\due}{}
\newcommand{\uno}{\tfrac12}
\newcommand{\mezzo}{}
\newcommand{\otto}{4}
\newcommand{\st}{s,t}
\newcommand{\eqref}[1]{(\ref{#1})}
\renewcommand{\textsc}{}
\newtheorem{theorem}{Theorem}[section]
\newtheorem{lemma}[theorem]{Lemma}
\newtheorem{corollary}[theorem]{Corollary}
\newproclaim{definition}[theorem]{Definition}
\newproclaim{remark}[theorem]{Remark}
\newproclaim{condition}{Condition}
\newtheorem{proposition}[theorem]{Proposition}
\newcommand{\fraca}[2]{{#1}/{#2}}
\makeatother

\begin{document}
\begin{frontmatter}

\title{Bakry--\'Emery curvature-dimension condition and Riemannian
Ricci curvature bounds}
\runtitle{Bakry--\'Emery curvature-dimension condition}

\begin{aug}
\author[A]{\fnms{Luigi} \snm{Ambrosio}\thanksref{T1,T3}\ead[label=e1]{l.ambrosio@sns.it}},
\author[B]{\fnms{Nicola} \snm{Gigli}\ead[label=e2]{giglin@math.jussieu.fr}}
\and
\author[C]{\fnms{Giuseppe} \snm{Savar\'e}\corref{}\thanksref{T3}\ead[label=e3]{giuseppe.savare@unipv.it}\ead[label=u3,url]{http://www.imati.cnr.it/savare}}
\runauthor{L. Ambrosio, N. Gigli and G. Savar\'e}
\affiliation{Scuola Normale Superiore, Pisa, Universit\'e Pierre et
Marie Curie--Paris 6 and~Universit\`a di Pavia}
\address[A]{L. Ambrosio\\
Scuola Normale Superiore\\
Piazza dei Cavalieri, 7\\
56126 Pisa\\
Italy\\
\printead{e1}} 
\address[B]{N. Gigli\\
Institut de Math\'{e}matiques de Jussieu\\
Case 247, 4 Place Jussieu \\
75252 Paris Cedex 05\\
France\\
\printead{e2}}
\address[C]{G. Savar\'e\\
Dipartimento di Matematica\\
Universit\`a di Pavia\\
Via Ferrata 1\\
27100 Pavia\\
Italy\\
\printead{e3}\\
\printead{u3}}
\end{aug}
\thankstext{T1}{Supported by ERC ADG GeMeThNES.}
\thankstext{T3}{Supported in part by
PRIN10-11 grant from MIUR for the project \emph{Calculus of Variations}.}

\received{\smonth{9} \syear{2012}}
\revised{\smonth{12} \syear{2013}}

%
\begin{abstract}
The aim of the present paper is to bridge the gap between
the Bakry--\'Emery and the Lott--Sturm--Villani
approaches to provide synthetic and abstract notions of lower Ricci
curvature bounds.

We start from a strongly local Dirichlet form $\cE$ admitting a \emph
{Carr\'e du champ} $\Gamma$
in a Polish measure space $(X,\mm)$ and a canonical distance ${\mathsf d}
_\cE$
that induces the original topology of $X$.
We
first characterize the distinguished class of
\emph{Riemannian Energy measure spaces}, where $\cE$ coincides with
the Cheeger energy
induced by ${\mathsf d}_\cE$ and
where every function $f$ with $\Gamma(f)\le1$ admits
a continuous representative.

In such a class, we show
that if $\cE$ satisfies a suitable weak form of
the \emph{Bakry--\'Emery curvature dimension condition}
$\BE K\infty$ then the metric measure space $(X,{\mathsf d},\mm)$ satisfies
the Riemannian Ricci curvature bound $\RCD K\infty$ according to
[\textit{Duke Math. J.} \textbf{163}  (2014) 1405--1490], thus showing the
equivalence of the two notions.

Two applications are then proved: the tensorization property for Riemannian
Energy spaces satisfying the Bakry--\'Emery $\BE KN$ condition
(and thus the corresponding one for $\RCD K\infty$ spaces without
assuming nonbranching)
and the stability of $\BE KN$ with respect to
Sturm--Gromov--Hausdorff convergence.
\end{abstract}

%

\pdfkeywords{49Q20, 47D07, 30L99, Ricci curvature, Barky-Emery condition, metric measure space, Dirichlet form,
Gamma calculus}
\end{frontmatter}

\tableofcontents[level=2]


\section{Introduction}\label{sec1}

Besides its obvious geometric relevance,
spaces with Ricci curvature bounded from
below play an important role in many probabilistic and analytic
investigations that reveal various deep connections between different
fields.

Starting from the celebrated paper by \textsc{Bakry--\'Emery} \cite
{Bakry-Emery84},
the curvature-dimension condition based on
the $\Gamma_2$-criterium in Dirichlet spaces
provides crucial tools for proving
refined estimates on Markov semigroups and many
functional inequalities, of Poincar\'e, Log-Sobolev, Talagrand
and concentration type
(see, e.g., \cite
{Ledoux01,Ledoux04,Ledoux11,Bakry06,Ane-et-al00,Bakry-Gentil-Ledoux14}).

This general functional-analytic approach is also well suited to deal
with genuinely infinite dimensional settings with
applications to Wiener measure on the paths of Brownian motion
with values in a Riemannian manifold, as in, for example, \cite{Bakry-Ledoux96}.
In fact, Ricci curvature also arises in Bismut-type formula
\cite{Hsu02} and its applications to gradient estimates
\cite{Arnaudon-Plank-Thalmaier03,Arnaoudon-Driver-Thalmaier07},
and to the construction of couplings between
Brownian motions \cite{Kendall86,Kuwada-Sturm13}.

The importance of curvature bounds in the framework of optimal transport
has been deeply analyzed in
\cite{Otto-Villani00,Cordero-McCann-Schmuckenschlager01,Sturm-VonRenesse05}.
These and other important results led \textsc{Sturm}
\cite{Sturm06I,Sturm06II}
and \textsc{Lott--Villani} \cite{Lott-Villani09}
to introduce a new synthetic notion of the curvature-dimension
condition,
in the general framework of metric-measure spaces.

The aim of the present paper is to bridge the gap between
the Bakry--\'Emery and the Lott--Sturm--Villani
approaches to provide synthetic and abstract notions of lower Ricci
curvature bounds.
In order to make this statement more precise, let us briefly
review the main points of both settings.

\textit{The Bakry--\'Emery condition $\BE KN$:
Dirichlet forms and $\Gamma$-calculus.}
The first approach is based on the functional $\Gamma$-calculus
developed by
\textsc{Bakry--\'Emery} since \cite{Bakry-Emery84};
see \cite{Bakry06,Bakry-Ledoux06}.

A possible starting point is a
local and symmetric Dirichlet form $\cE$ on the measure space
$(X,\cB,\mm)$ with dense domain $D(\cE)\subset L^2(X,\mm)$,
and the associated Markov semigroup $(\sP t)_{t\ge0}$
on $L^2(X,\mm)$ with generator $\DeltaE$
(general references are \cite
{Fukushima-Oshima-Takeda11,Ma-Rockner92,Bouleau-Hirsch91}).
In a suitable algebra $\cA$
of functions dense in the domain $D(\DeltaE)$ of $\DeltaE$
one introduces the \emph{Carr\'e du champ}
\[
\Gamma(f,g):=\tfrac{1}{2} \bigl(\DeltaE(fg)-f\DeltaE g-g\DeltaE f \bigr),
\qquad f,g\in\cA,
\]
related to $\cE$ by the local representation formula
\[
\cE(f,g)=\int_X \Gamma(f,g)\,\d\mm\qquad
\mbox{for every }f,g\in\cA.
\]
One also assumes that $\DeltaE$ is a
diffusion operator, that is, with the notation $\Gamma(f):=\Gamma
(f,f)$, it holds
%
\begin{eqnarray}
\DeltaE\phi(f)=\phi'(f)\DeltaE f+\phi''(f)
\Gamma(f)
\nonumber
\\
\eqntext{\mbox{for every }f\in\cA,\phi\in\rmC^2(\R) \mbox{ with
bounded derivatives}.}
\end{eqnarray}
The model example is provided by a smooth Riemannian manifold
$(\mathbb M^d,\sfg)$ endowed with the measure
$\mm:=\rme^{-V}\operatorname{Vol}_\sfg$ for a given smooth potential\break $V\dvtx
\mathbb M^d\to\R$. In this case, one typically chooses
$\cA=\rmC^\infty_c(\mathbb M^d)$ and
\begin{eqnarray*}
\cE(f,g)&=&\int_{\mathbb M^d} \langle\nabla f,\nabla g
\rangle_\sfg\,\d\mm\quad\mbox{so that}
\\
\Gamma(f)&=&\llvert \nabla f\rrvert _\sfg^2\quad
\mbox{and}\quad \DeltaE=\Delta_\sfg-\langle\nabla V,\nabla\cdot\rangle
_\sfg,
\end{eqnarray*}
where $\Delta_\sfg$ is the usual Laplace--Beltrami operator on $\M$.
This fundamental example shows that $\Gamma$ carries the
metric information of $\mathbb M^d$, since one can recover the
Riemannian
distance ${\mathsf d}_\sfg$ in $\mathbb M^d$ by the formula
%
\begin{equation}
\label{eq100} \mathsf{d}_\sfg(x,y)=\sup \bigl\{\psi(y)-\psi(x)\dvtx
\psi\in\cA, \Gamma(\psi) \le 1 \bigr\},\qquad x,y\in\mathbb M^d.
\end{equation}
A further iteration yields the $\Gamma_2$ operator, defined by
%
\begin{equation}
\label{eq97} 2\Gamma_2(f)=\DeltaE\Gamma(f)-2\Gamma(f,\DeltaE f),
\qquad f\in\cA.
\end{equation}
In the above example, Bochner's formula yields
\[
\Gamma_2(f)=\llVert \operatorname{Hess}_\sfg f
\rrVert _\sfg^2+ (\operatorname{Ric}_\sfg+
\operatorname{Hess}_\sfg V ) (\nabla f,\nabla f),
\]
and one obtains the fundamental inequality
%
\begin{equation}
\label{eq99} \Gamma_2(f)\ge K \Gamma(f)+\frac{1}N(\DeltaE
f)^2\qquad\mbox{for every }f\in\cA,
\end{equation}
if the quadratic form associated to the tensor
$\operatorname{Ric}_\sfg+\operatorname{Hess}_\sfg V $ is bounded from below by
$K \sfg+\frac{1}{N-d}\nabla V\otimes\nabla V$ for some
$K\in\R$ and $N> d$.
When $V\equiv0$, it is possible to show that
$(\mathbb M^d,\sfg)$ has Ricci curvature bounded from below by $K$
iff \eqref{eq99} is satisfied for $N\ge d$.

It is then natural to
use \eqref{eq99} as a definition of curvature-dimension bounds
even in the abstract setting:
it is the so-called \emph{Bakry--\'Emery curva\-ture-dimension condition},
that we denote here by $\BE KN $.

One of the most remarkable applications of \eqref{eq99} is
provided by pointwise gradient estimates for the
Markov semigroup (see, e.g., \cite{Bakry06,Bakry-Ledoux06} for
relevant and deep applications).
Considering here only the case $N=\infty$, \eqref{eq99} yields
%
\begin{equation}
\label{eq101} \Gamma(\sP t f)\le\rme^{-2K t} \sP t \bigl(\Gamma(f) \bigr)
\qquad \mbox{for every }f\in\cA,
\end{equation}
a property that is essentially equivalent to $\BE K\infty$
(we refer to \cite{Wang11} for other formulations of
$\BE KN $ for Riemannian manifolds, see also the
next Section~\ref{subsecBE})
and involves only first order ``differential'' operators.

Up to the choice of an appropriate functional setting
[in particular, the algebra $\cA$ and
the distance ${\mathsf d}$ associated to $\Gamma$ as
in \eqref{eq100} play a crucial role],
$\Gamma$-calculus and curvature-dimension inequalities
provide a very powerful tool to establish many
functional inequalities and geometric properties, often in sharp form.

\textit{Lower Ricci curvature bounds by
optimal transport: The $\CD K\infty$ condition.}
A completely different approach to lower Ricci bounds
has been recently proposed by \textsc{Sturm} \cite
{Sturm06I,Sturm06II} and
\textsc{Lott--Villani} \cite{Lott-Villani09}:
here, the abstract setting
is provided by metric measure spaces $(X,{\mathsf d},\mm)$, where
$(X,{\mathsf d})$ is a separable, complete and length metric space and
$\mm$ is a nonnegative $\sigma$-finite Borel measure.
Just for simplicity, in this \hyperref[sec1]{Introduction} we also assume
$\mm(X)<\infty$, but the theory covers the case
of a measure satisfying the exponential growth condition
$\mm(B_r(x))\le M\exp(c r^2)$ for some constants $M,c\ge0$.

The Lott--Sturm--Villani theory (LSV in the following) is
based on the notion of
displacement interpolation \cite{McCann97},
a powerful tool of optimal transportation
that allows one to extend the notion of geodesic interpolation
from the state space $X$ to
the space of Borel probability measures $\probt X$
with finite quadratic moment.
Considering here
only the case $N=\infty$,
a metric measure space $(X,{\mathsf d},\mm)$
satisfies the LSV lower Ricci curvature bound $\CD K\infty$
if
the relative entropy functional
%
\begin{equation}
\label{eq102} \operatorname{Ent}_{\mm}(\mu):=\int_X
f\log f\,\d\mm,\qquad\mu=f\mm,
\end{equation}
is displacement $K$-convex in the Wasserstein space
$(\probt X, W_2)$
(see \cite{Villani09,AGS08} and Section~\ref{subsecpreliminaries} below).
This definition is consistent with
the Riemannian case \cite{Sturm-VonRenesse05},
and thus equivalent to $\BE K\infty$
in such a smooth framework.

Differently from the Bakry--\'Emery's approach,
the LSV theory does not originally involve energy functionals or Markov
semigroups, but it is intimately connected to the
metric ${\mathsf d}$ (through the notion of displacement interpolation) and
to the measure $\mm$ [through the entropy functional \eqref{eq102}].
Besides many useful geometric and functional applications of this
notion \cite{Lott-Villani-Poincare,Rajala11,Gigli10},
one of its strongest features is its stability under measured\vadjust{\goodbreak}
Gromov--Hausdorff
convergence \cite{Fukaya87},
also in the weaker transport-formulation
proposed by \textsc{Sturm}
\cite{Sturm06I}.\vadjust{\goodbreak}

Starting from the $\CD K\infty$ assumption,
one can then construct an evolution semigroup {$(\sfH_t)_{t\ge0}$}
on the
convex subset of $\probt X$ given by
probability measures with finite entropy
\cite{Gigli10}: it is the metric gradient flow of the entropy functional
in $\probt X$ \cite{AGS08}. Since also Finsler geometries
(as in the flat case of $\R^d$ endowed with a non-Euclidean norm)
can satisfy the $\CD K\infty$ condition,
one cannot hope in such a general setting that
{$\sfH_t$} are linear operators. Still, {$(\sfH_t)_{t\ge0}$} can be
extended to
a continuous semigroup
of contractions in
$L^2(X,\mm)$ (and in any $L^p(X,\mm)$-space),
which can also be characterized
as the $L^2(X,\mm)$-gradient flow
{$(\sP t)_{t\ge0}$}
of a convex and $2$-homogeneous
functional, the Cheeger energy
\cite{Cheeger00,AGS11a}, Section~4.1, Remark~4.7,
%
\begin{eqnarray}
\label{eq103} &&\C(f)
\nonumber
\\[-4pt]
\\[-12pt]
&&\qquad :=\inf \biggl\{\liminf_{n\to\infty}\frac{1}2\int
_X \llvert \rmD f_n\rrvert ^2\,\d
\mm\dvtx f_n\in\Lip_b(X), f_n\to f\mbox{ in
}L^2(X,\mm) \biggr\}
\nonumber
\end{eqnarray}
[here, $\Lip_b(X)$ denotes the space of
Lipschitz and bounded real functions defined in $X$
and $\llvert \rmD f\rrvert $ is the local Lipschitz constant, or slope,
of the Lipschitz function $f$, see Section~\ref{subsecpreliminaries}].

The remarkable identification {between $(\sfH_t)_{t\ge0}$ and $(\sP
t)_{t\ge0}$}
has been first
proposed and proved in Euclidean spaces by a seminal paper of
\textsc{Jordan--Kinderleher--Otto}
\cite{Jordan-Kinderlehrer-Otto98} and then extended to
Riemannian manifolds 
\cite{Erbar10,Villani09},
Hilbert spaces \cite{Ambrosio-Savare-Zambotti09},
Finsler spaces \cite{Sturm-Ohta-CPAM}, Alexandrov spaces
\cite{GigliKuwadaOhta10} and eventually
to $\CD K\infty$ metric measure spaces
\cite{AGS11a}.

\textit{Spaces with Riemannian Ricci
curvature bounded from
below:
The\break $\operatorname{RCD}(K,\infty)$ condition.}
Having the energy functional \eqref{eq103} and
the contraction semigroup $(\sP t)_{t\ge0}$ at our disposal, it is then
natural
to investigate when LSV spaces
satisfy $\BE K\infty$.
In order to attack this question, one has of course to clarify
when the Cheeger energy \eqref{eq103} is a Dirichlet (thus \emph{quadratic})
form on $L^2(X,\mm)$ [or, equivalently, when $(\sP t)_{t\ge0}$ is a semigroup
of \emph{linear} operators] and when this property is
also stable under Sturm--Gromov--Hausdorff convergence.

One of the most important results of \cite{AGS11b}
(see also \cite{AGMR12} for general $\sigma$-finite measures)
is that $\CD K\infty$ spaces with a quadratic
Cheeger energy can be equivalently characterized as
those metric measure spaces
where there
exists the Wasserstein gradient flow {$(\sfH_t)_{t\ge0}$}
of the entropy functional
\eqref{eq102} in the $\operatorname{EVI}_K$-sense.
This condition means that
for all initial data $\mu\in\probt X$
with $\supp\mu\subset\supp\mm$
there exists a locally Lipschitz curve $t\mapsto{\sfH_t} \mu\in
\probt X$
satisfying the
evolution variational inequality:
%
\begin{eqnarray}
\label{eqEVIK1} \frac{\d}{\d t}\frac{W_2^2({\sfH_t}\mu,\nu)}2+\frac{K}2W_2^2({
\sfH_t}\mu,\nu)+\entv({\sfH_t}\mu)\leq\entv(\nu)
\nonumber
\\[-8pt]
\\[-8pt]
\eqntext{\mbox{for a.e. $t\in (0,\infty)$}}
\end{eqnarray}
for all $\nu\in\probt{X}$ with $\entv(\nu)<\infty$.

Such a condition is denoted by
$\RCD K\infty$ and it is stronger than\break
$\CD K\infty$,
since the existence of an $\operatorname{EVI}_K$ flow solving
\eqref{eqEVIK1}
yields {both} the geodesic $K$-convexity of the entropy functional
$\entv$
\cite{Daneri-Savare08} and the linearity of $(\sH{t})_{t\ge0}$
\cite{AGS11b}, Theorem~5.1,
but it is still stable under Sturm--Gromov--Hausdorff convergence.
When it is satisfied, the metric measure space $(X,{\mathsf d},\mm)$
is called in \cite{AGS11b} a space with \emph{Riemannian} Ricci
curvature bounded from below by $K$.

In $\RCD K\infty$-spaces the Cheeger energy is
associated to a strongly local Dirichlet form $\cE_\C(f,f):=2\C(f)$
admitting a Carr\'e du champ $\Gamma$.
With the calculus tools developed in \cite{AGS11b}, it can be proved that
$\Gamma$ has a further equivalent representation
$\Gamma(f)=\llvert \rmD f\rrvert _w^2$ in terms
of the \emph{minimal weak
gradient} $\llvert \rmD f\rrvert _w$
of $f$. The latter is
the element of minimal $L^2$-norm
among all the possible weak limits of $\llvert \rmD f_n\rrvert $ in
the definition \eqref{eq103}.

It follows that $\cE_\C$ can also be expressed by
$\cE_\C(f,f)=\int_X \llvert \rmD f\rrvert _w^2\,\d\mm$
and the set of Lipschitz functions $f$ with $\int_X |\rmD
f|^2\,\d\mm<\infty$
is strongly dense in the domain of $\cE_\C$.
In fact,
the Dirichlet form $\cE_\C$ enjoys
a further upper-regularity property, common to every Cheeger energy
(\cite{AGS11c}, Section~8.3):
\begin{longlist}[({a})]
\item[(a)] for every $f\in D(\cE)$ there exist
$f_n\in D(\cE)\cap\rmC_b(X)$ and upper semicontinuous bounded
functions $g_n\dvtx X\to\R$
such that
\[
\Gamma(f_n)\le g_n^2,\qquad
f_n\to f \mbox{ in } L^2(X,\mm ) ,\qquad \limsup
_{n\to\infty}\int_X g_n^2
\,\d\mm\le\cE(f,f) . 
\]
\end{longlist}
Here and in the following, $\rmC_b(X)$ denotes the space of
continuous and bounded real functions defined on $X$.

\textit{From $\RCD K\infty$ to $\BE K\infty$.}
The previous properties of the Cheeger energy
show that
the investigation of Bakry--\'Emery curvature bounds
makes perfectly sense in
$\RCD K\infty$ spaces. One of the main results of \cite{AGS11b} connecting
these two approaches shows in fact that
$\RCD K\infty$ yields $\BE K\infty$ in the
gradient formulation \eqref{eq101} for every $f\in D(\cE_\C)$.

In fact, an even more refined result holds
(\cite{AGS11b}, Theorem~6.2), since
it is possible to control the slope of $\sP t f$
in terms of the minimal weak gradient of $f$
\[
\llvert \rmD\sP t f\rrvert ^2\le\rme^{-2K t} \sP
t \bigl(\llvert \rmD f\rrvert _w^2 \bigr)\qquad
\mbox{whenever }f\in D(\cE_\C), \llvert \rmD f\rrvert _w
\in L^\infty(X,\mm),
\]
an estimate that has two useful geometric-analytic consequences:
\begin{longlist}[({b})]
\item[(b)]
${\mathsf d}$ coincides with
the intrinsic distance associated to the Dirichlet form $\cE_\C$
(introduced in \textsc{Biroli--Mosco}
\cite{Biroli-Mosco95},
see also \cite{Sturm95,Sturm98} and \cite{Stollmann10}), namely
\[
{\mathsf d}(x,y)=\sup \bigl\{\psi(y)-\psi(x)\dvtx \psi\in D(
\cE_\C)\cap\rmC_b(X), \Gamma(\psi)\le 1 \bigr\},\qquad
x,y\in X.
\]
\end{longlist}
\begin{longlist}[({c})]
\item[(c)] Every function $\psi\in D(\cE_\C)$ with
$\Gamma(\psi)\le1$ $\mm$-a.e. admits a continuous
(in fact $1$-Lipschitz) representative
$\tilde\psi$.
\end{longlist}

\textit{From $\BE K\infty$ to $\RCD K\infty$.}
In the present paper, we provide necessary and sufficient conditions
for the validity of the
converse implication, that is, $\BE K\infty\Rightarrow\RCD
K\infty$.

In order to state this result in a precise way, one has first to
clarify how the metric structure should be related to the Dirichlet
one. Notice that this problem is much easier from the point of view
of the metric measure setting, since
one has the canonical way \eqref{eq103} to construct the Cheeger energy.

Since we tried to avoid any local compactness assumptions on $X$ as
well as
doubling or Poincar\'e conditions on $\mm$,
we used the previous structural properties
(a), (b), (c) as a {guide to find a} reasonable set of assumptions
for our theory; notice that they are in any case necessary conditions
to get a $\RCD K\infty$ space.

{We thus start from a strongly local and symmetric Dirichlet form $\cE$
on a Polish topological space $(X,\tau)$ endowed with its Borel
$\sigma$-algebra and a finite (for the scope of this introduction)
Borel measure $\mm$.
In the algebra $\cV_{\infty}:=D(\cE)\cap L^\infty(X,\mm)$
we consider the subspace $\cGz$ of functions $f$ admitting a Carr\'e du
champ $\Gamma(f)\in L^1(X,\mm)$: they are characterized by the identity
%
\begin{equation}
\label{eq24} \cE(f,f\varphi)-\frac{1}2 \cE\bigl(f^2,\varphi
\bigr)=\int_X \Gamma(f)\varphi\,\d\mm\qquad \mbox{for every
}\varphi\in\cV_{\infty}.
\end{equation}
We can therefore introduce
the intrinsic distance ${\mathsf d}_\cE$
as in ({b})
%
\begin{eqnarray}\label{eqBMdistance}
&&\qquad {\mathsf d}_\cE(x,y)\nonumber
\\[-8pt]
\\[-8pt]
&&\qquad \qquad :=\sup \bigl\{\psi(y)-\psi(x)\dvtx \psi\in \cGz
\cap\rmC(X), \Gamma(\psi)\le 1 \bigr\}, \qquad
 x,y\in X, \nonumber
\end{eqnarray}
and, following
the standard approach, we will assume that
${\mathsf d}_\cE$ is a complete distance on $X$ and
the topology induced by ${\mathsf d}_\cE$ coincides
with $\tau$.

In this way, we end up
with} \emph{Energy measure spaces} $(X,\tau,\mm,\cE)$
and in this setting
we prove in Theorem~\ref{thmslope-bound}
that $\cE\le\cE_\C$,
where $\cE_\C$ is the Cheeger energy associated to ${\mathsf d}_\cE$;
moreover, Theorem~\ref{thmcE=C} shows that
$\cE=\cE_\C$ if and only if ({a}) holds
{(see \cite{Koskela-Shanmugalingam-Zhou12}, Section~5, for
a similar result in the case of doubling spaces satisfying a local
Poincar\'e condition and for
interesting examples where $\cE_\C$ is not quadratic and
$\cE\neq\cE_\C$).}
{It is also worth mentioning (Theorem~\ref{corlength}) that
for this class of spaces $(X,{\mathsf d}_\cE)$ is always a length
metric space, a result previously known in a locally compact framework
\cite{Sturm98,Stollmann10}.}

{The Bakry--\'Emery condition $\BE K\infty$
can then be stated in a weak integral form
(strongly inspired by \cite{Bakry06,Bakry-Ledoux06,Wang11})
just involving the Markov semigroup $(\sP t)_{t\ge0}$
[see \eqref{eqBEKnu.3} of
Corollary~\ref{leBE-equivalent} and
\eqref{eqGG1.a}, \eqref{eqGG1.adelta} for relevant definitions]
by asking that the differential inequality
%
\begin{equation}
\quad\frac{\partial^2}{\partial s^2} \int_X (\sP{t-s} f)^2 \sP s
\varphi\,\d\mm\ge 2K \frac\partial{\partial s}\int_X (
\sP{t-s} f)^2 \sP s\varphi\,\d \mm,\qquad {0<s<t,}
\hspace*{-20pt}\label{eq23} 
\end{equation}
%
is fulfilled
for {any $f\in L^2(X,\mm)$ and} any nonnegative $\varphi\in
L^2\cap L^\infty(X,\mm)$.
Notice that in the case $K=0$ \eqref{eq23} is equivalent to the
convexity in $(0,t)$ of the map $s\mapsto\int_X (\sP{t-s} f)^2 \sP
s\varphi\,\d\mm$.}

If we also assume that $\BE K\infty$ holds, it
turns out that
({c}) is in fact equivalent to a weak-Feller
condition on the semigroup $(\sP t)_{t\ge0}$, namely $\sP t$ maps
$\Lip_b(X)$
in $\rmC_b(X)$. Moreover, ({c}) implies
the upper-regularity
(\emph{a}) of $\cE$ and the
fact that every $f\in D(\cE)\cap L^\infty(X,\mm)$
admits a Carr\'e du champ $\Gamma$ satisfying \eqref{eq24}.

Independently of $\BE K\infty$, when properties ({a}) and (c)
are satisfied,
we call $(X,\tau,\mm,\cE)$ a
\emph{Riemannian Energy measure space},
since these space seem appropriate nonsmooth versions of
Riemannian manifolds. It is also worth mentioning that in this class of spaces
$\BE K\infty$ is equivalent to an (exponential) contraction property
for the semigroup $(\sH{t})_{t\ge0}$ with respect to the Wasserstein
distance $W_2$ (see Corollary~\ref{corBE=contraction}), in analogy
with \cite{Kuwada10}.

Our main equivalence result, Theorem~\ref{thmmain-identification},
shows that a
$\BE K\infty$ Riemannian Energy measure space
satisfies the $\RCD K\infty$ condition: thus, in view of the converse
implication
proved in \cite{AGS11b}, $\BE K\infty$ is essentially equivalent to
$\RCD K\infty$.
A more precise formulation of our result, in the simplified case when
the measure $\mm$ is
finite, is the following.

%
\begin{theorem}[(Main result)]\label{thmmainintro}
Let $(X,\tau)$ be a Polish space and let $\mm$ be a finite Borel
measure in $X$. Let
$\cE\dvtx L^2(X,\mm)\to[0,\infty]$ be a strongly local, symmetric
Dirichlet form
generating a {mass preserving} Markov semigroup $(\sP t)_{t\ge0}$ in
$L^2(X,\mm)$, let ${\mathsf d}_\cE$ be
the intrinsic distance defined by
\eqref{eqBMdistance}
and assume that:
\begin{longlist}[(ii)]
\item[(i)] ${\mathsf d}_\cE$ {is a complete distance on $X$ inducing} the
topology $\tau$ and any function $f \in \mathbb G_\infty$ with $\Gamma(f)\le 1$ admits a
  continuous representative;

\item[(ii)]
the Bakry--\'Emery {$\BE K\infty$} condition \eqref{eq23} 
is fulfilled by $(\sP t)_{t\ge0}$.
\end{longlist}
Then $(X,{\mathsf d}_\cE,\mm)$ is a $\RCD K\infty$ space.
\end{theorem}

We believe that this equivalence result, between the ``Eulerian''
formalism of the
Bakry--\'Emery $\BE K\infty$ theory and the
``Lagrangian'' formalism of the $\CD K\infty$ theory, is conceptually
important and that it could be a first step for a
better understanding of curvature conditions in metric measure spaces.
Also, this equivalence
is technically useful. Indeed, in the last section of this paper, we
prove the tensorization of $\BE K N$ spaces. Then,
in the case $N=\infty$, we can use the implication from $\BE K\infty$
to $\RCD K\infty$ to read this property in terms of
tensorization of $\RCD K\infty$ spaces: this was previously known (see
\cite{AGS11b}) only under an
a priori nonbranching assumption on the base spaces [notice that the
$\CD K N$ theory, even with $N=\infty$,
suffers at this moment the same limitation].
On the other hand, we use the implication from $\RCD K\infty$ to $\BE
K\infty$, \eqref{eqBEKNPT} below
and the strong stability properties which follow by the EVI$_K$
formulation to provide stability of the $\BE K N$
condition under a very weak convergence, the Sturm--Gromov--Hausdorff
convergence.

\textit{Plan of the paper.}
Section~\ref{secMarkov} collects notation
and preliminary results on Dirichlet forms, Markov semigroups
and functional $\Gamma$-calculus, following the presentation
of \cite{Bouleau-Hirsch91}, which avoids any
topological assumption.
Particular attention is devoted to various formulations
of the $\BE KN$ condition: they are discussed
in Section~\ref{subsecBE},
trying to present
an intrinsic approach that does not rely on
the introduction of a distinguished algebra
of functions $\cA$ and extra assumptions on the Dirichlet
form $\cE$, besides locality. In its weak formulation [see \eqref
{eqBEKnu.3} of
Corollary~\ref{leBE-equivalent} and
\eqref{eqGG1.a}, \eqref{eqGG1.adelta}],
%
\begin{eqnarray}
\label{eqBEKNPT} %
{\frac{1}4}\frac{\partial^2}{\partial s^2} \int
_X (\sP {t-s} f)^2 \sP s\varphi\,\d\mm&\ge& {
\frac{K}2} \frac\partial{\partial s}\int_X (
\sP{t-s} f)^2 \sP s\varphi\,\d\mm
\nonumber
\\[-8pt]
\\[-8pt]
&&{}+ \frac{1}{N} \int_X (\DeltaE\sP{t-s}
f)^2 \sP s\varphi\,\d\mm, %
\nonumber
\end{eqnarray}
which is well suited to study stability issues,
$\BE KN$ does not even need a densely defined
Carr\'e du Champ $\Gamma$, because only the semigroup $(\sP t)_{t\ge0}$
is involved.

Section~\ref{secMeMe} is devoted to
study the interaction between energy and metric structures.
A few metric concepts are recalled in Section~\ref{subsecpreliminaries},
whereas Section~\ref{secdual} shows how to
construct a dual semigroup
{$(\sH{t})_{t\ge0}$} in the space of probability measures
$\prob X$ under suitable Lipschitz estimates on
$(\sP t)_{t\ge0}$. By using refined properties
of the Hopf--Lax semigroup, we also extend
some of the duality results proved by \textsc{Kuwada}
\cite{Kuwada10} to general complete and separable
metric measure spaces, avoiding any doubling or Poincar\'e condition.

Section~\ref{subsecEMMspace} presents
a careful analysis
{of the intrinsic distance ${\mathsf d}_\cE$ \eqref{eqBMdistance}
associated to a Dirichlet form and of
Energy measure structures $(X,\tau,\mm,\cE)$.}
We will thoroughly discuss
the relations between
the Dirichlet form $\cE$ and the Cheeger energy $\C$ induced
by a distance ${\mathsf d}$, possibly different from the intrinsic distance
${\mathsf d}_\cE$ and {we will obtain a precise characterization
of the distinguished case when ${\mathsf d}={\mathsf d}_\cE$ and
$\cE=2\C$:
here, conditions ({a}), (b) play a crucial role.}

A further investigation
when $\BE K\infty$ is also assumed is carried out in
Section~\ref{subsecDir-BE},
leading to the class of \emph{Riemannian Energy measure spaces}.

Section~\ref{secproof} contains
the proof of the main equivalence result, Theorem
{\ref{thmmainintro}}, between
$\BE K\infty$ and $\RCD K\infty$.
Apart the basic estimates of Section~\ref{subsecstep1},
the argument is split into two main steps: Section~\ref{subseclogHarnack}
proves a first L\,logL regularization estimate for
the semigroup {$(\sH{t})_{t\ge0}$}, starting from
arbitrary measures in $\probt X$
(here, we follow the approach of \cite{Wang11}).
Section~\ref{subsecaction} contains
the crucial action estimates to prove
the EVI$_K$ inequality \eqref{eqEVIK1}.
Even if the strategy of the proof has been partly inspired by
the geometric heuristics discussed in \cite{Daneri-Savare08}
(where the Eulerian approach of \cite{Otto-Westdickenberg05} to
contractivity of gradient flows
has been extended to cover also convexity and evolutions in the EVI$_K$ sense)
this part is completely new and it uses
in a subtle way all the refined technical issues
discussed in the previous sections of the paper.

In the last Section~\ref{secapplications}, we discuss the above
mentioned applications of the equivalence
between $\BE K\infty$ and $\RCD K\infty$.

\section{Dirichlet forms, Markov semigroups, \texorpdfstring{$\Gamma$}{$Gamma$}-calculus}\label{secMarkov}

\subsection{Dirichlet forms and \texorpdfstring{$\Gamma$}{$Gamma$}-calculus}
\label{subsecDiFo}

Let $(X,\cB)$ be a measurable space, let $\mm\dvtx \cB\to[0,\infty]$
$\sigma$-additive
and let $L^p(X,\mm)$ be the Lebesgue spaces (for notational
simplicity, we omit the dependence on $\cB$). Possibly
enlarging $\cB$ and extending $\mm$ we assume that $\cB$ is $\mm$-complete.
In the next Sections~\ref{secMeMe} and \ref{secproof}, we will
typically consider the case
when $\cB$ is the $\mm$-completion of the Borel $\sigma$-algebra
generated by a Polish topology $\tau$ on $X$.

In all of this paper, we will assume that
%
\begin{equation}
\label{eq50} %
\begin{tabular}{@{\quad\qquad}p{320pt}@{\hspace*{-5pt}}} $\cE\dvtx
L^2(X,\mm)\to[0,\infty]$ is a strongly local, symmetric Dirichlet
form generating a Markov semigroup $(\sP t)_{t\ge0}$ in
$L^2(X,\mm)$.\end{tabular} %
\end{equation}

Let us briefly recall the precise meaning of this statement.

A \textit{symmetric Dirichlet form} $\cE$
is a $L^2(X,\mm)$-lower semicontinuous quadratic form satisfying the
Markov property
%
\begin{equation}
\label{eqMarkovian} \cE(\eta\circ f)\leq\cE(f) \qquad\mbox{for every normal
contraction $\eta\dvtx \R\to\R$},
\end{equation}
that is, a $1$-Lipschitz map satisfying $\eta(0)=0$.
We refer to \cite{Bouleau-Hirsch91,Fukushima-Oshima-Takeda11}
for equivalent formulations of \eqref{eqMarkovian}.
We also define
\[
\cV:=D(\cE)=\bigl\{f\in L^2(X,\mm)\dvtx \cE(f)<\infty
\bigr\}, \qquad \cV_{\infty}:=D(\cE)\cap L^\infty(X,\mm).
\]
We also assume that $\cV$
is dense in $L^2(X,\mm)$.

We still denote by $\cE(\cdot,\cdot)\dvtx \cV\to\R$ the
associated continuous and symmetric bilinear form
\[
\cE(f,g):=\tfrac{1}4 \bigl(\cE(f+g)-\cE(f-g) \bigr).
\]
We will assume \textit{strong locality of $\cE$}, namely
\[
\forall f,g\in\cV\dvtx \cE(f,g)=0\qquad\mbox{if $(f+a)g=0$ $
\mm$-a.e. in $X$ for some $a\in\R$.}
\]
It is possible to prove (see, e.g., \cite{Bouleau-Hirsch91},
Proposition~2.3.2) that $\cV_{\infty}$
is an algebra with
respect to pointwise
multiplication, so that
for every $f\in\cV_{\infty}$ the linear form on $\cV
_{\infty}$
%
\begin{equation}
\label{eq65bis} \boldsymbol{\Gamma}[f;\varphi]:=\due\cE(f,f\varphi)-\uno\cE
\bigl(f^2,\varphi\bigr),\qquad \varphi\in\cV_{\infty},
\end{equation}
is well defined and for every normal contraction $\eta\dvtx \R\to\R$ it
satisfies
\cite{Bouleau-Hirsch91}, Proposition~2.3.3,
%
\begin{equation}
\label{eq106} 0\le\boldsymbol{\Gamma}[\eta\circ f;\varphi]\le \boldsymbol{
\Gamma}[f;\varphi]\le\llVert \varphi\rrVert _\infty\cE(f)\qquad \mbox{for
every }f,\varphi\in\cV_{\infty}, \varphi\ge0.\hspace*{-27pt}
\end{equation}
Equation \eqref{eq106} shows that for every nonnegative $\varphi\in
\cV_{\infty}$ $f\mapsto\boldsymbol{\Gamma}[f;\varphi]$ is a
quadratic form in $\cV_{\infty}$ which satisfies the Markov property
and can be extended by continuity to
$\cV$.
We call $\cG$ the set of functions $f\in\cV$ such that
the linear form
$\varphi\mapsto\boldsymbol{\Gamma}[f;\varphi]$
can be represented by a an absolutely continuous measure w.r.t. $\mm$
with density $\Gamma(f)\in L^1_+(X,\mm)$:
%
\begin{equation}
\label{defenergymeasure} %
f\in\cG\quad \Leftrightarrow\quad  \boldsymbol{\Gamma}[f;\varphi]=
\int_X \Gamma(f)\varphi\,\d\mm \qquad\mbox{for every }
\varphi\in\cV_{\infty}. %
\end{equation}
{Since $\cE$ is strongly local, \cite{Bouleau-Hirsch91}, Theorem~6.1.1, yields
the representation formula}
%
\begin{equation}
\label{eqenergymeasurebis} \cE(f,f)=\int_X \Gamma(f)\,\d\mm\qquad
\mbox{for every }f\in\cG.
\end{equation}
It is not difficult to check that
$\cG$ is a closed vector subspace of $\cV$,
the restriction of $\cE$ to $\cG$ is still a strongly local Dirichlet
form
admitting {the}
{\emph{Carr\'e du champ}}
$\Gamma$ {defined by \eqref{defenergymeasure}}
(see, e.g., \cite{Bouleau-Hirsch91}, Definition~4.1.2):
$\Gamma$ is a quadratic
continuous map defined in $\cG$ with
values in $L^1_+(X,\mm)$.
We will see in the next Section~\ref{subsecBE} that
if $\cE$ satisfies the $\BE K\infty$ condition, then
$\cG$ coincides with $\cV$
and $\cE$ admits a functional $\Gamma$-calculus
on the whole space $\cV$.

Since we are going to use $\Gamma$-calculus techniques,
we use the $\Gamma$ notation also for the symmetric, bilinear and
continuous map
\[
\Gamma(f,g):=\tfrac{1}4 \bigl(\Gamma(f+g)-\Gamma(f-g) \bigr)
\in L^1(X,\mm ),\qquad f, g\in\cG,
\]
which, thanks to \eqref{eqenergymeasurebis}, represents the bilinear
form $\cE$ by the formula
\[
\cE(f,g)= \mezzo\int_X\Gamma(f,g) \,\d\mm
\qquad\mbox{for every }f,g\in\cG.
\]
Because of Markovianity and locality
$\Gamma(\cdot,\cdot)$ satisfies the chain rule
\cite{Bouleau-Hirsch91}, Corollary~7.1.2,
%
\begin{equation}
\label{eqGchain} \Gamma\bigl(\eta(f),g\bigr)=\eta'(f)\Gamma(f,g)
\qquad \mbox{for every }f, g\in\cG, \eta\in\Lip(\R), \eta(0)=0,\hspace*{-25pt}
\end{equation}
and the Leibniz rule:
\[
\Gamma(fg,h)=f\Gamma(g,h)+g\Gamma(f,h)\qquad \mbox{for every }f, g,
h\in\cGz:=\cG\cap L^\infty(X,\mm).
\]
Notice that by
\cite{Bouleau-Hirsch91}, Theorem~7.1.1,
\eqref{eqGchain} is well defined since for every
{Borel} set $N\subset\R$
(as the set where $\phi$ is not differentiable)
%
\begin{equation}
\label{eq72} \Leb1(N)=0\quad\Rightarrow\quad \Gamma(f)=0\qquad\mbox{$\mm$-a.e.
on }f^{-1}(N).
\end{equation}
Among the most useful consequences of \eqref{eq72} and
\eqref{eqGchain} that we will repeatedly use in the
sequel, we recall that for every $f,g\in\cG$
\[
\Gamma(f-g)=0\qquad\mbox{$\mm$-a.e. on }\{f=g\},
\]
and the following identities hold $\mm$-a.e.:
%
\begin{eqnarray}
\label{eq74} \Gamma(f\land g)&=& %
\cases{\displaystyle \Gamma(f),&
\quad on $\{f\le g\}$,
\cr
\displaystyle \Gamma(g),&\quad on $\{f\ge g\}$, }
\nonumber
\\[-8pt]
\\[-8pt]
\Gamma(f\lor g)&=& %
\cases{\displaystyle \Gamma(f),&\quad on $\{f\ge
g\}$,
\cr
\displaystyle \Gamma(g),&\quad on $\{f\le g\}$. } %
\nonumber
\end{eqnarray}
We conclude this section by stating the following lower semicontinuity
result along a sequence $(f_n)_n\subset\cG$ converging to $f\in\cG$:
%
\begin{eqnarray}
\label{eq20} 
&&f_n\weakto f,\qquad  \sqrt{
\Gamma(f_n)}\weakto G \qquad \mbox{in }L^2(X,\mm)\nonumber
\\[-8pt]
\\[-8pt]
&&\qquad \Rightarrow\quad \Gamma(f)\le G^2\qquad  \mbox{$\mm$-a.e. in $X$}.
\nonumber
\end{eqnarray}
It can be easily proved by using Mazur's lemma and the $\mm$-a.e.
convexity of
$f\mapsto\sqrt{\Gamma(f)}$,
namely
\[
\sqrt{\Gamma\bigl((1-t)f+tg\bigr)}\leq(1-t)\sqrt{\Gamma(f)}+t\sqrt{\Gamma (g)}
\qquad \mbox{$\mm$-a.e. in $X$, for all $t\in[0,1]$,}
\]
which follows since $\Gamma$ is quadratic and nonnegative.

\textit{The Markov semigroup and its generator.}
The Dirichlet form $\cE$ induces a densely defined self-adjoint operator
$\DeltaE\dvtx D(\DeltaE)\subset\cV\to L^2(X,\mm)$
defined by the integration by parts formula
$\cE(f,g)=-\int_X g \DeltaE f\,\d\mm$ for all $g\in\cV$.

{When $\cG=\cV$,} the operator $\DeltaE$ is of ``diffusion'' type,
since it satisfies the following chain rule
for every $\eta\in\rmC^2(\R)$ with $\eta(0)=0$ and bounded first and
second derivatives
{(see \cite{Bouleau-Hirsch91}, Corollary~6.1.4, and the next
\eqref{eqgoodextension})}: if $f\in D(\DeltaE)$
{with $\Gamma(f)\in L^2(X,\mm)$} then $\eta(f)\in D(\DeltaE)$ with
%
\begin{equation}
\label{eqGchain-laplace} \DeltaE\eta(f)=\eta'(f)\DeltaE f+
\eta''(f)\Gamma(f).
\end{equation}
\textit{The heat flow $\sP t$ associated to $\cE$} is well defined starting
from any initial condition $f\in L^2(X,\mm)$. Recall that in this framework
the heat flow $(\sP t)_{t\ge0}$ is an analytic Markov semigroup
and $f_t=\sP t f$ can be characterized as the unique $C^1$ map
$f\dvtx (0,\infty)\to L^2(X,\mm)$, with values in $D(\DeltaE)$, satisfying
\[
\cases{\displaystyle \frac{\d}{\d t} f_t=
\DeltaE f_t,&\quad for $t\in (0,\infty)$,
\cr
\displaystyle \lim
_{t\downarrow0}f_t=f,&\quad in $L^2(X,\mm)$. }
\]
{Because of this, $\DeltaE$ can equivalently be characterized in terms of
the strong convergence $(\sP t f-f)/t\to\DeltaE f$ in $L^2(X,\mm)$ as
$t\downarrow0$. }

One useful consequence of the Markov property is the $L^p$ contraction
of $(\sP t)_{t\ge0}$ from $L^p\cap L^2$ to $L^p\cap L^2$.
Because of the density of $L^p\cap L^2$ in $L^p$\vadjust{\goodbreak} when $p\in[1,\infty
)$, this allows to
extend uniquely $\sP t$ to
a strongly continuous semigroup of linear contractions in
$L^p(X,\mm)$, $p\in[1,\infty)$, for
which we retain the same notation.
Furthermore, $(\sP t)_{t\ge0}$ is sub-Markovian (cf. \cite
{Bouleau-Hirsch91}, Proposition~3.2.1), since it preserves
one-sided essential bounds, namely $f\leq C$ (resp., $f\geq C$)
$\mm$-a.e. in $X$ for some $C\geq0$ (resp., $C\leq0$) implies $\sP t
f\leq C$ (resp., $\sP tf\geq C$) $\mm$-a.e. in $X$ for all
$t\geq0$.

We will mainly be concerned with the mass-preserving case,
that is,
%
\begin{equation}
\label{eq4} \int_X \sP t f\,\d\mm=\int
_X f\,\d\mm\qquad\mbox{for every }f\in L^1(X,
\mm),
\end{equation}
a property which is equivalent to $1\in D(\cE)$
when
$\mm(X)<\infty$. In the next session
(see Theorem~\ref{thmcE=C}), we will discuss a metric
framework, which will imply \eqref{eq4}.

The semigroup
$(\sP t)_{t\ge0}$ can also be extended by duality
to a weakly$^*$-continuous semigroup of
contractions in $L^\infty(X,\mm)$, so that
\[
\int_X \sP t f \varphi\,\d\mm=\int
_X f \sP t\varphi\,\d\mm \qquad \mbox{for every }f\in
L^\infty(X,\mm), \varphi\in L^1(X,\mm).
\]
It is easy to show that if $f_n\in
L^\infty(X,\mm)$
weakly$^*$
converge to $f$ in $L^\infty(X,\mm)$
then $\sP t f_n\weaksto\sP t f$ in $L^\infty(X,\mm)$.

\textit{The generator of the semigroup in $L^1(X,\mm)$.}
Sometimes it will also be useful to consider the generator
${\DeltaE^{(1)}}\dvtx D({\DeltaE^{(1)}})\subset L^1(X,\mm)\to L^1(X,\mm
)$ of $(\sP
t)_{t\ge0}$ in
$L^1(X,\mm)$ (\cite{Pazy83}, Section~1.1):
%
\begin{eqnarray}
\label{eq70} &&f\in D\bigl({\DeltaE^{(1)}}\bigr),\qquad  {
\DeltaE^{(1)}}f=g\nonumber
\\[-8pt]
\\[-8pt]
&&\qquad \Leftrightarrow \quad \lim_{t\downarrow0}
\frac{1}t (\sP t f-f )=g
 \qquad \mbox{strongly in }L^1(X,\mm). \nonumber
\end{eqnarray}
Thanks to \eqref{eq70} {and $L^1$ contractivity} it is easy to check that
%
\begin{eqnarray}
\label{eq7} &&\quad  f\in D\bigl({\DeltaE^{(1)}}\bigr)
\nonumber
\\[-8pt]
\\[-8pt]&&\quad  \qquad \Rightarrow\quad
\sP t f\in D\bigl({\DeltaE^{(1)}}\bigr),\qquad {\DeltaE^{(1)}}
\sP t f=\sP t{\DeltaE^{(1)}}f
\qquad \mbox{for all $t\geq0$},\nonumber
\end{eqnarray}
and, when \eqref{eq4} holds,
\[
\int_X {\DeltaE^{(1)}}f\,\d\mm=0
\qquad\mbox{for every }f\in D\bigl({\DeltaE^{(1)}}\bigr).
\]
The operator ${\DeltaE^{(1)}}$ is $m$-accretive and coincides with
the smallest closed extension of $\DeltaE$ to $L^1(X,\mm)$:
(\cite{Bouleau-Hirsch91}, Proposition~2.4.2):
%
\begin{eqnarray}
\label{eq2} &&g={\DeltaE^{(1)}}f\nonumber
\\[-4pt]
\\[-12pt]
&&\qquad\Leftrightarrow\quad \cases{
\displaystyle\exists f_n\in D(\DeltaE)\cap
L^1(X,\mm) \mbox{ with $g_n=\DeltaE f_n\in
L^1(X,\mm)\dvtx $}
\cr
   \mbox{$f_n\to f, g_n
\to g$ strongly in $L^1(X,\mm)$.}}
\nonumber
\end{eqnarray}
Whenever $f\in D({\DeltaE^{(1)}})\cap L^2(X,\mm)$ and ${\DeltaE
^{(1)}}f\in
L^2(X,\mm)$,
one can recover $f\in D(\DeltaE)$ by
\eqref{eq7}, the integral formula $\sP t f-f=\int_0^t \sP r{\DeltaE^{(1)}}
f\,\d r$ and
the contraction property of $(\sP t)_{t\ge0}$ in every $L^p(X,\mm)$,
thus obtaining
%
\begin{eqnarray}
\label{eqgoodextension} &&f\in D\bigl({\DeltaE^{(1)}}\bigr)\cap
L^2(X,\mm),\qquad  {\DeltaE^{(1)}}f\in L^2(X,\mm)
\nonumber
\\[-8pt]
\\[-8pt]
&&\qquad\Rightarrow\quad f\in D(\DeltaE), \qquad \DeltaE f={\DeltaE^{(1)}}f.
\nonumber
\end{eqnarray}

\textit{Semigroup mollification.}
A useful tool to prove the above formula is given by the
mollified semigroup: we fix a
%
\begin{equation}
\label{eq65} \mbox{nonnegative kernel $\kappa\in\rmC^\infty_c(0,
\infty)$ with $\int_0^\infty\kappa(r)\,\d r=1$, }
\end{equation}
and for every $f\in L^p(X,\mm)$, $p\in[1,\infty]$, we set
%
\begin{equation}
\label{eqGG4} \frh^{\eps}f:= \frac{1}\eps\int
_0^\infty\sP r f \kappa(r/\eps)\,\d r,\qquad \eps>0,
\end{equation}
where the integral should be intended in the Bochner sense whenever
$p<\infty$ and
by taking the duality with arbitrary $\varphi\in L^1(X,\mm)$ when
$p=\infty$.

Since $\DeltaE$ is the generator of $(\sP t)_{t\ge0}$ in $L^2(X,\mm)$,
it is not difficult to check (\cite{Pazy83}, Proof of Theorem~2.7),
that if $f\in L^2\cap L^p(X,\mm)$ for some $p\in[1,\infty]$ then
\[
{-} \DeltaE\bigl(\frh^{\eps} f\bigr) =\frac{1}{\eps^2}
\int_0^\infty\sP r f \kappa'(r/\eps)\,
\d r \in L^2\cap L^p(X,\mm).
\]
Since ${\DeltaE^{(1)}}$ is the generator of $(\sP t)_{t\ge0}$
in $L^1(X,\mm)$,
the same property holds for ${\DeltaE^{(1)}}$ if $f\in L^1(X,\mm)$:
%
\begin{equation}
\label{eqbanff1} {-} {\DeltaE^{(1)}}\bigl( \frh^{\eps}f\bigr) =
\frac{1}{\eps^2} \int_0^\infty\sP r f
\kappa'(r/\eps)\,\d r \in L^1(X,\mm).
\end{equation}
%

\subsection{On the functional Bakry--\'Emery condition}
\label{subsecBE}

We will collect in this section various equivalent
characterizations of the Bakry--\'Emery condition $\BE KN$
given in \eqref{eq99} for the $\Gamma_2$
operator \eqref{eq97}.
We have been strongly inspired by
\cite{Bakry06,Bakry-Ledoux06,Wang11}:
even if the essential
estimates are well known, here we will take a particular care in
establishing all the results in a weak form,
under the minimal regularity assumptions
on the functions involved.
We consider here the case of finite dimension as well, despite the fact
that the next Sections~\ref{secMeMe} and \ref{secproof}
will be essentially confined
to the case $N=\infty$.
Applications of $\BE KN$ with $N<\infty$
will be considered in the last Section~\ref{secapplications}.

Let us denote by $ \boldsymbol{\Gamma}\dvtx (\cV_{\infty})^3\to\R$
the multilinear map
\[
\boldsymbol{\Gamma}[f,g;\varphi]:=\tfrac12 \bigl(\cE(f,g\varphi)+\cE
(g,f\varphi)-\cE(fg,\varphi) \bigr), \qquad \boldsymbol{\Gamma}[f;\varphi]=
\boldsymbol{\Gamma}[f,f;\varphi].
\]
Recalling \eqref{eq106}, one can easily prove the {uniform}
continuity property
%
\begin{eqnarray}
\label{eq110} &&f_n,\varphi_n\in \mathbb V_\infty, \qquad f_n\to f,\qquad
\varphi_n\to\varphi\qquad \mbox{in }\cV,\qquad  \sup_n
\llVert \varphi_n\rrVert _\infty<\infty
\nonumber
\\[-8pt]
\\[-8pt]
&&\qquad \Rightarrow\quad \exists\lim_{n\to\infty} \boldsymbol{
\Gamma}[f_n;\varphi_n]\in\R, 
\nonumber
\end{eqnarray}
which allows to extend $ \boldsymbol{\Gamma}$ to a real multilinear map
defined in $\cV\times\cV\times\cV_{\infty}$, for which
we retain the same notation.
The extension $ \boldsymbol{\Gamma}$ satisfies
%
\begin{equation}
\label{eq108} \boldsymbol{\Gamma}[f,g;\varphi]= \int_X
\Gamma(f,g)\varphi\,\d\mm\qquad\mbox{if }f,g\in\cG.
\end{equation}
We also set
%
\begin{equation}
\label{eqGG1-Gamma2} \boldsymbol{\Gamma}_2[f;\varphi]:= \tfrac{1}2
\boldsymbol{\Gamma}[f;\DeltaE\varphi]- \boldsymbol{\Gamma}[f,\DeltaE f;\varphi],
\qquad(f,\varphi)\in D( \boldsymbol{\Gamma}_2),
\end{equation}
where
\[
D( \boldsymbol{\Gamma}_2):= \bigl\{(f,\varphi)\in D(
\DeltaE)\times D(\DeltaE)\dvtx \DeltaE f\in\cV, \varphi, \DeltaE\varphi\in
L^\infty(X,\mm) \bigr\}.
\]
As for \eqref{eq108}, we have
\begin{eqnarray}
\boldsymbol{\Gamma}_2[f;\varphi]= \int
_X \biggl(\frac{1}2\Gamma(f) \DeltaE\varphi-
\Gamma(f,\DeltaE f)\varphi \biggr)\,\d\mm
\nonumber
\\
\eqntext{\mbox{if }(f,\varphi)\in D( \boldsymbol{\Gamma}_2), f,
\DeltaE f\in\cG.}
\end{eqnarray}
Since $(\sP t)_{t\ge0}$ is an analytic semigroup in $L^2(X,\mm)$,
for a given $f\in L^2(X,\mm)$ and $\varphi\in L^2\cap L^\infty(X,\mm)$,
we can consider the functions
%
\begin{eqnarray}
\label{eqGG1.a}
\sfA_{t}[f;\varphi](s) &:=&\frac{1}2\int
_X ( \sP{t-s}f )^2 \sP s \varphi\, \d\mm,
\qquad t>0, s\in[0,t],
\\
\label{eqGG1.adelta}
\sfA^{\Delta}_{t}[f;\varphi](s) &:=&
\frac{1}2\int_X (\DeltaE\sP {t-s}f )^2
\sP s \varphi\,\d\mm, \qquad t>0, s\in[0,t),\\
\sfB_{t}[f;\varphi](s)&:=& \boldsymbol{\Gamma}[
\sP{t-s}f; \sP s\varphi] 
,\qquad t>0, s\in(0,t),
\nonumber
\end{eqnarray}
and, whenever $\DeltaE\varphi\in L^2\cap L^\infty(X,\mm)$,
\begin{eqnarray*}
\sfC_{t}[f;\varphi](s) := \boldsymbol{\Gamma}_2[
\sP{t-s}f;\sP s \varphi], \qquad t>0, s\in[0,t).
\end{eqnarray*}
Notice that whenever $\DeltaE f\in L^2(X,\mm)$
\begin{eqnarray*}
\sfA^{\Delta}_{t}[f;\varphi](s)&=&
\sfA_{t}[\DeltaE f;\varphi](s) %
,\qquad t>0, s\in[0,t).
\end{eqnarray*}

%
%
\begin{lemma}
\label{leB-regularity}
For every $f\in L^2(X,\mm),\varphi\in L^2\cap L^\infty(X,\mm)$
and every $t>0$,
we have:
\begin{longlist}[(iii)]
\item[(i)]
the function $s\mapsto\sfA_{t}[f;\varphi](s)$ belongs to $\rmC
^0([0,t])\cap
\rmC^1((0,t))$;
\item[(ii)] the function $s\mapsto\sfA^{\Delta}_{t}[f;\varphi](s)$
belongs to $\rmC^0([0,t))$;
\item[(iii)] the function $s\mapsto\sfB_{t}[f;\varphi](s)$ belongs to
$\rmC^0((0,t))$ and
%
\begin{equation}
\label{eqGG2} \frac\partial{\partial s}\sfA_{t}[f;\varphi](s)=
\sfB_{t}[f;\varphi ](s)\qquad\mbox {for every }s\in(0,t).
\end{equation}
Equation \eqref{eqGG2} and the regularity of $\sfA$ and $\sfB$
extend to $s=t$ if
$f\in\cV$ and to $s=0$ if $\varphi\in\cV_\infty$.
\item[(iv)] If $\varphi$ is nonnegative, $s\mapsto\sfA_{t}[f;\varphi](s)$
and $s\mapsto\sfA^{\Delta}_{t}[f;\varphi](s)$ are nondecreasing.
\item[(v)]
If $\DeltaE\varphi\in L^2\cap L^\infty(X,\mm)$ then
$\sfC$ belongs to $\rmC^0([0,t))$, $\sfB$ belongs to $\rmC
^1([0,t))$, and
%
\begin{equation}
\label{eqGG3} \frac\partial{\partial s}\sfB_{t}[f;\varphi](s)= 2
\sfC_{t}[f;\varphi](s)\qquad\mbox{for every }s\in[0,t).
\end{equation}
In particular, $\sfA\in\rmC^2([0,t))$.
\end{longlist}
\end{lemma}

\begin{pf}
The continuity
of $\sfA$ is easy to check, since $s\mapsto( \sP{t-s}f)^2$ is
strongly continuous
with values in $L^1(X,\mm)$ and $s\mapsto\sP s \varphi$ is
weakly$^*$ continuous in $L^\infty(X,\mm)$.
Analogously, the continuity of $\sfB$ follows from the fact that
$s\mapsto\sP{t-s}f$ is a continuous curve in $\cV$ whenever
$s\in[0,t)$ thanks to the regularizing effect of the heat
flow and
\eqref{eq110}.
The continuity of $\sfC$ follows by a similar argument,
recalling the definition \eqref{eqGG1-Gamma2} and the fact that the curves
$s\mapsto\DeltaE\sP{t-s} f$
and $s\mapsto\DeltaE\sP s\varphi$
are continuous with values in
$\cV$
in the interval
$[0,t)$.

In order to prove \eqref{eqGG2} and \eqref{eqGG3},
let us first assume that $\varphi\in D(\DeltaE)$ with $\DeltaE
\varphi\in L^\infty(X,\mm)$
and $f\in L^2\cap L^\infty(X,\mm)$. Since
\begin{eqnarray*}
\lim_{h\to0}\frac{ \sP{t-(s+h)}f- \sP{t-s}f}h&=& -\DeltaE\sP{t-s}f\qquad
\mbox{strongly in $ \cV$ for $s\in[0,t)$,}
\\
\lim_{h\to0}\frac{ \sP{s+h}\varphi- \sP{s}
\varphi}h&=& \DeltaE\sP{s}\varphi\qquad
\mbox{weakly$^*$ in $ L^\infty(X,\mm)$ for $s\in[0,t)$,}
\end{eqnarray*}
we easily get
\begin{eqnarray*}
\frac\partial{\partial s}\sfA_{t}[f;\varphi](s)&=& \int
_X \biggl(- \sP{t-s}f \DeltaE\sP{t-s}f \sP s \varphi+
\frac{1}2 ( \sP{t-s} f )^2\DeltaE\sP s\varphi \biggr)\,\d\mm
\\
&=&\cE( \sP{t-s}f, \sP{t-s} f \sP s\varphi)- \frac{1}2 \cE\bigl((
\sP{t-s}f)^2, \sP s\varphi\bigr) =\sfB_{t}[f;\varphi](s),
\end{eqnarray*}
by the very definition \eqref{eq65bis} of $ \boldsymbol{\Gamma}$,
since $ \sP{t-s}f$ is
essentially bounded and therefore $( \sP{t-s} f)^2\in\cV_\infty$.
A similar computation yields \eqref{eqGG3}.

In order to extend the validity of \eqref{eqGG2} and
\eqref{eqGG3} to general $f\in L^2(X,\mm)$, we
approximate $f$ by truncation, setting $f_n:=-n\lor f\land n$, $n\in\N$,
and we pass to the limit
in the integrated form
\[
\sfA_{t}[f_n;\varphi](s_2)-
\sfA_{t}[f_n;\varphi](s_1)= \int
_{s_1}^{s_2}\sfB_{t}[f_n;
\varphi](s)\,\d s\qquad \mbox{for every }0\le s_1<s_2<t,
\]
observing that $ \sP{t-s}f_n$ converge strongly to $ \sP{t-s}f$
in $\cV$ as $n\to\infty$ for every $s\in[0,t)$,
so that \eqref{eq110} yields the pointwise convergence
of the integrands in the previous identity.
A similar argument holds for \eqref{eqGG3}, since
$\DeltaE\sP t f_n$ converges strongly to $\DeltaE\sP t f$ in $\cV$.

Eventually, we extend \eqref{eqGG2} to arbitrary $\varphi\in L^2\cap
L^\infty(X,\mm)$ by approximating $\varphi$ with
$\frh^\eps\varphi$ given by \eqref{eqGG4}, \eqref{eq65}.
It is not difficult to check that
$ \sP s (\frh^\eps\varphi) \to\sP s\varphi$
in $\cV$ as $\eps\down0$ with uniform $L^\infty$ bound if
$s>0$ (and also when $s=0$ if $\varphi\in\cV_\infty$).
\end{pf}

%
%
\begin{lemma}
Let us consider functions ${\mathsf a}\in\rmC^1([0,t))$, $\sfg\in
\rmC^0([0,t))$
and a parameter $\nu\ge0$.
The following properties are equivalent:
\begin{longlist}[(iii)]
\item[(i)]${\mathsf a},\sfg$ satisfy the differential inequality
\[
{\mathsf a}''\ge2K {\mathsf
a}' +\nu \sfg \qquad\mbox{in }\mathscr D'(0,t),
\]
and pointwise in $[0,t)$, whenever ${\mathsf a}\in\rmC^2([0,t))$.
\item[(ii)]${\mathsf a}',\sfg$ satisfy the differential inequality
\[
\frac\d{\d s} \bigl(\rme^{-2K s}{\mathsf a}'(s)
\bigr)\ge\nu\rme^{-2K
s}\sfg(s) \qquad\mbox{in }\mathscr
D'(0,t).
\]
\item[(iii)] For every $0\le s_1<s_2<t$
and every test function $\zeta\in\rmC^2([s_1,s_2])$, we have
%
\begin{equation}
\label{eqGG7} \int_{s_1}^{s_2} {\mathsf a}\bigl(
\zeta''+2K\zeta'\bigr)\,\d s+ \bigl[{
\mathsf a}' \zeta \bigr]_{s_1}^{s_2}- \bigl[{
\mathsf a} \bigl(\zeta'+2K\zeta\bigr) \bigr]_{s_1}^{s_2}
\ge \nu\int_{s_1}^{s_2} \sfg\zeta\,\d s.
\end{equation}
\item[(iv)] For every $0\le s_1<s_2$, we have
%
\begin{equation}
\label{eqGG6bis} \rme^{-2K(s_2-s_1)}{\mathsf a}'(s_2)
\ge{\mathsf a}'(s_1)+\nu\int_{s_1}^{s_2}
\rme^{-2K(s-s_1)}\sfg(s)\,\d s.
\end{equation}
\end{longlist}
\end{lemma}

The \emph{proof} is straightforward; we only notice that \eqref{eqGG7}
holds also for $s_1=0$ since ${\mathsf a}\in\rmC^1([0,t))$.

The inequality
\eqref{eqGG7} has two useful consequences that
we make explicit in terms of the functions $\rmI_K$ and $\rmI_{K,2}$
defined by
%
\begin{eqnarray}
\label{eq57} \rmI_{K}(t)&=&\int_0^t
\rme^{Ks}\,\d s =\frac{\rme^{K t}-1}{K},
\nonumber
\\[-8pt]
\\[-8pt]
\rmI_{K,2}(t)&=&\int_0^t
\rmI_{K}(s)\,\d s= \frac{\rme^{Kt}-Kt-1}{K^2},
\nonumber
\end{eqnarray}
with the obvious definition for $K=0$: $I_0(t)=t, I_{0,2}(t)=t^2/2$.

Choosing $s_1=0, s_2=\tau$
and
\[
\zeta(s):=\rmI_{2K}(\tau-s)=\frac{\rme^{2K(\tau-s)}-1}{2K} \qquad
\mbox{so that } \zeta'+2K\zeta=-1, \zeta(\tau)=0,
\]
we obtain
%
\begin{eqnarray}
\label{eqfromzeta1} \rmI_{2K}(\tau) 
{\mathsf
a}'(0)+
\nu\int_0^\tau
\rmI_{2K}(\tau-s) 
\sfg(s)\,\d s\le{\mathsf a}(\tau)-{\mathsf
a}(0)
\nonumber
\\[-8pt]
\\[-8pt]
\eqntext{\mbox{for every }\tau\in[0,t).}
\end{eqnarray}
Choosing
\[
\zeta(s):=\rmI_{-2K}(s)=\frac{1-\rme^{-2Ks}}{2K} \qquad \mbox{so
that } \zeta'+2K\zeta=1, \zeta(0)=0,
\]
we obtain
%
\begin{eqnarray}
\label{eqfromzeta2} {\mathsf a}(\tau)-{\mathsf a}(0)+ \nu\int_0^\tau
\rmI_{-2K}(s) \sfg(s)\,\d s\le{\mathsf a}'(\tau)
\rmI_{-2K}(\tau)
\nonumber
\\[-8pt]
\\[-8pt]
\eqntext{\mbox{for every }\tau\in[0,t).}
\end{eqnarray}
%

%
\begin{corollary}
\label{leBE-equivalent}
Let
$\cE$ be a Dirichlet form in $L^2(X,\mm)$ as in
\eqref{eq50}, and let $K\in\R$ and $\nu\ge0$.
The following conditions are equivalent:
\begin{longlist}[(iii)]
\item[(i)]
For every $(f,\varphi)\in D( \boldsymbol{\Gamma}_2)$,
with $\varphi\ge0$,
we have
\[
\boldsymbol{\Gamma}_2[f;\varphi]\ge K \boldsymbol{
\Gamma}[f;\varphi]+ \nu\int_X (\DeltaE f)^2
\varphi\,\d\mm.
\]
\item[(ii)] For every $f\in L^2(X,\mm)$
and every nonnegative $\varphi\in D(\DeltaE)\cap L^\infty(X,\mm)$
with $\DeltaE\varphi\in L^\infty(X,\mm)$, we have
%
\begin{equation}
\label{eqBEKnu.2} \sfC_{t}[f;\varphi](s)\ge K\sfB_{t}[f;
\varphi](s)+2\nu \sfA ^{\Delta}_{t}[f;\varphi](s) \qquad\mbox{for
every }0\le s<t.\hspace*{-25pt}
\end{equation}
\item[(iii)] For every $f\in L^2(X,\mm)$, every nonnegative
$\varphi\in L^2\cap L^\infty(X,\mm)$
and $t>0$
%
\begin{equation}
\label{eqBEKnu.3} \frac{\partial^2}{\partial s^2} \sfA_{t}[f;\varphi](s)\ge 2K \frac
\partial{\partial s}\sfA_{t}[f;\varphi](s) + 4\nu\sfA
^{\Delta}_{t}[f;\varphi](s) 
\end{equation}
in the sense of distribution of $\mathscr D'(0,t)$
(or, equivalently, the inequality \eqref{eqBEKnu.3}
holds pointwise in $[0,t)$ for every nonnegative $\varphi\in L^2\cap
L^\infty(X,\mm)$ with
$\DeltaE\varphi\in L^2\cap L^\infty(X,\mm)$).
\item[(iv)] For every $f\in L^2(X,\mm)$ and $t>0$, we have
$\sP t f\in\cG$ and
%
\begin{eqnarray}
\label{eqBEKnu.4} \rmI_{2K}(t)\Gamma(\sP t f)+ 2\nu \mathrm
I_{2K,2}(t) (\DeltaE\sP t f )^2\le \tfrac12\sP t
\bigl(f^2 \bigr) -\tfrac12(\sP t f)^2 \nonumber
\\[-8pt]
\\[-8pt]
\eqntext{\mbox{$
\mm$-a.e. in $X$}.}
\end{eqnarray}
\item[(v)]$\cG=\cV$ and for every $f\in\cV$
\[
\tfrac12\sP t \bigl(f^2 \bigr)-\tfrac12(\sP t
f)^2+ 2\nu \mathrm I_{-2K,2}(t) (\DeltaE\sP t f
)^2\le \rmI_{-2K,2}(t) \sP t \Gamma(f) \qquad\mbox{$\mm$-a.e.
in $X$.}
\]
\item[(vi)]$\cG$ is dense in $L^2(X,\mm)$ and
for every $f\in\cG$ and $t>0$ $\sP t f$ belongs to $\cG$ with
%
\begin{equation}
\label{eqBEKnu.6} \qquad\Gamma(\sP t f)+2\nu\rmI_{-2K}(t) (\DeltaE\sP t f
)^2 \le\rme^{-2K t}\sP t \Gamma(f) \qquad\mbox{$\mm$-a.e. in
$X$.}
\end{equation}
\end{longlist}
If one of these equivalent properties holds, then $\cG=\cV$ (i.e., $\cE
$ admits the Carr\'e du Champ $\Gamma$ in $\cV$).
\end{corollary}

\begin{pf}
The implication (i)~$\Rightarrow$ (ii) is obvious, {choosing $s=0$}.
The converse implication is also true under the regularity assumption of
(i): it is sufficient to pass to the limit in
\eqref{eqBEKnu.2} as $s\uparrow t$ and then as $t\downarrow0$.

(ii)~$\Rightarrow$ (iii) follows by \eqref{eqGG3}
when $\DeltaE\varphi\in L^2\cap L^\infty(X,\mm)$;
the general case follows by approximation
by the same argument we used in the proof of Lemma~\ref{leB-regularity}.

(iii)~$\Rightarrow$ (iv):
by applying \eqref{eqfromzeta1} (with obvious notation) we
get
\[
\rmI_{2K}(t) \boldsymbol{\Gamma}[\sP t f;\varphi]+ 2\nu \mathrm
I_{2K,2}(t)\int_X (\DeltaE\sP t f )^2
\varphi\,\d\mm \le \frac{1}2\int_X \bigl(\sP t
\bigl(f^2 \bigr) -(\sP t f)^2 \bigr)\varphi\,\d\mm
\]
for every nonnegative $\varphi\in\cV_{\infty}$.
Thus, setting $h:=\sP t(f^2)-(\sP t f)^2\in L^1_+(X,\mm)$,
the linear functional $\ell$ on $\cV_{\infty}$ defined by
$\ell(\varphi):={\rmI_{2K}(t)} \boldsymbol{\Gamma}[\sP t f;\varphi]$
satisfies
%
\begin{equation}
\label{eq19} 0\le\ell(\varphi)\le\int_X h \varphi\,\d
\mm\qquad \mbox{for every }\varphi\in\cV_{\infty}, \varphi\ge0.
\end{equation}
Since $\cV_{\infty}$ is a lattice of functions {generating
$\cB$ [because
$\cB$ is complete and
$\cV_{\infty}$ is dense in $L^2(X,\mm)$]} satisfying the Stone
property
$\varphi\in\cV_{\infty} \Rightarrow\varphi\land1\in
\cV_{\infty}$
and clearly \eqref{eq19} yields $\ell(\varphi_n)\to0$
whenever $(\varphi_n)_{n\ge0}\subset\cV_{\infty}$ is a sequence
of functions pointwise decreasing to $0$,
Daniell construction (\cite{Bogachev07}, Theorem~7.8.7),
and the Radon--Nikodym theorem yields
$ \boldsymbol{\Gamma}[\sP t f;\varphi]=\int_X g\varphi\,\d\mm$
for some $g\in L^1_+(X,\mm)$, so that
$\sP tf\in\cG$ and \eqref{eqBEKnu.4} holds.

This argument also shows that
$\cG$ is invariant
under the action of $(\sP t)_{t\ge0}$ and
dense in $L^2(X,\mm)$.
A standard approximation
argument yields the density in $\cV$ (see, e.g., \cite{AGS11b}, Lemma~4.9) and, therefore, $\cG=\cV$
(since $\cG$ is closed in $\cV$; see also
\cite{Bouleau-Hirsch91}, Proposition~4.1.3).

Analogously, (iii)~$\Rightarrow$ (v) follows by \eqref{eqfromzeta2}, while
(iii)~$\Rightarrow$ (vi) follows by \eqref{eqGG6bis}.

Let us now show that (vi)~$\Rightarrow$ (iii).
Since $\cG$ is dense in $L^2(X,\mm)$ and invariant with
respect to $(\sP{t})_{t\ge0}$, we already observed that $\cG=\cV$.
Let us now write \eqref{eqBEKnu.6} with $h>0$ instead
of $t$
and with $f:= \sP{t-s}v$ for some $0<h< s<t$.
Multiplying by $ \sP{s-h} \varphi$ and integrating with
respect to $\mm$,
we obtain
\[
\sfB_{t}[v;\varphi](s-h)+4\nu \rmI_{-2K}(h)
\sfA^{\Delta}_{t}[v;\varphi](s-h)\le\rme^{-2Kh}{\sfB
_{t}[v;\varphi](s)}.
\]
It is not restrictive to assume
$\DeltaE\varphi\in L^2\cap L^\infty(X,\mm)$, so
that $\sfB$ is of class $\rmC^1$ in $(0,t)$. We
subtract $\sfB_{t}[v;\varphi](s)$ from both sides of the inequality,
we divide by $h>0$ and let $h\down0$ obtaining
\[
\frac\partial{\partial s}\sfB_{t}[v;\varphi](s)-2K\sfB
_{t}[v;\varphi](s)\ge4\nu \sfA^{\Delta}_{t}[v;
\varphi](s),
\]
that is, \eqref{eqBEKnu.3}.

To show that (iv)~$\Rightarrow$ (iii), we first write
\eqref{eqBEKnu.4} at $t=h>0$ in the form
\[
\rmI_{2K,2}(h) \bigl(K \Gamma(\sP h f)+2\nu (\DeltaE\sP h f
)^2 \bigr) \le\tfrac12 \sP h\bigl(f^2\bigr)-\tfrac12 (
\sP h f )^2- h \Gamma(\sP h f),
\]
obtaining by subtracting $h \Gamma(\sP h f)$ from both
sides of the inequality. Then we choose $f= \sP{t-s-h}v$
and we multiply the inequality by $ \sP s\varphi$, with
$\varphi\in L^2\cap L^\infty(X,\mm)$ nonnegative
and $\DeltaE\varphi\in L^2\cap L^\infty(X,\mm)$.
We obtain
\begin{eqnarray*}
&&\rmI_{2K,2}(h) \bigl(2K \sfB_{t}[v;\varphi](s)+ 4\nu
\sfA^{\Delta}_{t}[v;\varphi](s) \bigr)
\\
&&\qquad\le\sfA_{t}[v;\varphi](s+h)- \sfA_{t}[v;
\varphi](s)-h\sfB_{t}[v;\varphi](s).
\end{eqnarray*}
Since $\sfA$ is of class $\rmC^2$ and $\sfA'=\sfB$, dividing by $h^2>0$
and passing to the limit as $h\down0$ a simple Taylor expansion
yields
\[
\frac{1}2\frac{\partial^2}{\partial s^2}\sfA_{t}[v;\varphi](s) \ge
\frac{1}2 \bigl(2K \sfB_{t}[v;\varphi](s)+ 4\nu
\sfA^{\Delta}_{t}[v;\varphi](s) \bigr).
\]
A similar argument shows the last implication (v)~$\Rightarrow$ (iii).
\end{pf}

%
\begin{definition}[{[The condition $\BE KN$]}]\label{defBEKN}
Let $K\in\R$ and $\nu\ge0$. We say that a Dirichlet form
$\cE$ in $L^2(X,\mm)$ as in \eqref{eq50}
satisfies a functional $\BE KN$ condition
if one of the equivalent properties in
Corollary~\ref{leBE-equivalent} holds
with $N:=1/\nu$.
\end{definition}

Notice that
\[
\BE KN \Rightarrow \BE{K'}
{N'}\qquad\mbox{for every }K'\le K, N'
\ge N,
\]
in particular $\BE KN \Rightarrow\BE K\infty$.

%
\begin{remark}[(Carr\'e du Champ in the case $N=\infty$)]
If a strongly local Dirichlet form $\cE$ satisfies $\BE K\infty$ for
some $K\in\R$, then
it admits a Carr\'e du Champ $\Gamma$ on $\cV$,
that is, $\cG=\cV$, by (v) of Corollary~\ref{leBE-equivalent}; moreover,
the spaces
%
\begin{equation}
\begin{aligned} \cVu&:= \bigl\{\varphi\in\cV_{\infty}\dvtx
\Gamma(\varphi )\in L^\infty(X,\mm) \bigr\},
\\
\cVd&:= \bigl\{\varphi\in\cVu\dvtx \DeltaE\varphi\in L^\infty(X,\mm)
\bigr\}, \end{aligned} %
\label{eq67}
\end{equation}
are dense in
$\cV$:
in fact \eqref{eqBEKnu.6} shows that they are invariant under the
action of $(\sP t)_{t\ge0}$ and combined with \eqref{eqBEKnu.4}
[and possibly combined with a further mollification as in \eqref{eqGG4}
in the case of $\cVd$] it also shows that
any element of $L^2\cap L^\infty(X,\mm)$
belongs to
their closure w.r.t. the $L^2(X,\mm)$ norm.
The invariance and the standard approximation argument of, for example,
\cite{AGS11b}, Lemma~4.9, yield the density in $\cV$.
\end{remark}

\section{Energy metric measure structures}
\label{secMeMe}

In this section,
besides the standing assumptions we made on $\cE$, we shall
study the relation between the measure/energetic structure of $X$
and an additional metric structure. Our main object will be the
canonical distance ${\mathsf d}_\cE$
associated to the Dirichlet form $\cE$, that we will introduce and
study in the next Section~\ref{subsecEMMspace}.
Before doing that, we will recall the metric notions that will be
useful in the following. Since many properties will just depend of
a few general compatibility conditions between the metric and the
energetic structure, we will try to enucleate such a conditions
and state the related theorems in full generality.

Our first condition just refers to the measure $\mm$ and a
distance ${\mathsf d}$ and it does not involve the Dirichlet form $\cE$:

{
\renewcommand\thelonglist{\MD.\alph{longlist}}
\renewcommand\labellonglist{(\thelonglist)}
%
\begin{condition*}[(MD: Measure-Distance interaction)]\label{MD}
${\mathsf d}$ is a distance on $X$
such that:
\begin{longlist}[(\MD.b)]
\item\label{MDa} $(X,{\mathsf d})$ is a complete and separable
metric space, $\cB$ coincides with the completion of the
Borel $\sigma$-algebra of $(X,{\mathsf d})$ with respect to $\mm$, and
$\supp(\mm)=X$;
\item\label{MDb}
$\mm(B_r(x))<\infty$ for every $x\in X$, $r>0$.
\end{longlist}
\end{condition*}
}%

Besides the finiteness condition (\ref{MDb}), we will often
assume a
further exponential growth
condition on the measures of the balls of $(X,{\mathsf d})$,
namely that there exist $x_0\in X$, $M>0$ and $c\geq0$ such that
\renewcommand{\theequation}{MD.exp}%
\begin{equation}
\label{eq80} \mm\bigl(B_r(x_0)\bigr)\le M\exp\bigl(c
r^2\bigr)\qquad\mbox{for every }r\ge0.
\end{equation}
\renewcommand{\theequation}{\arabic{section}.\arabic{equation}}%
\setcounter{equation}{0}%
In this case, we will collectively refer to the above conditions
(\hyperref[MD]{\MD}) and \eqref{eq80} as (MD$+$exp).
%
\subsection{Metric notions}
\label{subsecpreliminaries}

In this section, we recall a few basic definitions and results which
are related to a metric
measure space $(X,{\mathsf d},\mm)$ satisfying (\hyperref[MD]{\MD}).

\textit{Absolutely continuous curves, Lipschitz
functions and slopes.}
$\AC p{[a,b]}X$, $1\leq p\leq\infty$, is the collection of all the
absolutely continuous curves
$\gamma\dvtx [a,b]\to X$
with finite $p$-energy: $\gamma\in\AC p{[a,b]}X$ if
there exists $v\in L^p(a,b)$ such that
%
\begin{equation}
\label{eq47} {\mathsf d}\bigl(\gamma(s),\gamma(t)\bigr)\le\int
_s^t v(r)\,\d r\qquad\mbox {for every } a\le s
\le t\le b.
\end{equation}
The metric velocity of $\gamma$, defined by
\[
\llvert \dot\gamma\rrvert (r):=\lim_{h\to0}
\frac{{\mathsf d}(\gamma(r+h),\gamma
(r))}{\llvert h\rrvert },
\]
exists for $\Leb1$-a.e. $r\in(a,b)$, belongs to $L^p(a,b)$,
and provides the minimal function $v$, up to $\Leb1$-negligible sets,
such that \eqref{eq47} holds.
The length of an absolutely continuous curve $\gamma$ is then defined by
$\int_a^b\llvert \dot\gamma\rrvert (r)\,\d r$.

We say that $(X,{\mathsf d})$ is a \emph{length space} if
for every $x_0, x_1\in X$
%
\begin{equation}
\label{eq21} {\mathsf d}(x_0,x_1)=\inf \biggl\{\int
_0^1\llvert \dot\gamma\rrvert (r)\,\d r\dvtx
\gamma\in\AC{} {[0,1]}X, \gamma(i)=x_i \biggr\}.
\end{equation}
We denote by $\Lip(X)$ {the space of all Lipschitz functions}
$\varphi\dvtx X\to\R$, by $\Lip_b(X)$ the subspace of bounded functions
and by $\Lip^1(X)$ the subspace of functions with Lipschitz constant
less than $1$.

Every Lipschitz function $\varphi$
is absolutely continuous along any absolutely continuous curve;
we say that a bounded Borel function
$g\dvtx  X\to[0,\infty)$ is an \emph{upper gradient} of $\varphi\in\Lip
(X)$ if
for any curve $\gamma\in\AC{}{[a,b]}X$ the absolutely continuous map
$\varphi\circ\gamma$ satisfies
%
\begin{equation}
\label{eq15} \biggl\llvert \frac\d{\d t}\varphi\bigl(\gamma(t)\bigr)\biggr
\rrvert \le g\bigl(\gamma(t)\bigr) \llvert \dot\gamma\rrvert (t)\qquad\mbox{for $
\Leb 1$-a.e. $t\in(a,b)$.}
\end{equation}
Among the upper gradients of a function $\varphi\in\Lip(X)$,
its \emph{slopes} and its \emph{local Lipschitz constant}
play a crucial role: they are defined by $0$ at every isolated point and by
\begin{eqnarray*}
&\displaystyle \bigl\llvert \rmD^\pm\varphi\bigr\rrvert (x) := \limsup
_{y\to x}\frac
{ (\varphi(y)-\varphi(x) )_\pm}{{\mathsf d}(y,x)},\qquad \llvert \rmD\varphi\rrvert (x):=
\limsup_{y\to x}\frac
{\llvert \varphi(y)-\varphi(x)\rrvert }{{\mathsf d}(y,x)},&
\\
&\displaystyle \bigl\llvert \rmD^*\varphi\bigr\rrvert (x) := \mathop{\limsup
_{y,z\to x}}\limits
_{y\neq z}\frac
{\llvert \varphi(y)-\varphi(z)\rrvert }{{\mathsf d}(y,z)},&
\end{eqnarray*}
at every accompulation point $x\in X$. Whenever $(X,\mathsf d)$
    is a length space we have
\begin{equation}
\label{eq18} \bigl\llvert \rmD^*\varphi\bigr\rrvert (x) =\limsup
_{y\to x}\llvert \rmD\varphi\rrvert (y), \qquad\Lip(\varphi)=\sup
_{x\in X}\bigl\llvert \rmD\varphi(x)\bigr\rrvert = \sup
_{x\in X}\bigl\llvert \rmD^*\varphi(x)\bigr\rrvert .
\hspace*{-25pt}
\end{equation}
In fact, \eqref{eq15} written for $g:=\llvert \rmD\varphi\rrvert $ and
the length condition \eqref{eq21} easily yield
\[
\bigl\llvert \varphi(y)-\varphi(z)\bigr\rrvert \le{\mathsf d}(y,z)
\sup_{B_{2r}(x)} \llvert \rmD \varphi \rrvert \qquad\mbox{if } y,z\in
B_r(x)
\]
and provide the inequality $\llvert \rmD^*\varphi\rrvert \leq\limsup_{y\to x}\llvert \rmD\varphi\rrvert (y)$.
The proof of the converse inequality is trivial and a similar
argument shows the last identity in \eqref{eq18}.

\textit{The Hopf--Lax evolution formula.}
Let us suppose that $(X,{\mathsf d})$ is a metric space;
the Hopf--Lax evolution map
$Q_t\dvtx \rmC_b(X)\to\rmC_b(X)$, $t\ge0$, is defined by
%
\begin{equation}
Q_t f(x):=\inf_{y\in X} f(y)+\frac{{\mathsf d}^2(y,x)}{2t},
\qquad {Q_0f(x)=f(x).}\label{eq11}
\end{equation}
We introduce as in \cite{AGS11a}, Section~3, the maps
\[
\sfD^+(x,t) :=\sup_{(y_n)}\limsup
_n {\mathsf d}(x,y_n),\qquad \sfD^-(x,t) :=\inf
_{(y_n)}\liminf_n {\mathsf
d}(x,y_n),
\]
where the supremum and the infimum run among minimizing sequences for
\eqref{eq11}.
We recall that $\sfD^+$ and $\sfD^-$ are respectively upper and lower
semicontinuous,
nondecreasing w.r.t. $s$, and that $\sfD^+(x,r)\leq\sfD^-(x,s)\leq
\sfD^+(x,s)$ whenever $0<r<s$. These properties imply
$\sfD^-(x,s)=\sup_{r<s}\sfD^+(x,r)$. We shall need the inequality
%
\begin{equation}
Q_{s'}f(x)-Q_sf(x)\leq\frac{(\sfD^+(x,s))^2}{2} \biggl(
\frac{1}{s'} -\frac{1}s \biggr),\qquad s'>s,\label{eq13}
\end{equation}
as well as the pointwise properties
%
\begin{equation}
\label{eqidentities0} -\frac{\d^\pm}{\d s}Q_sf(x)=\frac{(\sfD^\pm(x,s))^2}{2s^2}, \qquad
\llvert \rmD Q_sf\rrvert (x){\le} \frac{\sfD^+(x,s)}{s},
\end{equation}
(these are proved in Proposition~3.3 and Proposition~3.4
of \cite{AGS11a}). Since
\[
{\mathsf d}(y,x)>2s\Lip(f)\quad\Rightarrow\quad f(y)+\frac{{\mathsf d}^2(x,y)}{2s}>f(x)\ge
Q_s f(x),
\]
we immediately find $\sfD^+(x,s)\le2s \Lip(f)$.

Since by \eqref{eq13} the map $s\mapsto Q_s f(x)$ is locally
Lipschitz in $(0,\infty)$, integrating the first identity of \eqref
{eqidentities0}
in the interval $(\eps,t)$, $\eps>0$, and then letting $\eps\down0$
we get
\[
f(x)-Q_t f(x)=\int_0^t
\frac{(\sfD^+(x,s))^2}{2s^2}\,\d s= \frac{t} 2\int_0^1
\biggl(\frac{\sfD^+(x,tr)}{tr} \biggr)^2\,\d r.
\]
Combining the above identity with the formula expressing the
descending slope (see \cite{AGS08}, Lemma~3.1.5)
\[
\bigl\llvert \rmD^- f\bigr\rrvert ^2(x)=2\limsup
_{t\down0}\frac{f(x)-Q_t f(x)}t,
\]
we end up with
%
\begin{equation}
\label{eq25} \llvert \rmD f\rrvert ^2(x)\ge\bigl\llvert \rmD^- f
\bigr\rrvert ^2(x)= \limsup_{t\down0} \int
_0^1 \biggl(\frac{\sfD^+(x,tr)}{tr}
\biggr)^2\,\d r.
\end{equation}
When $(X,{\mathsf d})$ is a length space
$(Q_t)_{t\ge0}$ is a semigroup and we have the refined identity
\cite{AGS11a}, Theorem~3.6,
%
\begin{equation}
\label{eqidentities} \frac{\d^+}{\d s}Q_sf(x)=-\frac{1}{2}\llvert
\rmD Q_s f\rrvert ^2(x)=- \frac
{(\sfD^+(x,s))^2}{2s^2}.
\end{equation}
In addition, \eqref{eqidentities0} and the length property of $X$
yield the a priori bounds
%
\begin{equation}
\label{eqaprioriQt} \operatorname{Lip}(Q_s f)\leq2 \operatorname{Lip}(f),\qquad
\operatorname{Lip} \bigl(Q_\cdot f(x)\bigr)\leq2 \bigl[
\operatorname{Lip}(f) \bigr]^2.
\end{equation}

\textit{The Cheeger energy.}
The Cheeger energy of a function $f\in L^2(X,\mm)$ is
defined as
\[
\C(f):=\inf \biggl\{\liminf_{n\to\infty}\frac{1}2
\int_X \llvert \rmD f_n\rrvert ^2
\,\d\mm\dvtx f_n\in\Lip_b(X), f_n\to f\mbox{
in }L^2(X,\mm) \biggr\}.
\]
If $f\in L^2(X,\mm)$ with $\C(f)<\infty$, then there exists a unique
function $\llvert \rmD f\rrvert _w\in L^2(X,\mm)$, called \emph{minimal weak
gradient of $f$},
satisfying the two conditions
\begin{eqnarray*}
 &&\Lip_b(X)\cap L^2(X,\mm)\ni
f_n\weakto f,\qquad  \llvert \rmD f_n\rrvert \weakto G\qquad
\mbox{in }L^2(X,\mm)\\
 &&\qquad \Rightarrow\quad\llvert \rmD f\rrvert
_w\le G,\\
&& \C(f)=\frac{1}2 \int_X \llvert \rmD
f\rrvert _w^2\,\d\mm.
\end{eqnarray*}
In the next Section~\ref{subsecEMMspace}, we will also use
a {more refined}
approximation result,
replacing $\llvert \rmD f_n\rrvert $ with $\llvert \rmD^* f_n\rrvert $ in the approximation,
proved in \cite{AGS11c}, Section~8.3
(see also \cite{Ambrosio-Colombo-DiMarino12} for a more detailed proof):
for every $f\in L^2(X,\mm)$ with $\C(f)<\infty$
there exist $f_n\in\Lip_b(X)\cap L^2(X,\mm)$ such that
%
\begin{equation}
\label{eq105} f_n\to f,\qquad\bigl\llvert \rmD^* f_n
\bigr\rrvert \to\llvert \rmD f\rrvert _w\qquad\mbox{strongly in
}L^2(X,\mm).
\end{equation}

\textit{Wasserstein distances.}
The metric structure allows us to introduce the
corresponding spaces $
\prob{X}$ of Borel
probability measures and ${\mathscr P}_p(X)$ of Borel probability
measures with finite
$p$th moment, namely $\mu\in{\mathscr P}_p(X)$ if
\[
\int_X {\mathsf d}^p(x,x_0)
\,\d\mu(x){<\infty} \qquad {\mbox{for some, and then for all, $x_0\in
X$.}}
\]
The $L^p$-Wasserstein transport (extended) distance $W_p$ on $\prob X$
is defined by
%
\begin{eqnarray}
\label{eq10} &&W_p^p(\mu_1,
\mu_2)
\nonumber
\\[-8pt]
\\[-8pt]
&&\qquad :=\inf \biggl\{\int_{X\times
X}{\mathsf
d}^p(x_1,x_2)\,\d\mmu(x_1,x_2)
\dvtx \mmu\in
\prob{X\times X}, \pi^i_\sharp \mmu=
\mu_i \biggr\},
\nonumber
\end{eqnarray}
where $\pi^i\dvtx X\times X\ni(x_1,x_2)\to x_i$ is the coordinate map and
for a Borel measure $\mu\in\prob Y$ on a metric space $Y$ and every
Borel map $\rr\dvtx Y\to X$, the push-forward measure $\rr_\sharp\mu\in
\prob X$
is defined by
\[
\rr_\sharp\mu(B):=\mu\bigl(\rr^{-1}(B)\bigr)\qquad\mbox{for
every Borel set }B\subset X.
\]
In particular, the competing measures $\mmu\in
\prob{X\times X}$ in
\eqref{eq10} have
marginals $\mu_1$ and $\mu_2$, respectively.

We also introduce a family of bounded distances on $\prob X$
associated to a
\[
\begin{tabular}{@{}p{350pt}@{}} continuous, concave and bounded modulus
of continuity $\beta\dvtx [0,\infty)\to[0,\infty)$, with $0=\beta(0)<\beta(r)$
for every $r>0$. \end{tabular} %
\]
As in \eqref{eq10}, we set
\[
W_{(\beta)}(\mu_1,\mu_2):=\inf \biggl
\{\int_{X^2}\beta \bigl({\mathsf d}(x_1,x_2)
\bigr)\,\d\mmu(x_1,x_2)\dvtx \mmu\in
\prob{X
\times X}, \pi ^i_\sharp \mmu=\mu_i \biggr\}.
\]
$W_{(\beta)}$ is thus the $L^1$-Wasserstein distance induced
by the
bounded distance ${\mathsf d}_\beta(x_1,x_2):=\beta({\mathsf d}(x_1,x_2))$.
$(\prob X,W_{(\beta)})$ is then a complete and separable
metric space,
whose topology coincides with the topology of weak convergence of
probability measures.

\textit{Entropy and $\RCD K\infty$ spaces.}
In the following, we will fix $x_0\in X, z>0, c\ge0$ such that
%
\begin{equation}
\label{eq8} \tilde\mm=\frac{1}z \rme^{-V^2}\mm\in
\Probabilities{X}\qquad \mbox{with } V(x):=\sqrt c {\mathsf d}(x,x_0).
\end{equation}
Notice that in the case $\mm(X)<\infty$ we can always take
$V\equiv c=0$ with $z=\mm(X)$.
When $\mm(X)=\infty$, the possibility to choose $x_0\in X, z>0, c\ge0$
satisfying \eqref{eq8}
follows from \eqref{eq80} [possibly with a different constant $c$; it
is in fact
equivalent to \eqref{eq80}].
If $\nn\dvtx \cB\to[0,\infty]$ is $\sigma$-additive, the \emph
{relative entropy} $\operatorname{Ent}_{\nn}(\rho)$ of a
probability measure $\rho\dvtx \cB\to[0,1]$ with respect to $\nn$ is
defined by
\[
\operatorname{Ent}_{\nn}(\rho):= %
\cases{
\displaystyle \int_X f\log f\,\d\nn,&\quad if $\rho=f\nn$;
\cr
\displaystyle +\infty,&\quad otherwise. } %
\]
The expression makes sense if $\nn$ is a probability measure, and
thanks to Jensen's inequality
defines a nonnegative functional. More generally,
we recall (see \cite{AGS11a}, Lemma~7.2, for the simple proof) that,
when $\nn=\mm$ and \eqref{eq8}
hold,
the formula above makes sense on measures $\rho=f\mm\in\probt{X}$
thanks to the fact that the negative part
of $f\log f$ is $\mm$-integrable. Precisely, defining $\tilde\mm\in
\Probabilities{X}$ as in \eqref{eq8} above,
the following formula for the change of reference measure will be
useful to estimate the negative part of $\entv(\rho)$:
\[
\entv(\rho)=\operatorname{Ent}_{\tilde\mm}(\rho)-\int
_X V^2(x)\,\d \rho(x)-\log z.
\]

%
\begin{definition}[{[$\RCD{K}{\infty}$ spaces]}]
Let $(X,{\mathsf d},\mm)$ be a metric measure space satisfying
(MD$+$exp) and the length property \eqref{eq21}.
We say that $(X,{\mathsf d},\mm)$ has Riemannian curvature
bounded from below by $K\in\R$ if for all $\rho\in\probt{X}$ there
exists a solution
$(\sH{t}\rho)_{t\ge0}\subset\probt X$
of the $\EVI_K$-differential inequality starting from
$\rho$, namely
$\sH{t}\rho\to\rho$ as $t\downarrow0$
and (denoting by $\frac{\d^+}{\d t}$ the upper right derivative)
%
\begin{eqnarray}
\label{defEVIK} \frac{\d^+}{\d t}\frac{W_2^2(\sH{t}\rho,\nu)}2+\frac{K}2W_2^2(
\sH{t}\rho,\nu)+\entv(\sH{t}\rho)\leq\entv(\nu)
\nonumber
\\[-8pt]
\\[-8pt]
\eqntext{\mbox{for every $t\in(0,\infty)$}}
\end{eqnarray}
for all $\nu\in\probt{X}$ with $\entv(\nu)<\infty$.
\end{definition}

As we already quoted in the \hyperref[sec1]{Introduction},
among the properties of\break $\operatorname{RCD}(K,\infty)$ spaces proved
in \cite{AGS11b}
we recall that the Cheeger energy
\renewcommand{\theequation}{QCh}%
\begin{equation}
\label{eq88} %
\begin{tabular}{@{\quad\qquad}p{320pt}@{\hspace*{-5pt}}} $\C$ is quadratic, that
is, $\C(f)=\frac{1}2{\cE_\C}(f)$ for a Dirichlet form $
\cE_\C$ as in \eqref{eq50}, with $\llvert \rmD f\rrvert
_w^2=\Gamma (f)$ for every $f\in D(\C )=\cV$,
\end{tabular} %
\end{equation}
\renewcommand{\theequation}{\arabic{section}.\arabic{equation}}%
\setcounter{equation}{14}
(in particular $\cG=\cV$ and $\cE_\C$ admits the Carr\'e du Champ
$\Gamma$ in $\cV$)
and {$\cE_\C$}
satisfies the $\BE K\infty$ condition.
A further crucial property will be recalled in Section~\ref{subsecEMMspace} below; see
Condition (\hyperref[ED]{\ED}) and Remark~\ref{remQCh-vs-Dirichlet}.

\subsection{The dual semigroup and its contractivity properties}\label
{secdual}

In this section, we study the contractivity property
of the dual semigroup of $(\sP t)_{t\ge0}$ in the spaces of Borel
probability measures.

Thus, $\cE$ is a strongly local Dirichlet form as in \eqref{eq50},
$(\sP t)_{t\ge0}$ satisfies the mass-preserving property \eqref{eq4} and
${\mathsf d}$ is a distance on $X$ satisfying condition (\hyperref[MD]{\MD})
[assumption \eqref{eq80} is not needed here].

We see how, under the mild contractivity property
%
\begin{eqnarray}
\label{eqconvpuntLip1} \sP tf\in\operatorname{Lip}_b(X)\quad \mbox{and} \quad
\operatorname{Lip}(\sP t f)\leq C(t)\operatorname {Lip}(f)
\nonumber
\\[-8pt]
\\[-8pt]
\eqntext{\mbox{for all }f\in\operatorname{Lip}_b(X)\cap
L^2(X,\mm),}
\end{eqnarray}
with $C$ bounded on all intervals $[0,T]$, $T>0$, a dual semigroup
$\sfH_t$ in ${\mathscr P}(X)$ can be defined, satisfying the
contractivity property
\eqref{eqW1contractivity} below w.r.t. $W_{(\beta)}$ and
to $W_1$.
This yields also the fact that $\sP t$ has a (unique) pointwise
defined version $\tsP{t}$, canonically defined also on bounded Borel functions,
and mapping $\rmC_b(X)$ to $\rmC_b(X)$ [we will always
identify $\sP t f$ with $\tsP{t} f$ whenever
$f\in\rmC_b(X)$]. Then we shall prove, following the lines of \cite
{Kuwada10},
that in length metric spaces the pointwise Bakry--\'Emery-like assumption
%
\begin{eqnarray}
\label{eqBEpointwise} \llvert \rmD\sP t f\rrvert ^2(x)\leq
C^2(t) \tsP{t}\llvert \rmD f\rrvert ^2(x)
\nonumber
\\[-8pt]
\\[-8pt]
\eqntext{\mbox{for all $x\in X$, $f\in\mbox {Lip}_b(X)\cap
L^2(X,\mm)$,}}
\end{eqnarray}
with $C$ bounded on all intervals $[0,T]$, $T>0$,
provides contractivity of $\sH{t}$ even w.r.t. $W_2$. Notice
that formally \eqref{eqBEpointwise} implies \eqref{eqconvpuntLip1},
but one has to take into
account that \eqref{eqBEpointwise} involves a pointwise defined
version of the
semigroup, which {might} depend on \eqref{eqconvpuntLip1}.

A crucial point here is that we want to avoid
doubling or local Poincar\'e assumptions on
the metric measure space.
For the aim of this section, we introduce the following notation:
%
\begin{equation}
\begin{tabular}{@{\quad\qquad}p{320pt}@{\hspace*{-5pt}}} $\cZ$ is the collection of probability
densities $f\in L^1_+(X,\mm)$, $\mathcal K$ is the set of nonnegative
bounded Borel functions $f\dvtx X\to\R$ with bounded support. \end{tabular}
\label{eq16}
\end{equation}

%
\begin{proposition}
\label{propprecise-rep}
Let $\cE$ and $(\sP t)_{t\ge0}$ be as in \eqref{eq50} and \eqref{eq4}
and let ${\mathsf d}$ be a distance on $X$ satisfying the condition
(\hyperref[MD]{\MD}). If
\eqref{eqconvpuntLip1} holds, then:
\begin{longlist}[(iii)]
\item[(i)]
The mapping $\sH{t}(f\mm):=(\sP t f)\mm$, $f\in\cZ$,
uniquely extends to a $W_{(\beta)}$-Lipschitz map $\sfH_t\dvtx
{\mathscr P}(X)\to\prob X$ satisfying
for every $\mu,\nu\in\mathscr P(X)$
%
\begin{eqnarray}
\label{eq55} W_{(\beta)}(\sfH_t\mu,\sfH_t\nu)&
\leq& \bigl(C(t)\lor 1 \bigr) W_{(\beta)}(\mu,\nu), 
\\
\label{eqW1contractivity} W_1(\sfH_t\mu,\sfH_t\nu)&
\leq& C(t)W_1(\mu,\nu), 
\end{eqnarray}
with $C(t)$ given by \eqref{eqconvpuntLip1}.
\item[(ii)] Defining $\tsP{t}f(x):=\int_X f\,\d\sH{t}\delta_x$ on
bounded or nonnegative Borel functions,
$\tsP{t}$ {is everywhere defined} and maps $\rmC_b(X)$ to
$\rmC_b(X)$.
Moreover,
$\tsP{t}$ is a version of $\sP t$ for all Borel functions $f$ with
$\int_X \llvert f\rrvert \,\d\mm<\infty$,
namely $\tsP{t}f(x)$ is
defined and $\sP tf(x)=\tsP{t}f(x)$ for
$\mm$-a.e. $x\in X$. In {particular,}
$\tsP{t} f$ is $\mm$-a.e. defined for every Borel function
semiintegrable w.r.t. $\mm$.
\item[(iii)] $\sfH_t$ is dual to $\tsP{t}$ in the following sense:
%
\begin{eqnarray}
\label{eqdual} %
\int_X f\,\d\sH{t}\mu=\int
_X \tsP{t} f\,\d\mu
\nonumber
\\[-8pt]
\\[-8pt]
\eqntext{\mbox{for all $f\dvtx X\to\R$ bounded Borel, $\mu\in{\mathscr P}(X)$.}}
\end{eqnarray}
\item[(iv)]
For every $f\in\rmC_b(X)$ and $x\in X$, we have
$\lim_{t\downarrow0}\tsP{t} f(x)=f(x)$. In particular,
for every $\mu\in\prob X$ the map $t\mapsto\sfH_t\mu$
is weakly continuous in $\prob X$.
\end{longlist}
\end{proposition}

\begin{pf} The concavity of $\beta$ yields that $\beta$ is
subadditive, so that ${\mathsf d}_\beta$ is a distance.
Let us first prove that $\sP t$ maps
${\mathsf d}_\beta$-Lipschitz functions in
${\mathsf d}_\beta$-Lipschitz functions.

We use the envelope representation
\[
\beta(r)=\inf_{(a,b)\in\rmB}a+br,\qquad \rmB= \bigl\{(a,b)\in[0,
\infty)^2
\dvtx \beta(s)\le a+bs \mbox{ for every }s\ge0
\bigr\},
\]
and the fact that a function $\varphi\dvtx X\to\R$ is $\ell$-Lipschitz
with respect to a
distance ${\mathsf d}$ on $X$ if and only if
\[
\varphi(x)\le R_\ell\varphi(x):=\inf_{y\in X}
\varphi(y)+\ell {\mathsf d}(x,y)\qquad\mbox{for every }x\in X.
\]
It is easy to check that if $\varphi$ is bounded, then $R_\ell\varphi
$ is bounded
and satisfies
%
\begin{equation}
\label{eq60} \inf_X \varphi\le R_\ell
\varphi(x)\le\varphi(x)\qquad\mbox{for every } x\in X, \ell\ge0.
\end{equation}
{Furthermore, if $\varphi$ has also bounded support then $R_\ell
\varphi$ has bounded support as well,}
so that $R_\ell$ maps $\mathcal K$ in $\mathcal K$. The contractivity property
\eqref{eqconvpuntLip1} then yields for every $\varphi\in\Lip
_b(X)\cap L^2(X,\mm)$ with
$\Lip(\varphi)\le b$
\[
\sP t \varphi\le R_{C(t)b}(\sP t \varphi).
\]
Let us now suppose that $\varphi\in\mathcal K$
is ${\mathsf d}_\beta$-Lipschitz, with
Lipschitz constant less than $1$, so that for every $(a,b)\in\rmB$
\[
\varphi(x)\le\inf_{y\in X} \varphi(y)+\beta\bigl({\mathsf
d}(x,y)\bigr)\le \inf_{y\in X} \varphi(y)+a+b {\mathsf
d}(x,y)=a+R_b \varphi(x).
\]
Since $(\sP t)_{t\ge0}$ is order preserving,
we get
for $\varphi\in\mathcal K$
\[
\sP t \varphi\le a+\sP t( R_b \varphi)\le a+R_{C(t)b} \bigl(
\sP t (R_b \varphi) \bigr) \le a+R_{C(t)b} (\sP t \varphi ),
\]
where we used the right inequality of \eqref{eq60} and the fact that
$\Lip(R_b\varphi)\le b$.
It follows that for every $x,y\in X$ and every $(a,b)\in\rmB$
\[
\sP t \varphi(x)\le\sP t\varphi(y)+a+C(t)b{\mathsf d}(x,y),
\]
{that is,}
\[
\sP t\varphi(x)-\sP t \varphi(y)\le\beta\bigl(C(t){\mathsf d}(x,y)\bigr).
\]
By Kantorovich duality, for $f, g\in\cZ$ we get
\begin{eqnarray*}
W_{(\beta)}(\sP tf \mm,\sP tg \mm)& =&\sup \biggl\{\int
_X \varphi (\sP tf-\sP tg )\,\d\mm 
\dvtx \varphi\in\mathcal K, \operatorname{Lip}_{{\mathsf d}_\beta}(\varphi)\leq1
\biggr\}
\\
&=&\sup \biggl\{\int_X \sP t\varphi (f-g )\,\d\mm
\dvtx \varphi\in \mathcal K, \operatorname{Lip}_{{\mathsf d}_\beta}(
\varphi)\leq1 \biggr\}
\\
&\leq& \bigl(C(t)\lor1\bigr)W_{(\beta)}(f\mm,g\mm).
\end{eqnarray*}
Hence, \eqref{eq55} holds when $\mu=f\mm$, $\nu=g\mm$.
By the density of $\{f\mm\dvtx f\in\cZ\}$ in ${\mathscr P}(X)$
w.r.t. $W_{(\beta)}$ we
get \eqref{eq55} for arbitrary $\mu,\nu\in\prob X$. A similar
argument yields \eqref{eqW1contractivity}.

(ii) Continuity of $x\mapsto\tsP{t}f(x)$ when $f\in\rmC_b(X)$ follows
directly by the continuity of $x \mapsto\sfH_t\delta_x$.
The fact that $\tsP{t}f$ is a version of $\sP t$ when $f$ is Borel and
$\mm$-integrable
is a simple consequence of the fact that $\sP t$ is self-adjoint; see
\cite{AGS11b} for details.

(iii) When $f, g\in\rmC_b(X)\cap L^2(X,\mm)$ and $\mu=g\mm$, the
identity \eqref{eqdual} reduces to the fact that
$\sP t$ is selfadjoint. The general case can be easily achieved using
a monotone class argument.

(iv) In the case of $\varphi\in
\Lip_b(X)\cap L^2(X,\mm)$ it is easy to prove that
$\tsP{t}\varphi(x)\to\varphi(x)$ for all $x\in X$ as $t\downarrow0$,
since $\tsP{t}\varphi$
are equi-Lipschitz, converge in $L^2(X,\mm)$ to $\varphi$ and $\supp
\mm=X$.
By \eqref{eqdual}, it follows that
\begin{eqnarray}
\lim_{t\downarrow0}\int_X \varphi\,\d
\sfH_t\mu= \lim_{t\downarrow0}\int_X
\tsP{t} \varphi\,\d\mu= \int_X \varphi\,\d\mu \nonumber
\\
\eqntext{\mbox{for every }\varphi\in\Lip_b(X)\cap L^2(X,\mm).}
\end{eqnarray}
By a density argument, we obtain that the same holds on $\Lip_b(X)$,
so that $t\mapsto\sH{t}\mu$
is weakly continuous. Since $\tsP{t} f(x)=\int_X f\,\d\sfH_t\delta
_x$, we conclude that
$\tsP{t} f(x)\to f(x)$ for arbitrary $f\in\rmC_b(X)$.
\end{pf}

Writing $\mu=\int_X\delta_x\,\d\mu(x)$ and recalling the
definition of $\tsP{t}$, we can also write
\eqref{eqdual} in the form
%
\begin{equation}
\label{eqollo} \sfH_t\mu=\int_X
\sfH_t\delta_x\,\d\mu(x)\qquad\forall\mu\in {\mathscr
P}(X).
\end{equation}
In order to prove that \eqref{eqBEpointwise} yields the contractivity
property
\renewcommand{\theequation}{$W_2$-cont}%
\begin{equation}
\label{eqKWcontraction} W_2(\sfH_t\mu,\sfH_t\nu)\le
C(t) W_2(\mu,\nu)\qquad\mbox{for every } \mu, \nu\in{\mathscr P}(X),
t\ge0,\hspace*{-40pt}
\end{equation}
\renewcommand{\theequation}{\arabic{section}.\arabic{equation}}%
\setcounter{equation}{22}%
we need the following auxiliary results.

%
\begin{lemma}\label{leuscint}
Assume that $(\mu_n)\subset\prob X$ weakly converges to $\mu\in
\prob X$, and that $f_n$ are equibounded Borel
functions satisfying
\[
\limsup_{n\to\infty}f_n(x_n)\leq f(x)\qquad
\mbox{whenever $x_n\to x$}
\]
for some Borel function $f$. Then $\limsup_n\int_Xf_n\,\d\mu_n\leq
\int_X f\,\d\mu$.
\end{lemma}

\begin{pf} Possibly adding a constant, we can assume that all
functions $f_n$ are nonnegative.
For all integers $k$ and $t>0$, it holds
\[
\mu \Biggl(\overline{\bigcup_{m=k}^\infty
\{f_m>t\}} \Biggr)\geq \limsup_{n\to\infty}
\mu_n \Biggl(\overline{\bigcup_{m=k}^\infty
\{f_m>t\}} \Biggr)\geq \limsup_{n\to\infty}
\mu_n\bigl(\{f_n>t\}\bigr).
\]
Taking the intersection of the sets in the left-hand side and noticing that it is contained, by assumption, in
$\{f\geq t\}$, we get $\limsup_n\mu_n(\{f_n>t\})\leq\mu(\{f\geq t\}
)$. By Cavalieri's formula and Fatou's lemma,
we conclude.
\end{pf}

%
%
\begin{lemma}\label{lebasicKW}
Assume \eqref{eqconvpuntLip1}, \eqref{eqBEpointwise} and the length
property \eqref{eq21}.
For all $f\in\operatorname{Lip}_b(X)$ nonnegative and with bounded support,
$Q_t f$ is Lipschitz, nonnegative with bounded support and it holds
\[
\bigl\llvert \sP t Q_1f(x)-\sP t f(y)\bigr\rrvert \leq
\tfrac{1}2 C^2(t){\mathsf d}^2(x,y)\qquad\mbox{for
every }t\geq0, x, y\in X. 
\]
\end{lemma}

\begin{pf} It is immediate to check that
$Q_s f(x)=0$ if
$f(x)=0$, so that
the support
of all functions $Q_sf$, $s\in[0,1]$, are contained in a given ball
and $Q_s f$ are also equi-bounded.

The stated inequality is trivial for $t=0$, so assume $t>0$.
By \eqref{eqidentities}, for every $s>0$, setting $r_k:=s-1/k\uparrow
s$, the
sequence $r_k^2\llvert \rmD Q_{r_k}f\rrvert ^2(x)$ monotonically
converges to the function $\sfD^-(x,s)$;
we can thus pass to the limit in the upper gradient inequality
[which is a consequence of \eqref{eqBEpointwise}]
\[
\bigl\llvert \sP tQ_{r_k}f(\gamma_1)-\sP
tQ_{r_k}f(\gamma_0)\bigr\rrvert \leq C(t)\int
_0^1\sqrt{\tsP{t}\llvert \rmD Q_{r_k}f
\rrvert ^2(\gamma_s)} \llvert \dot\gamma_s
\rrvert \,\d s
\]
to get that
the function $C(t)G_s$, with $G_s(x):=s^{-1}\sqrt{\tsP{t}  (\sfD
^-(x,s)^2 )}$, is an upper gradient for $\sP t(Q_s f)$.
Moreover, combining \eqref{eq13} and \eqref{eq11}, we obtain
%
\begin{eqnarray}
\label{eqveryfine} %
\limsup_{h\downarrow0}\frac{Q_{s+h}f(x_h)-Q_sf(x_h)}{h}&
\leq& \frac{1}{2s^2}\limsup_{h\downarrow0}-\bigl(
\sfD^+(x_h,s)\bigr)^2
\nonumber
\\[-8pt]
\\[-8pt]
&\leq& -\frac{1}{2s^2}\bigl(\sfD^-(x,s)\bigr)^2 %
\nonumber
\end{eqnarray}
along an arbitrary sequence $x_h\to x$.

Let $\gamma$ be a Lipschitz curve with $\gamma_1=x$ and $\gamma
_0=y$. We interpolate with a parameter $s\in[0,1]$,
setting $g(s):=\sP tQ_sf(\gamma_s)$. Using \eqref{eqaprioriQt} and
\eqref{eqconvpuntLip1}, we obtain that $g$ is absolutely continuous
in $[0,1]$, so that we need only to estimate $g'(s)$. For $h>0$, we write
\[
\frac{g(s+h)-g(s)}{h}=\int_X\frac{Q_{s+h}f-Q_s f}{h} \,\d\sH {t}
\delta_{\gamma_{s+h}}+ \frac{\sP tQ_sf(\gamma_{s+h})-\sP tQ_sf(\gamma_s)}{h}
\]
and estimate the two terms separately. The first term can be estimated
as follows:
%
\begin{eqnarray}
\label{eqKuwadacont1} \limsup_{h\downarrow0}\int_X
\frac{Q_{s+h}f-Q_s f}{h} \,\d\sH {t}\delta_{\gamma_{s+h}} &\leq&-\frac{1}{2s^2}\int
_X \sfD^-(\cdot,s)^2 \,\d\sH{t}
\delta_{\gamma_s}
\nonumber
\\[-8pt]
\\[-8pt]
&=&-\frac{1}2 G_s^2(\gamma_s).
\nonumber
\end{eqnarray}
Here, we applied Lemma~\ref{leuscint} with
$f_h(x)=(Q_{s+h}f(x)-Q_sf(x))/h$, $\mu_h=\sH{t}\delta_{\gamma_{s+h}}$
and $\mu=\sH{t}\delta_{\gamma_s}$, taking \eqref{eqveryfine} into account.

The second term can be estimated as follows.
By the upper gradient property of $C(t)G_s$ for $\sP t(Q_s f)$, we get
%
\begin{equation}
\label{eqKuwadacont2} \limsup_{h\downarrow0} \frac{\llvert \sP tQ_sf(\gamma_{s+h})-\sP tQ_sf(\gamma_s)\rrvert }{h}\leq
G_s(\gamma_s)C(t)\llvert \dot\gamma_s
\rrvert
\end{equation}
for a.e. $s\in(0,1)$, more precisely at any Lebesgue point of $|\dot
\gamma|$ and of $s\mapsto G_s(\gamma_s)$.
Combining \eqref{eqKuwadacont1} and \eqref{eqKuwadacont2} and using
the Young inequality, we get
$\llvert \sP t Q_1f(x)-\sP t f(y)\rrvert \leq C^2(t)\frac{1}2\int_0^1|\dot\gamma
_s|^2\,\d s$.
Minimizing with respect to $\gamma$ gives the result.
\end{pf}

%
\begin{theorem} \label{thmKuwadaequivalence}
Let $\cE$ and $(\sP t)_{t\ge0}$ be as in \eqref{eq50} and
\eqref{eq4}, and let ${\mathsf d}$ be a distance on $X$ under the
assumptions (\hyperref[MD]{\MD}) and \eqref{eq21}.
Then \eqref{eqconvpuntLip1} and
\eqref{eqBEpointwise} are satisfied by $(\sP t)_{t\ge0}$
if and only if \eqref{eqKWcontraction} holds.
\end{theorem}

\begin{pf}
We only prove the $W_2$ contraction assuming that \eqref
{eqconvpuntLip1} and
\eqref{eqBEpointwise} hold, since the converse implication
have been already proved in \cite{Kuwada10}
(see also \cite{AGS11b}, Theorem~6.2) and it does not play any role in
this paper.

We first notice that Kantorovich duality provides the identity
\[
\tfrac12 W_2^2(\sfH_t\delta_x,
\sfH_t\delta_y)=\sup\bigl\llvert \sP t
Q_1f(x)-\sP t f(y)\bigr\rrvert,
\]
where the supremum runs in the class of bounded, nonnegative Lipschitz
functions $f$ with bounded support.
Therefore, Lemma~\ref{lebasicKW} gives
%
\begin{equation}
\label{eqKuwadacont3} W_2^2(\sfH_t
\delta_x,\sfH_t\delta_y)\leq
C^2(t){\mathsf d}^2(x,y).
\end{equation}
Now, given $\mu, \nu\in
\prob{X}$ with $W_2(\mu,\nu)<\infty$ and
a corresponding optimal plan $\ggamma$,
we may use a measurable selection theorem (see, e.g., \cite
{Bogachev07}, Theorem~6.9.2) to select in a $\ggamma$-measurable way
optimal plans
$\ggamma_{xy}$ from $\sfH_t\delta_x$ to $\sfH_t\delta_y$. Then we define
\[
\ggamma_0:=\int_{X\times X}\ggamma_{xy}\,\d
\ggamma(x,y),
\]
and notice that, because of \eqref{eqollo}, $\ggamma_0$ is an
admissible plan from $\sfH_t\mu$ to $\sfH_t\nu$.
Since \eqref{eqKuwadacont3} provides the inequality $\int{\mathsf
d}^2
\d
\ggamma_0\leq C^2(t)\int{\mathsf d}^2\,\d\ggamma$, we conclude.
\end{pf}

\subsection{Energy measure spaces}
\label{subsecEMMspace}
In this section, we want to study more carefully the
interaction between the energy and the metric structures, particularly
in the case when
the initial structure is not provided by a distance, but rather by a
Dirichlet form $\cE$.

Given a Dirichlet form $\cE$ in $L^2(X,\mm)$ {as in
\eqref{eq50}}, assume that $\cB$ is the $\mm$-completion of
the Borel $\sigma$-algebra of $(X,\tau)$, where $\tau$ is a given
topology in $X$. Then, under these
structural assumptions, we define a first set of ``locally
\mbox{1-Lipschitz}'' functions as follows:
\[
\cL:= \bigl\{\psi\in\cG\dvtx \Gamma(\psi)\le1 \mm\mbox{-a.e. in }X
\bigr\},\qquad\cL_\rmC:=\cL\cap\rmC(X).
\]
With this notion at hand, we can generate canonically the
intrinsic (possibly infinite) pseudo-distance
\cite{Biroli-Mosco95}:
\[
{\mathsf d}_\cE(x_1,x_2):=\sup
_{\psi\in\cL_\rmC} \bigl\llvert \psi(x_2)-\psi(x_1)
\bigr\rrvert \qquad \mbox{for every }x_1,x_2\in X.
\]
We also introduce {1-Lipschitz} truncation functions
$\rmS_k\in\rmC^1(\R)$, $k>0$, defined by
%
\begin{equation}
\label{eq10tris} \rmS_k(r):=k\rmS(r/k) \qquad\mbox{with } \rmS(r)=
\cases{\displaystyle 1,&\quad if $\llvert r\rrvert \le1$,
\cr
\displaystyle 0,&\quad if $\llvert r\rrvert \ge3$, } %
\qquad\bigl
\llvert \rmS'(r)\bigr\rrvert \le1.\hspace*{-30pt}
\end{equation}
We have now all the ingredients to define the following structure.

%
\begin{definition}[(Energy measure space)]\label{defDirichlet}
Let $(X,\tau)$ be a Polish space, let $\mm$ be a Borel measure
with full support, let $\cB$ be the $\mm$-completion of the Borel
$\sigma$-algebra and
let $\cE$ be a Dirichlet form in $L^2(X,\mm)$ satisfying \eqref
{eq50} of Section~\ref{subsecDiFo}.
We say that $(X,\tau,\mm,\cE)$ is a Energy measure space if:
\begin{longlist}[(b)]
\item[(a)] There exists a function
%
\begin{equation}
\mbox{$\theta\in\rmC(X), \theta\ge0$, such that $\theta_k:=
\rmS_k\circ\theta$ belongs to $\cL$ for every $k>0$.}\hspace*{-25pt}
\label{eq28}
\end{equation}
\item[(b)] ${\mathsf d}_\cE$ is a finite distance
in $X$ which induces the topology $\tau$ and $(X,{\mathsf d}_\cE)$ is
complete.
\end{longlist}
\end{definition}

Notice that if $\mm(X)<\infty$ and $1\in D(\cE)$ then
(a) is always satisfied by choosing $\theta\equiv0$.
In the general case, condition (a) is strictly related
to the finiteness property of the measure of balls (\ref{MDb}).
In fact, we shall see in Theorem~\ref{thmEdmsp-equivalence}
that $(X,{\mathsf d}_\cE,\mm)$
satisfies the measure distance condition
(\hyperref[MD]{\MD}).

%
\begin{remark}[(Completeness and length property)]
\label{remS1}
Whenever ${\mathsf d}_\cE$ induces the topology $\tau$ [and thus
$(X,{\mathsf d}_\cE)$ is a separable space], completeness is not
a restrictive assumption, since it can always be obtained by taking the
abstract completion $\bar X$ of $X$ with respect to ${\mathsf d}_\cE$.
Since $(X,\tau)$ is a Polish space, $X$
can be identified with a Borel subset of $\bar X$
(\cite{Bogachev07}, Theorem~6.8.6),
and $\mm$ can be easily extended to a Borel measure $\bar\mm$
on $\bar X$ by setting $\bar\mm(B):=\mm(B\cap X)$; in particular
$\bar X\setminus X$ is $\bar\mm$-negligible and $\cE$ can be
considered as a Dirichlet form on $L^2(\bar X,\bar\mm)$ as
well. Finally, once completeness is assumed, the length property
is a consequence of the definition of the intrinsic distance
${\mathsf d}_\cE$; see \cite{Sturm98,Stollmann10} in the locally
compact case and the next Corollary~\ref{corlength} in the
general case.
\end{remark}

In many cases, $\tau$ is already induced by a
distance ${\mathsf d}$ satisfying the compatibility condition (\hyperref[MD]{\MD}),
so that we are actually dealing with a structure
$(X,\mathsf d,\mm,\cE)$. In this situation, it is natural to
investigate under which
assumptions the identity ${\mathsf d}={\mathsf d}_\cE$ holds:
this in particular guarantees that $(X,\tau,\mm,\cE)$
is an Energy measure space according to Definition~\ref
{defDirichlet}. In the following remark, we examine the case when $\cE
$ is canonically
generated starting from ${\mathsf d}$ and $\mm$, and then we investigate
possibly more general situations.

%
\begin{remark}[(The case of a quadratic Cheeger energy)]
\label{remQCh-vs-Dirichlet}
Let $(X,{\mathsf d},\mm)$ be a metric measure space
satisfying (\hyperref[MD]{\MD}) and let us
assume that the Cheeger energy is quadratic (i.e., $(X,{\mathsf d},\mm
)$ is
infinitesimally
Hilbertian according to \cite{Gigli12}), $\cE_\C:=2\C$.
Then it is clear that any $1$-Lipschitz function $f\in L^2(X,\mm)$
belongs to $\cL_\rmC$, hence
${\mathsf d}_\cE\geq{\mathsf d}$. It follows that ${\mathsf
d}={\mathsf d}_\cE$ if
and only if
every continuous function $f\in D(\C)$ with $\llvert \rmD f\rrvert _w\le1$ {$\mm
$-a.e. in $X$}
is $1$-Lipschitz w.r.t. ${\mathsf d}$. In particular, this is the case of
$\RCD K\infty$ spaces.

If $X=[0,1]$ endowed with the Lebesgue measure and the Euclidean
distance, and if
$\mm=\sum_n 2^{-n}\delta_{q_n}$, where $(q_n)$ is an enumeration of
$\Q\cap[0,1]$, then it is easy to check that
$\C\equiv0$ (see \cite{AGS11a} for details), hence ${\mathsf d}_\cE
(x,y)=\infty$ whenever $x\neq y$.
\end{remark}

If ${\mathsf d}$ is a distance on $X$ satisfying (\hyperref[MD]{\MD}),
in order to provide links between the Dirichlet form
$\cE$ and the distance ${\mathsf d}$ we introduce
the following condition.

{
\renewcommand\thelonglist{\ED.\alph{longlist}}
\renewcommand\labellonglist{(\thelonglist)}
%
\begin{condition*}[(ED: Energy-Distance interaction)]\label{ED}
${\mathsf d}$ is a distance on $X$ such that:
\begin{longlist}[(\ED.b)]
\item\label{EDa} every function $\psi\in\cL_\rmC$
is $1$-Lipschitz with respect to ${\mathsf d}$;
\item\label{EDb} every function {$\psi\in\Lip(X,{\mathsf d})$ with
$\llvert \rmD\psi\rrvert \leq1$ and bounded support}
belongs to $\cL_\rmC$.
\end{longlist}
\end{condition*}
}

%
\begin{theorem}
\label{thmEdmsp-equivalence}
If $(X,\tau,\mm,\cE)$ is an Energy measure space
according to Definition~\ref{defDirichlet} then
the canonical distance ${\mathsf d}_\cE$ satisfies conditions
(\hyperref[MD]{\MD}), (\hyperref[ED]{\ED}).

Conversely, if $\cE$ is a Dirichlet form as in \eqref{eq50} and
${\mathsf d}$ is a distance on $X\times X$ inducing the topology $\tau
$ and
satisfying
conditions
(\hyperref[MD]{\MD}), (\hyperref[ED]{\ED}), then
$(X,\tau,\mm,\cE)$ is an Energy measure space and
%
\begin{equation}
\label{eq9} {\mathsf d}(x_1,x_2)={\mathsf
d}_\cE(x_1,x_2)\qquad \mbox{for every
}x_1, x_2\in X.
\end{equation}
\end{theorem}

\begin{pf}
Let us first assume that $(X,\tau,\mm,\cE)$ is an Energy measure
space.
(\ref{MDa}) is immediate since ${\mathsf d}_\cE$ is complete by assumption
and $\tau$ is separable.
(\ref{EDa}) is also a direct consequence of the definition of $\cE$,
since
%
\begin{equation}
\label{eq22} \bigl\llvert \psi(x_2)-\psi(x_1)\bigr
\rrvert \le{\mathsf d}_\cE(x_1,x_2)\qquad
\mbox{for every }\psi \in \cL_\rmC, x_1,x_2\in
X.
\end{equation}
Let us now prove (\ref{MDb}) and (\ref{EDb}).

We first observe that the function $\theta$ of \eqref{eq28} is
bounded on each ball $B_r(y)$, $y\in X$ and $r>0$, otherwise
we could find a sequence of points $y_k\in B_r(y)$, $k\in\N$, such that
$\theta(y_k)\ge3k$ and, therefore, $\theta_k(y)-\theta_k(y_k)\ge k$
whenever $\theta(y)\le k$. This contradicts the fact that $\theta_k$
is $1$-Lipschitz
by \eqref{eq22}.
As a consequence, for every $y\in X$ and $r>0$ there
exists $k_{y,r}\in\N$ such that
\[
\theta_k(x)\equiv k\qquad\mbox{for every }x\in
B_r(y), k\ge k_{r,y}.
\]
In particular, since $\theta_k\in L^2(X,\mm)$,
we get that all the sets $B_r(y)$ with $y\in X$ and $r>0$
have finite measure, so that (\ref{MDb}) holds.

We observe that by the separability of $X\times X$
we can find a countable family
$(\psi_n)\subset\cL_\rmC$ such that
\[
{\mathsf d}_\cE(x_1,x_2)=\sup
_n \bigl\llvert \psi_n(x_2)-
\psi_n(x_1)\bigr\rrvert \qquad\mbox{for every
}x_1,x_2\in X.
\]
We set
\[
{\mathsf d}_{k,N}(x_1,x_2):= \Bigl(\sup
_{n\le N} \bigl\llvert \psi_n(x_2)-
\psi_n(x_1)\bigr\rrvert \Bigr)\land \theta_k(x_2),
\]
observing that for every $y\in X$ the map
$x\mapsto{\mathsf d}_{k,N}(y,x)$ belongs to $\cL_\rmC$.
Passing to the limit as $N\to\infty$, it is easy to check that
${\mathsf d}_{k,N}(y,\cdot)\to{\mathsf d}_k(y,\cdot)={\mathsf d}_\cE(y,\cdot)\land
\theta_k$ pointwise in
$X$ and, therefore, in $L^2(X,\mm)$,
since $\theta_k\in L^2(X,\mm)$. We deduce that
${\mathsf d}_k(y,\cdot)\in\cL$ for every $y\in X$ and $k\in\N$.

Let us now prove that every map
$f\in\Lip(X,{\mathsf d}_\cE)$ with $\llvert \rmD f\rrvert \leq1$ and bounded support
belongs to $\cL$; it is not restrictive to assume $f$ nonnegative.
Since $(X,\mathsf{d}_\cE)$ is a length metric space (see Theorem 3.10
    below), $f$ is $1$-Lipschitz and it is easy to check that, setting
$f_k=f\land\theta_k$, it holds
\begin{eqnarray*}
f_k(x)&=& \Bigl(\inf_{z\in X} \bigl(f(z)+{{\mathsf
d}_\cE}(z,x) \bigr) \Bigr)\land \theta_k(x)= \inf
_{z\in X} \bigl( \bigl(f_k(z)+{{\mathsf
d}_\cE}(z,x) \bigr)\land \theta_k(x) \bigr)
\\
&=& \Bigl( \inf_{z\in X} \bigl(f_k(z)+{{\mathsf
d}_\cE}(z,x)\land \theta_k(x) \bigr) \Bigr)\land
\theta_k(x)
\\
&=& \Bigl( \inf_{z\in X} \bigl(f_k(z)+{\mathsf
d}_k(z,x) \bigr) \Bigr)\land \theta_k(x).
\end{eqnarray*}
Let $(z_i)$ be a countable dense set of $X$. The functions
\[
f_{k,n}(x):= \Bigl(\min_{1\le i\le n} f_k(z_i)+{
\mathsf d}_k(z_i,x) \Bigr)\land\theta_k(x)
\]
belong to $\cL$, are nonincreasing with respect to $n$,
and satisfy $0\le f_{k,n}\le\theta_k$.
Since
\[
\bigl\{x\in X\dvtx \theta_k(x)>0\bigr\}\subset\bigl\{x\in X\dvtx
\theta_{3k}(x)=3k\bigr\},
\]
we easily see that $\mm (\supp(\theta_k) )<\infty$.
Passing to the limit as
$n\up\infty$, since $z\mapsto{\mathsf d}_{k}(z,x)$ is continuous, they
converge monotonically to
\[
\Bigl(\inf_{z\in X} \bigl(f_k(z)+{\mathsf
d}_{k}(z,x) \bigr) \Bigr)\land \theta_k(x)=f_k(x),
\]
and their energy is uniformly bounded by
$\mm (\supp(\theta_k) )$.
This shows that $f_k\in\cL$. Eventually, letting $k\up\infty$ and
recalling that $\supp(f_k)\subset\supp(f)$
and $\mm(\supp(f))<\infty$ by
(\ref{MDa}), we obtain $f\in\cL$.

The converse implication is easier: it is immediate to check that
(\ref{EDa}) is equivalent to
%
\begin{equation}
\label{eq30} {\mathsf d}(x_1,x_2)\ge{\mathsf
d}_\cE(x_1,x_2) \qquad\mbox{for every
}x_1,x_2\in X;
\end{equation}
if (\ref{EDb}) holds and balls have finite measure according to
(\ref{MDb}), we have $x\mapsto\rmT_k({\mathsf d}(y,x))\in\cL$ for
every $y\in
X$, where $\rmT_k(r):=r\land\rmS_k(r)$. Since
\[
{\mathsf d}(x_1,x_2)=\rmT_k\bigl({\mathsf
d}(x_2,x_1)\bigr)-\rmT_k\bigl({\mathsf
d}(x_2,x_2)\bigr)\qquad \mbox{whenever } k>{\mathsf
d}(x_1,x_2),
\]
we easily get the converse inequality to \eqref{eq30} and, therefore,
\eqref{eq9} and property (b) of Definition \eqref{defDirichlet}.
In order to get also (a), it is sufficient to take $\theta
(x):={\mathsf d}
(x,x_0)$ for an arbitrary
$x_0\in X$.
\end{pf}

%
%
\begin{theorem}[(Length property of ${\mathsf d}_\cE$)]
\label{corlength}
If $(X,\tau,\mm,\cE)$ is an Energy measure space, then
$(X,{\mathsf d}_\cE)$ is a length metric space, that is, it also satisfies
\eqref{eq21}.
\end{theorem}

\begin{pf}
We follow the same argument as in
\cite{Sturm98,Stollmann10}.
Since $(X,\mathsf d_{\mathcal E})$ is complete, it is well known
(see, e.g., \cite{Burago-Burago-Ivanov01}, Theorem~2.4.16)
that the length condition is equivalent to
show that for every couple of points $x_0, x_1\in X$ and $\eps\in(0,r)$
with $r:={\mathsf d}_\cE(x_0,x_1)$ there exists
an $\eps$-midpoint $y\in X$ such that
\[
{\mathsf d}_\cE(y,x_i)< \frac{r}2 +\eps,\qquad
i=0, 1.
\]
We argue by contradiction assuming that $B_{r/2+\eps}(x_0)\cap
B_{r/2+\eps}(x_1)=
\varnothing$ for some $\eps\in(0,r)$ and we introduce the function
\[
\psi(x):= \bigl(\tfrac{1}2(r+\eps)-{\mathsf d}_\cE(x,x_0)
\bigr)_+- \bigl(\tfrac{1}2(r+\eps)-{\mathsf d}_\cE(x,x_1)
\bigr)_+.
\]
$\psi$ is Lipschitz, has bounded support and it is easy to check
that $\Gamma(\psi)\le 1$ $\mm$-a.e. in $X$, since
$B_{r/2+\eps}(x_0)$ and $B_{r/2+\eps}(x_1)$ are disjoint.
It turns out that it is $1$-Lipschitz
by (\ref{eq22}). On the other hand, $\psi(x_0)-\psi(x_1)=r+\eps>{\mathsf
d}_\cE
(x_0,x_1)$.
\end{pf}

We now examine some additional properties of Energy measure spaces.

%
\begin{proposition}
\label{propGuppergradient}
Let $(X,\tau,\mm,\cE)$ be an Energy measure space.
Let $f\in\cG\cap\rmC_b(X)$ and let $\zeta\dvtx X\to[0,\infty)$ be
a bounded upper semicontinuous function such that
$\Gamma(f)\le\zeta^2$ $\mm$-a.e. in $X$.
Then $f$ is Lipschitz (with respect to the induced
distance ${\mathsf d}_\cE$) and
$\llvert \rmD^* f\rrvert \le\zeta$. In particular, $\zeta$ is an upper gradient
of~$f$.
\end{proposition}

\begin{pf} We know that (\hyperref[ED]{\ED}) holds {with
${\mathsf d}={\mathsf d}_{\cE}$}, by the previous theorem.
Since $\zeta$ is bounded, $f$ is Lipschitz by (\ref{EDa}).
We fix $x\in X$ and
for every $\eps>0$ we set
$G_{\eps}:=\sup_{B_{3\eps}(x)}\zeta$.
The
Lipschitz function
\[
\psi_\eps(y):= \bigl[\bigl\llvert f(y)-f(x)\bigr\rrvert
\lor \bigl(G_{\eps}{\mathsf d}(y,x) \bigr) \bigr]\land
\bigl(G_{\eps}\rmS_\eps\bigl({\mathsf d}(x,y)\bigr) \bigr)
\]
belongs to $\cV_{\infty}$ {
and $
\sqrt{\Gamma(\psi)}
\leq\max\{\zeta,G_\eps\}$, since $\llvert \rmS_\eps'\rrvert \leq1$; moreover,
$\psi_\eps(y)=0$ if ${\mathsf d}(y,x)\ge3\eps$,}
so that $
\sqrt{\Gamma(\psi_\eps)}
\le G_{\eps}$ $\mm$-a.e. in $X$. It follows
that $\psi_\eps$ is $G_\eps$-Lipschitz
and $\psi_\eps(y)\le G_{\eps} {\mathsf d}(y,x)$ for every $y\in X$ since
$\psi_\eps(x)=0$, so that
{$G_{\eps}\rmS_\eps({\mathsf d}(x,y))=G_\eps\eps$ for ${\mathsf
d}(x,y)<\eps$ gives}
\[
\bigl\llvert \rmD f(x)\bigr\rrvert \le\limsup_{y\to x}
\frac{\llvert f(y)-f(x)\rrvert }{{\mathsf d}(y,x)} \le \limsup_{y\to x} \frac{\psi_\eps(y)}{{\mathsf d}(y,x)} \le
G_{\eps}.
\]
Since $\eps>0$ is arbitrary and $\lim_{\eps\down0}G_\eps=\zeta(x)$
we obtain $\llvert \rmD f(x)\rrvert \le\zeta(x)$.
Since $\zeta$ is upper semicontinuous and $X$ is a length space,
we also get $\llvert \rmD^* f\rrvert \le\zeta$.
\end{pf}

The following result provides a first inequality between $\cE$ and $\C
$, in the case when
a priori the distances ${\mathsf d}$ and ${\mathsf d}_\cE$ are
different, and we
assume only (\ref{EDb}).

%
%
\begin{theorem}
\label{thmslope-bound}
$\cE$ be a Dirichlet form in $L^2(X,\mm)$ satisfying
\eqref{eq50} of Section~\ref{subsecDiFo}
and let ${\mathsf d}$ be a distance on $X$ satisfying condition (\hyperref[MD]{\MD}).
Then condition (\ref{EDb}) is satisfied if and only if
for every function $f\in\Lip(X,\mathsf{d})$ with bounded support
we have
%
\begin{equation}
\label{eq34} f\in\cG,\qquad\llvert \rmD f\rrvert ^2\ge\Gamma(f)
\qquad\mbox{$\mm$-a.e. in $X$.}
\end{equation}
In addition, if (\hyperref[MD]{\MD}) and (\ref{EDb}) hold, we have
%
\begin{eqnarray}
\label{eq112} &\displaystyle2\C(g)\ge\cE(g)\qquad\mbox{for every }g\in
L^2(X,\mm ),&
\\
\label{eq112bis} &\displaystyle D(\C)\subset \cG\subset\cV,\qquad\llvert \rmD g
\rrvert _w^2\ge\Gamma(g)\qquad\mbox{for every }g\in D(
\C).&
\end{eqnarray}
\end{theorem}

\begin{pf}
The implication \eqref{eq34}~$\Rightarrow$ (\ref{EDb}) is trivial;
let us consider the converse one.

Since the statement is local, possibly replacing $f$ with
$ 0\lor
(f+c)\land\rmS_k({\mathsf d}(x_0,\cdot))$ with $c=0\lor\sup_{B_{3k}(x_0)}(-f)$
(notice that $f$ is bounded) and $k$ large enough, we can assume that
$f$ is nonnegative and $f\leq\rmS_k(d(x_0,\cdot))$.

Recall the Hopf--Lax formula \eqref{eq11} for the map
$Q_t$; if $(z_i)$ is a countable dense subset of $X$ we define
%
\begin{eqnarray}
\label{eq37} Q_t^n f(x)&=& \biggl(\min
_{1\le i\le n} f(z_i)+\frac{1}{2t} {\mathsf
d}^2(z_i,x) \biggr),
\nonumber
\\[-8pt]
\\[-8pt]
Q_t^{n,k} f(x)&=&Q_t^nf(x)\land{2}
\rmS_k\bigl({\mathsf d}(x_0,x)\bigr),
\nonumber
\end{eqnarray}
and we set $I_n(x)= \{i\in\{1,\ldots,n\}\dvtx  z_i\mbox{ minimizes
\eqref{eq37}} \}$. By the density of $(z_n)$, it is clear
that $Q_t^{n,k} f\downarrow Q_t f\land2\rmS_k({\mathsf d}(x_0,\cdot))=Q_t
f$, because $Q_tf\leq f\leq\rmS_k({\mathsf d}(x_0,\cdot))$.
Therefore, if $z_n(x)\in \{z_i\dvtx i\in I_n(x)\}$, it turns out that $(z_n(x))$ is a
minimizing sequence for $Q_tf(x)$, namely
\[
\frac{1}{2t}{\mathsf d}^2\bigl(x,z_n(x)\bigr)+f
\bigl(z_n(x)\bigr)\to Q_t f(x)\qquad\mbox{as }n\to\infty.
\]
The very definition of $\sfD^{+}(t,x)$ then gives
%
\begin{equation}
\label{eqaop1} \limsup_{n\to\infty}\frac{1}t{\mathsf d}
\bigl(x,z_n(x)\bigr)\le\sfD^{+}(t,x).
\end{equation}
The locality property, the fact that $ (f(z_i)+{\mathsf
d}^2(z_i,\cdot
)/2t )\land
\rmS_k({\mathsf d}(x_0,\cdot))\in\cV$
and the obvious bound
\[
{\mathsf d}^2(z_i,x)\le4 k t\qquad\mbox{if } i\in
I_n(x), x\in\bigl\{Q^{n}_t f\leq2
\rmS_k\bigl(d(x_0,\cdot)\bigr)\bigr\},
\]
yield
\[
\sqrt{\Gamma\bigl(Q_t^{n,k} f\bigr) (x)}\le
\frac{1}t \max_{i\in I_n(x)}{\mathsf d}(x,z_i)
\le2\sqrt{\frac{k}t} %
\]
for $\mm$-a.e. {$x\in\{Q^{n}_t f\leq2\rmS_k(d(x_0,\cdot))$\}}.
If we define $z_n(x)$ as a value $z_j$ that realizes the maximum
for {${\mathsf d}(z_i,x)$} as $i\in I_n(x)$, the previous formula yields
%
\begin{eqnarray}
\label{eq24b} \sqrt{\Gamma\bigl(Q_t^{n,k} f\bigr)
(x)}\leq\frac{1}t {\mathsf d}\bigl(x,z_n(x)\bigr) \le2
\sqrt{\frac{k}t}
\nonumber
\\[-8pt]
\\[-8pt]
\eqntext{\mbox{for }\mm\mbox{-a.e. }x\in\bigl\{Q^n_t f\leq2
\rmS_k\bigl(d(x_0,\cdot)\bigr)\bigr\},}
\end{eqnarray}
so that
\begin{eqnarray*}
\Gamma\bigl(Q_t^{n,k} f\bigr)&\leq& 4 \biggl(1
\lor\frac{k}t \biggr) \qquad \mbox{$\mm$-a.e. in $X$},
\\
\Gamma\bigl(Q_t^{n,k} f\bigr)&=&0\qquad \mbox{$\mm$-a.e. in
$X\setminus B_{3k}(x_0)$}.
\end{eqnarray*}
Since $Q_t^{n,k} f$
and $\Gamma(Q_t^{n,k} f)$ are
uniformly bounded and supported in a bounded set,
and $Q_t^{n,k} f$
pointwise converges to
$Q_t f$ as $n\to\infty$, considering any weak limit point $G$ of
$\sqrt{\Gamma(Q_t^{n,k} f)}$ in $L^2(X,\mm)$ we obtain
by \eqref{eq20}, \eqref{eqaop1} and \eqref{eq24b}
%
\begin{eqnarray}
\label{eqaop2} \Gamma(Q_t f) (x)\le G^2(x)\le
\frac{ (\sfD^{+}(x,t) )^2}{t^2}
\nonumber
\\[-8pt]
\\[-8pt]
\eqntext{\mbox{for $\mm$-a.e. $x\in\bigl\{Q_tf<2\rmS_k
\bigl({\mathsf d}(x_0,\cdot )\bigr)\bigr\}$.}}
\end{eqnarray}
{Since, by the inequalities
$Q_t f\leq f\leq\rmS_k({\mathsf d}(x_0,\cdot))$, $Q_t f$
vanishes on $\{Q_tf\geq2\rmS_k({\mathsf d}(x_0,\cdot))\}$,
the inequality above holds $\mm$-a.e. in $X$.}

Since $f$ is Lipschitz, it follows that $\sfD^+(x,t)/t$ is uniformly
bounded. Integrating \eqref{eq25} on an arbitrary bounded Borel set
$A$ and
applying Fatou's lemma, from \eqref{eqaop2} we get
\begin{eqnarray*}
\int_A \llvert \rmD f\rrvert ^2\,\d\mm&\ge&
\int_A \limsup_{t\down0} \int
_0^1 \biggl(\frac{\sfD^+(x,tr)}{tr}
\biggr)^2\,\d r\,\d\mm(x)
\\
&\ge& \limsup_{t\downarrow0} \int_0^1
\int_A \biggl(\frac{\sfD^+(x,tr)}{tr} \biggr)^2 \,\d
\mm(x)\,\d r
\\
&\ge& \limsup_{t\downarrow0} \int_0^1
\int_A \Gamma(Q_{tr} f) (x) \,\d\mm(x)\,\d r
\\
&\ge& \int_0^1 \liminf_{t\downarrow0}
\biggl(\int_A \Gamma(Q_{tr} f)\,\d\mm \biggr)\,
\d r \ge\int_A \Gamma(f)\,\d\mm,
\end{eqnarray*}
where in the last inequality we applied
\eqref{eq20} once more. Since $A$ is arbitrary we conclude.

In order to prove \eqref{eq112},
it is not restrictive to assume $\C(g)<\infty$. By the very
definition of the Cheeger energy, we can then find
a sequence of functions $f_n\in\Lip_b(X)\cap L^2(X,\mm)$ converging
to $g$ in
$L^2(X,\mm)$ with $\lim_{n\to\infty}\int_X |\rmD
f_n|^2\,\d\mm=2\C(g)$.

By replacing $f_n$ with $g_n (x)=f_n(x)\rmS_1({\mathsf d}(x,x_0)/n)$,
we can even obtain a sequence of Lipschitz functions with bounded
support. \eqref{eq34} and the lower semicontinuity of $\cE$ then provide
\eqref{eq112} and \eqref{eq112bis}.
\end{pf}

In order to conclude our analysis of the relations
between $\cE$ and $\C$ for Energy measure spaces $(X,\tau,\mm,\cE)$,
we introduce a further property.

%
\begin{definition}[(Upper regularity)]
\label{defcE-reg}
Let $(X,\tau,\mm,\cE)$ be an Energy measure space.
We say that the Dirichlet form $\cE$ is upper-regular if
for every $f$ in a dense subset of $\cV$ there exist
$f_n\in\cG\cap\rmC_b(X)$
converging strongly to $f$ in $L^2(X,\mm)$
and $g_n\dvtx X\to\R$ bounded and upper semicontinuous such that
%
\begin{equation}
\label{eq113} \sqrt{\Gamma(f_n)}\le g_n\qquad  \mm
\mbox{-a.e.},\qquad 
\limsup_{n\to\infty}\int
_X g_n^2\,\d\mm\le\cE(f).
\end{equation}
\end{definition}

%
\begin{theorem}
\label{thmcE=C}
Let $(X,\tau,\mm,\cE)$ be an Energy measure space.
Then the Cheeger energy associated to $(X,{\mathsf d}_\cE,\mm)$
coincides with
$\cE$, that is,
%
\begin{equation}
\label{eq39} \cE(f)=2\C(f)\qquad\mbox{for every }f\in L^2(X,\mm),
\end{equation}
if and only if $\cE$ is upper-regular.
In this case $\cG=\cV$, $\cE$ admits a Carr\'e du Champ
$\Gamma$ and
%
\begin{equation}
\label{eq26} \Gamma(f)=\llvert \rmD f\rrvert _w^2
\qquad\mbox{$\mm$-a.e. in $X$ for every }f\in\cV.
\end{equation}
In particular, the space $\cV\cap\Lip_b(X,\mathsf{d}_\cE)$ is dense in $\cV$.
If moreover \eqref{eq80} holds, then $(\sP t)_{t\ge0}$ satisfies
the mass preserving property \eqref{eq4}.
\end{theorem}

\begin{pf}
Since $\C$ is always upper-regular by \eqref{eq105},
the condition is clearly necessary.
In order to prove its sufficiency, by \eqref{eq112} of
Theorem~\ref{thmslope-bound} we have just to prove
that every $f\in\cV$ satisfies
the inequality $2\C(f)\le\cE(f)$.
If $f_n,g_n$ are sequences as in \eqref{eq113},
Proposition~\ref{propGuppergradient} yields that
$f_n$ are Lipschitz and
\[
\llvert \rmD f_n\rrvert \le g_n,\qquad
\C(f_n)\le\frac{1}2\int_X \llvert \rmD
f_n\rrvert ^2\,\d\mm\le \frac{1}2\int
_X g_n^2\,\d\mm.
\]
Passing to the limit as $n\to\infty$, we obtain the
desired inequality thanks to the lower-semicontinuity of $\C$ in
$L^2(X,\mm)$. The last statement of the Theorem follows
by \cite{AGS11a}, Theorem~4.20.
\end{pf}

\subsection{Riemannian Energy measure spaces and the \texorpdfstring{$\mathrm{BE}(K,\infty)$}{$BE(K,infty)$} condition}
\label{subsecDir-BE}

In this section, we will discuss
various consequences of the Energy measure space axiomatization in
combination with $\operatorname{BE}(K,\infty)$.

Taking into account the previous section, the Bakry--\'Emery condition
$\BE KN $ as stated in
Definition~\ref{defBEKN} makes perfectly sense for
a Energy measure space $(X,\tau,\mm,\cE)$.
In the next result, we will show that under a weak-Feller property
on the semigroup $(\sP t)_{t\ge0}$
we gain upper-regularity of $\cE$,
the identifications $\cE=2\C$ of Theorem~\ref{thmcE=C} and
$\cL=\cL_\rmC$.

%
\begin{theorem}
\label{thmFeller=Riemannian}
Let $(X,\tau,\mm,\cE)$ be a Energy measure space
satisfying the $\BE K\infty$ condition.
Then its Markov semigroup $(\sP t)_{t\ge0}$ satisfies the
weak-Feller property
\renewcommand{\theequation}{w-Feller}%
\begin{eqnarray}
\label{eqwFeller} &&f\in\Lip_b(X,\mathsf{d}_\cE) \mbox{ with bounded
support}, \qquad  |\rmD f|\leq1\nonumber
\\[-8pt]
\\[-8pt]
&&\qquad \Rightarrow\quad  \sP t f\in\rmC_b(X) \qquad \forall t
\geq0
\nonumber
\end{eqnarray}
\renewcommand{\theequation}{\arabic{section}.\arabic{equation}}%
\setcounter{equation}{41}%
if and only if
%
\begin{equation}
\label{eq114} \cL=\cL_\rmC,
\end{equation}
that is, if every function $f\in\cL$ admits a continuous
representative.
In this case, $\cE$ is upper-regular and, as a consequence,
\eqref{eq39}, \eqref{eq26} hold.
\end{theorem}

\begin{pf}
The implication \eqref{eq114}~$\Rightarrow$ \eqref{eqwFeller} is easy,
since {for any $f\in\Lip_b(X,{\mathsf d})$ with $\llvert \rmD f\rrvert \leq1$ and
bounded support},
the Bakry--\'Emery condition $\BE K\infty$ [i.e., \eqref{eqBEKnu.6}
with $\nu=0$] and the bound $\Gamma(f)\le1$
given by (\ref{EDb})
yields $\rme^{Kt}\sP t f\in\cL=\cL_\rmC$.

Now we prove the converse implication, from \eqref{eqwFeller} to
\eqref{eq114}.
The Bakry--\'Emery condition $\BE K\infty$ in conjunction with
(\ref{EDb}) and \eqref{eqwFeller} yield
$\rme^{Kt}\sP t f\in\cL_\rmC$ for every
$f\in\Lip_b(X,{\mathsf d})$ with $\llvert \rmD f\rrvert \leq1$ and bounded support
and $t>0$,
thus in particular $\rme^{Kt}\sP t f$ is $1$-Lipschitz by (\ref{EDa}).
Let us now fix $f\in\cL\cap L^\infty(X,\mm)$
and let us consider a sequence of uniformly bounded functions
{$f_n\in\Lip_b(X)\cap L^2(X,\mm)$
with bounded support} converging to $f$ in $L^2(X,\mm)$.
By the previous step, we know that $\sP t f_n\in\Lip_b(X)$
and the estimate \eqref{eqBEKnu.4}
shows that $\Gamma(\sP t f_n)\le C/t$ for a constant $C$ independent
of $n$. (\hyperref[ED]{\ED}) then shows that
$\Lip(\sP t f_n)\le C/t$; passing to the limit as $n\to\infty$,
we can find a subsequence $n_k\to\infty$
such that $\lim_k \sP{t} f_{n_k}(x)=\sP t f(x)$
for every $x\in X\setminus\cN$, with $\mm(\cN)=0$.
Since
$ \sP{t}f_n$ are uniformly Lipschitz functions,
also $\sP t f$ is Lipschitz in $X\setminus\cN$ so that
it admits a Lipschitz representative
$\tilde f_t$ in $X$.

On the other hand, $\BE K\infty$ and
(\ref{EDa}) show that $\Lip(f_t)\le\rme^{-Kt}$.
Passing to the limit along a suitable sequence $t_n\downarrow0$
and repeating the previous argument we obtain
that $f$ admits a Lipschitz representative.

Let us prove now that $\cE$ is upper regular, by
checking \eqref{eq113} for every $f$ in the space $\cVu$ of \eqref{eq67},
which is dense in $\cV$.
Observe that the estimate
\eqref{eqBEKnu.4}, \eqref{eq114} and (\ref{EDa}) yield
that, for every $t>0$ and every
$g\in L^2\cap L^\infty(X,\mm)$, the function
$\sP t g$ admits a Lipschitz, thus continuous, and bounded representative
$g_t$. Choosing in particular $g:=\sqrt{\Gamma(f)}$, we obtain
by \eqref{eqBEKnu.6}
\begin{eqnarray*}
\Gamma(\sP t f)&\le&\rme^{-2Kt} g_t^2,\qquad \sP
t f\to f \mbox{ in }L^2(X,\mm),
\\
\lim_{t\down0}\int_X \rme^{-2Kt}
g_t^2\,\d\mm&=& \int_X
g^2\,\d\mm=\cE(f).
\end{eqnarray*}
\upqed
\end{pf}

According to the previous theorem we introduce the
natural, and smaller, class of Energy measure spaces
$(X,\tau,\mm,\cE)$, still with no curvature bound, but well adapted
to the Bakry--\'Emery condition. In such a class, that we call
\emph{Riemannian Energy measure spaces},
the Dirichlet form $\cE$ coincides with the Cheeger energy $\C$
associated to the
intrinsic distance ${\mathsf d}_\cE$ and every function in $\cL$
admits a
continuous (thus $1$-Lipschitz, by the
Energy measure space axiomatization) representative.

%
\begin{definition}[(Riemannian Energy measure spaces)]
$(X,\tau,\mm,\cE)$ is a Riemannian Energy measure space if the
following properties hold:
\begin{longlist}[(b)]
\item[(a)] $(X,\tau,\mm,\cE)$ is a Energy measure space;
\item[(b)] $\cE$ is upper regular according to Definition~\ref{defcE-reg};
\item[(c)] every function in $\cL$ admits a continuous representative.
\end{longlist}
\end{definition}

The next theorem presents various equivalent characterizations of
Riemannian Energy measure spaces
in connection with $\BE K\infty$.

%
%
\begin{theorem}
\label{thmequivalence}
The following conditions are equivalent:%
{\renewcommand\theenumi{\roman{enumi}}\renewcommand\labelenumi{(\theenumi)}%
\begin{longlist}[(iii)]
\item\label{itemi}
$(X,\tau,\mm,\cE)$ is a Riemannian Energy measure space satisfying
$\BE K\infty$.
\item\label{itemibis}
$(X,\tau,\mm,\cE)$ is a Energy measure space satisfying
\eqref{eqwFeller} and\break $\BE K\infty$.
\item\label{itemitris}
$(X,\tau,\mm,\cE)$ is a Energy measure space satisfying
$\cL=\cL_\rmC$ and
$\BE K\infty$.
\item\label{itemii}
$(X,\tau,\mm,\cE)$ is a Energy measure space with
$\cE$ upper regular,
and for every function $f\in L^2(X,\mm)\cap\Lip_b(X,\mathsf{d}_\cE)$
with $\llvert \rmD f\rrvert \in L^2(X,\mm)$
%
\begin{equation}
\label{eq43}\quad  \sP t f\in\Lip(X,\mathsf{d}_\cE),\qquad \llvert \rmD\sP t f\rrvert
^2\le\rme^{-2Kt}\sP t \bigl( \llvert \rmD f\rrvert
^2 \bigr)\qquad \mm\mbox{-a.e. in $X$}.
\end{equation}
\item\label{itemiii}
$\cE$ is a Dirichlet form in $L^2(X,\mm)$ as in \eqref{eq50},
there exists a length distance $\mathsf d$ on $X$ inducing the topology $\tau$
and satisfying conditions (\hyperref[MD]{\MD}), (\ref{EDb}),
and for every $f\in\cL_\rmC\cap L^\infty(X)$ and $t>0$
%
\begin{equation}
\label{eq44} \qquad\sP t f\in\Lip_b(X,\mathsf{d}),\qquad \llvert \rmD\sP t f
\rrvert ^2\le\rme^{-2Kt}\sP t \Gamma(f) \qquad \mm\mbox{-a.e.
in $X$}.
\end{equation}
\end{longlist}
}%
If one of the above equivalent conditions holds
with \eqref{eq80}, then
\eqref{eqconvpuntLip1} holds {with $C(t)=\rme^{-Kt}$},
the semigroups $(\tsP{t})_{t\ge0}$ and $(\sH{t})_{t\ge0}$
are well defined according to
Proposition~\ref{propprecise-rep},
$\sH{t}(\mu)\ll\mm$ for every $t>0$ and $\mu\in\prob X$,
and the strong Feller property
\renewcommand{\theequation}{S-Feller}%
\begin{equation}
\label{eqFeller} \mbox{$\tsP{t}$ maps $L^2\cap L^\infty(X,
\mm)$ into $\Lip_b(X,\mathsf{d}_\cE)$}
\end{equation}
\renewcommand{\theequation}{\arabic{section}.\arabic{equation}}%
\setcounter{equation}{44}%
holds with
%
\begin{equation}
\label{eqG0} \llvert \rmD\tsP{t}f\rrvert ^2= \Gamma(\sP t f)
\qquad\mm\mbox{-a.e. in $X$} \mbox{ for every }t>0, f\in L^2\cap
L^\infty(X,\mm).\hspace*{-35pt}
\end{equation}
Eventually, defining $\rmI_{2K}(t)$ as in \eqref{eq57}, there holds
%
\begin{equation}
\label{eqGLip} 2\rmI_{2K}(t) \llvert \rmD\tsP{t} f\rrvert
^2\le\tsP{t} {f^2}\qquad\mbox{for every }t\in (0,\infty),
f\in L^\infty(X,\mm),\hspace*{-5pt}
\end{equation}
and in particular
%
\begin{equation}
\label{eqGLip2} \sqrt{2\rmI_{2K}(t)}\Lip(\heat t f)\le\llVert f\rrVert
_{L^\infty(X,\mm
)}\qquad \mbox{for every }t\in(0,\infty).
\end{equation}
\end{theorem}

\begin{pf}
The equivalence
(\ref{itemi})~$\Leftrightarrow$ (\ref{itemibis})~$\Leftrightarrow
$ (\ref{itemitris})
is just the statement of Theorem~\ref{thmFeller=Riemannian}.

(\ref{itemiii}) $\Rightarrow$
(\ref{itemibis}) and (\ref{itemiii})~$\Rightarrow$
(\ref{itemii}).
Thanks to (\ref{EDb}),
we immediately get \eqref{eqwFeller}. Equation \eqref{eq44} also yields
the density of $\cV_{\infty}$ in $\cV$ and, thanks to
\eqref{eq34}, condition
\eqref{eqBEKnu.6} for $\BE K\infty$. Since $(\sP t)_{t\ge0}$ is
order preserving,
\eqref{eq34} and (\ref{itemiii}) yield (\ref{EDa}): if $f\in
\cL_\rmC$ then \eqref{eq44}
and the length property yield $\sP t f$ Lipschitz with constant less then $\rme^{-K t}$.
Since along a suitable vanishing sequence
$t_n\downarrow0$ $\sP{t_n} f\to f$ $\mm$-a.e. as $n\to\infty$,
arguing as in the proof of Theorem~\ref{thmFeller=Riemannian} it is easy
to check that $f$ is $1$-Lipschitz.
We can thus apply Theorem~\ref{thmequivalence}
to get that $(X,\tau,\mm,\cE)$ is an Energy measure space
with ${\mathsf d}\equiv{\mathsf d}_\cE$.
Since (\ref{itemibis}) in particular shows
that $\cE$ is upper-regular,
we proved that (\ref{itemiii})~$\Rightarrow$ (\ref{itemii}) as well.

(\ref{itemii}) $\Rightarrow$ (\ref{itemibis}). By the density of
$\cV\cap\Lip_b(X,\mathsf{d}_\cE)$ in $D(\C)=\cV$ stated in
Theorem~\ref{thmcE=C} and the upper bound \eqref{eq34} we get
\eqref{eqBEKnu.6} which is one of the equivalent characterizations
of $\BE K\infty$. Moreover, \eqref{eq43} clearly yields \eqref{eqwFeller}.

(\ref{itemi}) $\Rightarrow$ (\ref{itemiii}) with the choice $\mathsf d:=\mathsf d_\cE$ and \eqref{eqGLip}.
Let us observe that the estimate
\eqref{eqBEKnu.4} and the property $\cL\subset\Lip(X)$ yield that
for every $t>0$ and every function $f\in L^2\cap L^\infty(X,\mm)$
$\sP t f$ admits a Lipschitz representative satisfying \eqref{eqGLip2}.
Moreover, if $f$ is also Lipschitz, \eqref{eqBEKnu.6} yields
the estimate \eqref{eqconvpuntLip1} with $C(t):=\rme^{-K t}$.
We can then apply proposition \ref{propprecise-rep}
and conclude that when $f\in\rmC_b(X)\cap L^2(X,\mm)$
the Lipschitz representative of $\sP t f$ coincides with
$\tsP{t} f$. Since by definition
\[
\tsP{t} f(x)=\int_X f\,\d\sfH_t
\delta_x\qquad\mbox{for all Borel $f$ bounded from below}
\]
we can use a monotone class argument to prove the identification of
$\tsP{t}$
with the continuous version of $\sP t$
in the general case of bounded, Borel and square integrable functions.
Notice that we use \eqref{eqGLip2} to convert monotone equi-bounded
convergence of $f_n$ into pointwise convergence on $X$ of
(the continuous representative of) $\sP t
f_n$, $t>0$.

Another immediate application of \eqref{eqGLip2}
is the absolute continuity
of {$\sH{t}\mu$}
w.r.t. $\mm$ for all $\mu\in{\mathscr P}(X)$ and $t>0$.
Indeed, if $A$ is a Borel and $\mm$-negligible set, then $\tsP{t}\chi_A$
is identically null
(being equal to $\sP t\chi_A$, hence continuous, and null $\mm$-a.e.
in $X$),
hence \eqref{eqdual} gives {$\sH{t}\mu(A)=0$}. As a consequence, we
can also compute
$\tsP{t} f(x)$ for $\mm$-measurable functions $f$,
provided $f$ is semi-integrable with respect
to $\sH{t}\delta_x$.
If now $f\in\cL$, \eqref{eqBEKnu.6} then yields
\[
\Gamma(\sP t f)\le\rme^{-2Kt}\tsP{t}\Gamma(f)\qquad\mbox{$
\mm $-a.e. in $X$},
\]
and Proposition~\ref{propGuppergradient} yields \eqref{eq44} since
$\tsP{t}\Gamma(f)$ is continuous and bounded.
A~similar argument
shows \eqref{eqGLip}, starting from \eqref{eqBEKnu.4}.

Let us eventually prove \eqref{eqG0}.
Since the inequality $\geq$ is true by assumption, let us see why
\eqref{eq44} provides the converse one:
we start from
\[
\llvert \rmD\tsP{t} f\rrvert ^2\le\rme^{-2K\eps}\tsP{\eps}
\bigl(\Gamma(\sP {t-\eps} f) \bigr) \qquad\mbox{for every }\eps\in(0,t).
\]
Recalling that $\Gamma(\heat{t-\eps} f)$ converges strongly in
$L^2(X,\mm)$ to $\Gamma(\heat{t}f)$
as $\eps\downarrow0$, we get \eqref{eqG0}.
\end{pf}

Recalling Theorem~\ref{thmequivalence},
Theorem~\ref{thmKuwadaequivalence},
the characterization
\eqref{eq43}, {and the notation \eqref{eq16}}, we immediately have

%
\begin{corollary}
\label{corBE=contraction}
Let $(X,\tau,\mm,\cE)$ be a
Energy measure space satisfying
the upper-regularity property \eqref{eq113}
(in particular a Riemannian one) and \eqref{eq80}.

Then $\BE K\infty$ holds if and only if the semigroup
$(\sP t)_{t\ge0}$ satisfies the contraction property
%
\begin{equation}
\label{eq56} W_2 \bigl((\sP t f)\mm,(\sP t g)\mm \bigr)\le
\rme^{-K t} W_2(f\mm,g\mm)
\end{equation}
{for every probability densities $f, g\in L^1_+(X,\mm)$.}
\end{corollary}

\section{Proof of the equivalence result}
\label{secproof}

In this section, we assume
that\break $(X,\tau,\mm,\allowbreak \cE)$ is a Riemannian Energy measure space
satisfying \eqref{eq80} (relative to ${\mathsf d}_\cE$) and the $\BE
K\infty$ condition, as
discussed in Section~\ref{subsecDir-BE}.
In particular, all results of the previous sections on existence of the dual
semigroup $(\sH{t})_{t\ge0}$, its $W_2$-contractivity
and regularizing properties of $(\sP t)_{t\ge0}$ are applicable.
Furthermore, by Theorem~\ref{thmcE=C}, the Dirichlet form setup
described in Section~\ref{secMarkov} and the metric
setup described in Section~\ref{subsecpreliminaries} are completely
equivalent. In particular, we can apply the results of \cite{AGS11a}.

\subsection{Entropy, Fisher information and moment estimates}
\label{subsecstep1}
Let us first recall that
the \emph{Fisher information functional}
$\mathsf F\dvtx L^1_+(X,\mm)\to[0,\infty]$ is defined by
\[
\mathsf F(f):=\otto\cE(\sqrt{f}), \qquad\sqrt{f}\in\cV,
\]
set equal to $+\infty$ if $\sqrt{f}\in L^2(X,\mm)\setminus\cV$.
Since $f_n\to f$ in $L^1_+(X,\mm)$ implies
$\sqrt{f_n}\to\sqrt{f}$ in $L^2(X,\mm)$, $\mathsf F$ is $L^1$-lower
semicontinuous.

%
\begin{proposition}\label{propFisher}
Let $f\in L^1_+(X,\mm)$. Then $\sqrt{f}\in\cV$ if and only if
$f_N:=\min\{f,N\}\in\cV$ for all $N$
and $\int_{{\{f>0\}}}\Gamma(f)/{f}\,\d\mm<\infty$. If this is the case,
\[
\mathsf{F}(f)=\int_{\{f>0\}}\frac{\Gamma(f)}{f}\,\d
\mm.
\]
In addition, $\mathsf F$ is convex in $L^1_+(X,\mm)$.
\end{proposition}

We refer to \cite{AGS11a}, Lemma~4.10, for the \emph{proof},
{observing only that $\Gamma(f)$ is well defined $\mm$-a.e. in
the class of functions satisfying $f_N:=\min\{f,N\}\in\cV$ for all
$N$ thanks to the locality property \eqref{eq74}}.
By applying the results of \cite{AGS11a},
we can prove that $(\sH{t})_{t\ge0}$ is a continuous semigroup in
$\probt X$ and we can
calculate the dissipation rate of the entropy functional
along it. Some of the results below are very simple
in the case $\mm(X)<\infty$, {where the function $V$ in \eqref{eq8} reduces
to a constant.}


%
\begin{lemma}[(Estimates on moments, Fisher information, Entropy, metric
derivative)]\label{leendis}
If $f\in L^2(X,\mm)$ is a probability density, for $\mu_t=\sP t f \mm
$ it holds
%
\begin{equation}
\label{eq63} \llvert \dot\mu_t\rrvert ^2\le\mathsf F(
\sP t f)\qquad\mbox{for $\Leb{1}$-a.e. $t>0$.}
\end{equation}
If $f\in L^1(X,\mm)$ is a probability density, for every $T>0$ there
exists $C_T>0$ such that
%
\begin{eqnarray}
\label{eq66}\qquad  \int_0^T \mathsf F( \sP s f)
\,\d s+\int_0^T \int_X
V^2 \sP sf\,\d \mm\,\d s &\le& C_T \biggl(
\operatorname{Ent}_{\tilde\mm}(f\mm)+\int_XV^2f
\,\d\mm \biggr),
\\
\label{eqextraextra} \entv(\sP tf)&\leq&\entv(f\mm)\qquad\forall t\geq0.
\end{eqnarray}
Finally, for every $\bar\mu\in\probt X$ the map $t\mapsto\mu
_t:=\sfH_t\bar\mu$
is a continuous curve in $\probt X$ with respect to $W_2$.
\end{lemma}

\begin{pf} The estimate \eqref{eq63} follows by \cite{AGS11a}, Lemma~6.1,
which can be applied here since the Dirichlet form $\cE$ coincides
with the Cheeger energy (Theorem~\ref{thmcE=C}).
The estimate \eqref{eq66} follows \cite{AGS11a}, Theorem~4.20,
thanks to the integrability condition \eqref{eq8}, when $f\in
L^2(X,\mm)$. In the general
case, it can be recovered by a truncation argument, using the lower
semicontinuity of
$\mathsf F$. The estimate \eqref{eqextraextra} can be obtained by a
similar approximation
argument from the detailed energy dissipation computation made in \cite
{AGS11a}, Theorem~4.16(b)
[the more general underlying fact is that,
independently of curvature assumptions,
$L^2$ gradient flows are contained in $W_2$ gradient flows of the Entropy
in the sense (weaker than $\EVI$) of energy dissipation].

Concerning the continuity of the map $t\mapsto\sfH_t\bar\mu$ with
respect to $W_2$ for every $\bar\mu\in\probt X$, it is a standard consequence
of contractivity and existence of a dense set of initial conditions
[namely, {thanks to \eqref{eq63} and \eqref{eq66}},
the set $D(\entv):=\{\nu\in\probt X\dvtx  \operatorname{Ent}_{\mm}(\nu
)<\infty\}$] for which
continuity holds up to $t=0$.
\end{pf}

\textit{Integration by parts for probability densities.}
We shall see now that
assuming the Bakry--\'Emery $\BE K\infty$ condition,
integration by parts formulae for the ${\DeltaE^{(1)}}$ operator
can be extended to
probability densities with finite Fisher information,
provided that the set of test functions $\varphi$ is restricted to
the spaces
$\cVu, \cVd$ defined in \eqref{eq67}.
Recall that $\cVu, \cVd$ are dense in $\cV$ w.r.t. the strong
topology and that
\[
\Gamma(\varphi_\eps-\varphi)\weaksto0\qquad\mbox{in
$L^\infty(X,\mm)$ as $\eps\down0$ for every }\varphi \in\cVu,
\]
where $\varphi_\eps$ are defined as in \eqref{eqGG4}, since
$\Gamma(\varphi_\eps-\varphi)$ is uniformly bounded and converges to
$0$ in $L^1(X,\mm)$
{(by the strong convergence of $\varphi_\eps$ to $\varphi$ in $\cV$)}.

In the sequel, we introduce an extension of the bilinear form
$\Gamma(f,g)$, denoted $\tilde\Gamma(f,g)$, which is particularly
appropriate to deal with probability densities $f$ with finite Fisher
information and test functions $\varphi\in\cVu$.

%
\begin{definition}[{[Extension of $\Gamma(f,\varphi)$]}]
Let $f= g^2\in L^1_+(X,\mm)$ with $\mathsf{F}(f)=4\cE(g)<\infty$ and
{$g\geq0$}.
For all $\varphi\in\cVu$ we define
%
\begin{equation}
\label{eq69} \tilde\Gamma(f,\varphi) :=2g\Gamma(g,\varphi).
\end{equation}
\end{definition}

The definition is well posed, consistent with the case when $f\in\cV$
and it holds
%
\begin{equation}
\label{eq69bis} \tilde\Gamma(f,\varphi)=\lim_{N\to\infty}
\Gamma(f_N,\varphi) \qquad\mbox{in $L^1(X,\mm)$.}
\end{equation}
Indeed, thanks to \eqref{eqGchain} and to the fact that if $\sfF
(f)<\infty$ then
$f_N=(g_N)^2\in\cV$, where $g_N:=g\land\sqrt N$;
it follows that $\Gamma(f_N,\varphi)=2g\nchi_N\Gamma(g,\varphi)$,
$\nchi_N$ being the characteristic function of the set $\{f< N\}$
and
\begin{eqnarray*}
\int_X\bigl\llvert \tilde\Gamma(f,\varphi)-
\Gamma(f_N,\varphi)\bigr\rrvert \, \d \mm &=&2\!\!\int
_X \llvert 1-\nchi_N\rrvert g\Gamma(g,\varphi)
\,\d\mm
\\
&\leq& \biggl(\bigl\llVert \Gamma(\varphi)\bigr\rrVert _\infty\mathsf
F(f) \int_{\{f\ge N\}} f\, \d \mm \biggr)^{\fraca12}
\end{eqnarray*}
thus proving the limit in \eqref{eq69bis}. The same argument provides
the estimate
%
\begin{equation}
\label{eqFisherbound} \qquad \int_X\psi\bigl\llvert \tilde\Gamma(f,
\varphi) \bigr\rrvert \,\d\mm\leq\sqrt{{\mathsf F}(f)} \biggl(\int
_X\psi^2\Gamma (\varphi) f\,\d\mm
\biggr)^{1/2},\qquad\varphi\in\cVu, \psi\geq0.
\end{equation}

%
\begin{theorem}[(Integration by parts of ${\DeltaE^{(1)}}$)]
\label{thmfurtherDeltaEu}
If $\BE K\infty$ holds, then for every $f\in L^1_+(X,\mm)$ with
$\mathsf F(f)<\infty$ we have
%
\begin{equation}
\label{eqintbyparts1}\int_X \tilde\Gamma(f,\varphi)\,\d\mm=-\int
_X f\DeltaE\varphi\,\d\mm \qquad \mbox{for every }\varphi
\in\cVd.
\end{equation}
In addition, if $f\in D({\DeltaE^{(1)}})$ it holds
%
\begin{equation}
\label{eqdefDelta1} \int_X\tilde\Gamma(f,\varphi)\,\d\mm=-\int
_X \varphi{\DeltaE ^{(1)}}f\, \d\mm\qquad\forall
\varphi\in\cVu.
\end{equation}
\end{theorem}

\begin{pf} Formula \eqref{eqintbyparts1} follows by the limit
formula in
\eqref{eq69bis} simply integrating by parts before passing
to the limit as $N\to\infty$.
Assuming now $f\in D({\DeltaE^{(1)}})$ we have
\begin{eqnarray*}
-\int_X {\DeltaE^{(1)}}f \varphi\,\d\mm&=& \lim
_{t\down0}\frac{1}t\int_X (f-\sP t
f)\varphi\,\d\mm= \lim_{t\down0}\frac{1}t\int
_X f(\varphi-\sP t\varphi)\,\d\mm
\\
&=&\lim_{t\down0}\frac{1}t \int_0^t
\int_X f \DeltaE ( \sP s\varphi)\,\d\mm\,\d s
\\
&=&\lim_{t\down0} \frac{1}t \int_0^t
\int_X \tilde\Gamma(f,\sP s \varphi)\,\d\mm\,\d s=\int
_X \tilde\Gamma(f,\varphi),
\end{eqnarray*}
where the last limit follows by \eqref{eqFisherbound} and the fact
that
\[
\hspace*{-10pt}\Gamma \biggl(\frac{1}t\int_0^t
\sP s\varphi\,\d s-\varphi \biggr)\le \frac{1}t\int_0^t
\Gamma(\sP s\varphi-\varphi)\,\d s\weaksto0\qquad \mbox{in }L^\infty(X,
\mm) \mbox{ as }t\down0.
\]
\upqed
\end{pf}

\subsection{Log-Harnack and L\,logL estimates}
\label{subseclogHarnack}
%

%
%
\begin{lemma}\label{lemreglog1}
Let $\omega\dvtx [0,\infty)\to\R$ be a function
of class $C^2$, let $f\in\Lip_b(X)\cap\cV$
and let $\mu\in{\mathscr P}(X)$.
The function
\[
G(s):=\int_X \omega (\sP{t-s}f )\,\d\sH{s}\mu,\qquad s
\in[0,t], 
\]
belongs to $C^0([0,t])\cap C^1(0,t)$ and for every $s\in(0,t)$ it holds
%
\begin{equation}
\label{eqGreg3} G'(s)=\int_X
\omega''(\sP{t-s}f)\Gamma(\sP{t-s} f)\,\d{\sH{s}\mu}.
\end{equation}
\end{lemma}

\begin{pf}
Since $ {\sH{s}\mu}$ are all probability measures
is not restrictive to assume $\omega(0)=0$. Continuity of $G$ is
obvious, since $\sP{t-s} f$ are equi-Lipschitz, equi-bounded
and the semigroup $\sH{t}$ is weakly-continuous.
Let us first consider the case $\mu=\zeta\mm$ with $\zeta\in
{L^1\cap L^\infty}(X,\mm)$
[in particular $\zeta\in L^2(X,\mm)$].
Setting $f_{t-s}:=\sP{t-s}f$ and $\zeta_s:=\sP s\zeta$,
we observe that a.e. in the open interval $(0,t)$ the following
properties hold:
\begin{itemize}
\item[--] $s\mapsto\zeta_s$ is differentiable in $L^2(X,\mm)$;
\item[--] $s\mapsto\omega(f_{t-s})$ is differentiable in $L^2(X,\mm)$, with
derivative $-\omega'(f_{t-s})\DeltaE f_{t-s}$.
\end{itemize}
Therefore, the chain rule \eqref{eqGchain-laplace} gives
\begin{eqnarray*}
G'(s)&=&-\int_X \omega'(f_{t-s})
\zeta_s \DeltaE f_{t-s}\,\d\mm+ \int_X
\omega(f_{t-s})\DeltaE\zeta_s \,\d\mm
\\
&=& \int_X \bigl(\Gamma\bigl(\omega'(f_{t-s})
\zeta_s,f_{t-s}\bigr) -\Gamma\bigl(\omega(f_{t-s}),
\zeta_s\bigr) \bigr)\,\d\mm
\\
&=&\int_X \omega''(f_{t-s})
\Gamma(f_{t-s}) \zeta_s\,\d\mm
\end{eqnarray*}
for $\Leb{1}$-a.e. $s\in(0,t)$. But, since the right-hand side is
continuous, the formula holds pointwise in $(0,t)$.
The formula for arbitrary measures $\mu=\zeta\mm\in{\mathscr P}(X)$
then follows
by monotone approximation (i.e., considering $\zeta_n=\min\{\zeta,n\}
/c_n$ with $c_n\uparrow1$ normalizing constants),
by using the uniform $L^\infty$ bound on
$\omega''(f_{t-s})$ and on $\Gamma(f_{t-s})$;
the
formula \eqref{eqGreg3}
for $G'$ still provides a continuous function
since $\sfH_s \mu=\zeta_s\mm$ and
$\zeta_s=\sP s \zeta$ is strongly continuous in $L^1(X,\mm)$.
Finally, if $\mu\in{\mathscr P}(X)$,
for every $\eps>0$ and $s\in(\eps,t-\eps)$, setting
$\mu_\eps=\sH{\eps}\mu$ we have
\[
G(s)=\tilde G_\eps(s-\eps) \qquad\mbox{where } \tilde
G_\eps(s):=\int_X \omega \bigl(\sP {(t-\eps)-s}f
\bigr)\,\d\sH{s}\mu_\eps.
\]
Since $\tilde G_\eps$ is a $C^1$ function in $(\eps,t-\eps)$ by the previous
considerations and $\eps>0$ is arbitrary, we conclude that $G\in
C^1(0,t)$ and
for every $s\in(\eps,t-\eps)$ its derivative coincides with
\begin{eqnarray*}
\tilde G_\eps'(s-\eps)&=& \int_X
\omega''\bigl(\sP {t-\eps-(s-\eps)}f\bigr)\Gamma\bigl(
\sP{t-\eps-(s-\eps)} f\bigr)\,\d{ \sH{s-\eps} \mu_\eps}
\\
&=&\int_X \omega''(\sP{t-s}f)
\Gamma(\sP{t-s} f)\,\d{ \sH{s}\mu}.
\end{eqnarray*}
%
\upqed\end{pf}

In order to prove the L\,logL regularization, we use
the next lemma, which follows by a careful adaptation to our more abstract
context of a result by Wang \cite{Wang11}, Theorem~1.1(6).

%
%
\begin{lemma}[(Log-Harnack inequality)]$\!\!\!\!$
For every nonnegative $f\in\break L^1(X,\mm)+L^\infty(X,\mm)$,
$t>0$, $\eps\in[0,1]$, and $x, y\in X$
we have $\log(1+f)\in L^1(X,\sH{t}\delta_y)$
with
%
\begin{equation}
\label{eq62} \tsP{t}\bigl(\log(f+\eps)\bigr) (y)\leq\log \bigl(\tsP{t} f(x)+
\eps \bigr)+ \frac{{\mathsf d}^2(x,y)}{4 \mathrm I_{2K}(t)}.
\end{equation}
\end{lemma}

\begin{pf}
In the following, we set $\omega_\eps(r):=\log(r+\eps)$,
for $r\ge0$ and $\eps\in(0,1]$.
Let us first assume in addition that $f\in\Lip_b(X)\cap L^1(X,\mm
)\cap\cV$,
let $\gamma\dvtx [0,1]\to X$ be a Lipschitz curve connecting $x$ to $y$ in $X$,
and, recalling the definition \eqref{eq57} of $\rmI_K$, let
\[
\tilde\gamma_r=\gamma_{\theta(r)}\qquad\mbox{with } \theta(r)=
\frac{\rmI_{2K} (r)}{\rmI_{2K} (t)}, r\in[0,t].
\]
We set $f_{t-s}:= \sP{t-s} f$ and, for $r\in[0,t]$ and $s\in(0,t)$,
we consider the functions
\[
G(r,s):=\int_X \omega_\eps(f_{t-s})
\,\d{\sH{s}}\delta_{\tilde
\gamma_r} =F_s(\tilde\gamma_r)
\qquad\mbox{with } F_s(x):=\tsP{s} \bigl(\omega_\eps(f_{t-s})
\bigr) (x).
\]
Notice that Lemma~\ref{lemreglog1} with $\mu=\delta_{\tilde\gamma
_r}$ ensures that
for every $r\in[0,t]$ the function $s\mapsto G(r,s)$ is
Lipschitz in $[0,t]$, with
%
\begin{equation}
\label{eqalimon} \qquad\frac\partial{\partial s}G(r,s)= \int_X
\omega_\eps''(f_{t-s})
\Gamma(f_{t-s})\,\d{\sH{s}}\delta _{\tilde\gamma_r} \qquad{\mbox{for
every 
$s\in(0,t)$.}}
\end{equation}
This gives immediately that $G(r,\cdot)$ are uniformly Lipschitz in $[0,t]$
for $r\in[0,t]$. On the other hand, since $\tilde\gamma$ is
Lipschitz, the map
$r\mapsto\sP s\delta_{\tilde\gamma_r}$ is Lipschitz in $[0,t]$
with
respect to the $L^1$-Wasserstein distance
$W_1$
uniformly w.r.t. $s\in[0,t]$. Hence, taking also the fact that
$\omega(f_{t-s})$ are equi-Lipschitz into account, it follows that
also the maps
$G(\cdot,s)$ are Lipschitz in $[0,t]$, with Lipschitz constant
uniform w.r.t. $s\in[0,t]$. These properties imply that the map
$s\mapsto G(s,s)$ is Lipschitz in $[0,t]$.

Since the chain rule and \eqref{eq44} [which can be applied
since we can subtract the constant $\omega_\eps(0)$ from $F_s$
without affecting the calculation of its slope]
give
\[
\llvert \rmD F_s\rrvert ^2\le\rme^{-2Ks} \sP s
\bigl(\bigl(\omega_\eps '(f_{t-s})
\bigr)^2\Gamma(f_{t-s} ) \bigr),
\]
we can use $\theta'(r)=\rme^{2K r}/\rmI_{2K}(t)$ to get the
pointwise estimate
\begin{eqnarray*}
&&\rme^{K(s-r)} \limsup_{h\down0}\frac{\llvert G(r+h,s)-G(r,s)\rrvert }h
\\
&&\qquad \le \sqrt{\theta'(r)}\frac{\llvert \dot\gamma_{\theta(r)}\rrvert }{\sqrt{\rmI_{2K}(t)}}
\rme^{Ks} \llvert \rmD F_s\rrvert (\tilde
\gamma_r)
\\
&&\qquad \le \theta'(r)\frac{\llvert \dot\gamma_{\theta(r)}\rrvert ^2}{4\rmI_{2K}(t)}+ \int
_X \bigl(\omega_\eps'(f_{t-s})
\bigr)^2\Gamma(f_{t-s}) \,\mathrm{d} \sP s\delta_{\tilde\gamma_r}.
\end{eqnarray*}
Applying the calculus lemma \cite{AGS08}, Lemma~4.3.4, and using the identity
$\omega_\eps''=-(\omega_\eps')^2$, the previous inequality with
$r=s$ in combination with \eqref{eqalimon} gives
\begin{eqnarray*}
\frac{\d}{\d s}G(s,s)&\le&\lim_{h\down0}\frac{G(s,s-h)-G(s,s)}h+
\limsup_{h\down0}\frac{\llvert G(s+h,s)-G(s,s)\rrvert }h
\\
&\le&\theta'(s)\frac
{\llvert \dot\gamma_{\theta(s)}\rrvert ^2}{4 \rmI_{2K}(t)} \qquad\mbox{for $\Leb{1}$-a.e. $s
\in(0,t)$.}
\end{eqnarray*}
An integration in $(0,t)$ and a minimization w.r.t. $\gamma$ yield
%
\begin{equation}
\label{eqGalmost} \int_X \omega_\eps(f)\,\d{
\sH{t}}\delta_y\le \omega_\eps(f_t) (x)+
\frac{{\mathsf d}^2(x,y)}{4 \rmI_{2K}(t)}.
\end{equation}
If $f\in L^\infty(X,\mm)$,
we consider
a uniformly bounded sequence $(f_n)$ contained in
$\Lip_b(X,\mm)\cap L^1(X,\mm)\cap\cV$
converging to $f$ pointwise $\mm$-a.e.
Since $\omega_\eps\ge\log(\eps)$ and $\tsP{t} f_n$ converges to
$\tsP{t} f$ pointwise, Fatou's lemma yields \eqref{eqGalmost}
also in this case. Finally, a truncation argument extends the validity of
\eqref{eqGalmost}
and \eqref{eq62} to arbitrary nonnegative $f\in
L^1(X,\mm)+L^\infty(X,\mm)$.
Passing to the limit as $\eps\down0$,
we get \eqref{eq62} also in the case $\eps=0$.

Finally, notice that \eqref{eqGalmost} for $\eps=1$ and
the fact that $\tsP{t} f(x)$ is
finite for $\mm$-a.e. $x$
yield the integrability of {$\log(1+f)$} w.r.t.
$\sH{t}\delta_y$.
\end{pf}

In the sequel, we set
\[
\sfH_t\delta_y=u_t[y] \mm\quad
\mbox{so that}\quad \tsP{t} f(y)=\int_X fu_t[y]
\,\d\mm
\]
for every $\mm$-measurable and {$\mm$-}semi-integrable function $f$.

%
\begin{corollary}
For every $t>0$
and $y\in X$, we have
%
\begin{equation}
\label{eqGquasi} \int_X u_t[y]\log
\bigl(u_t[y] \bigr)\,\d\mm\le\log \bigl(u_{2t}[y](x)
\bigr)+ \frac{{\mathsf d}^2(x,y)}{4 \rmI_{2K}(t)}\qquad\mbox{for $\mm$-a.e. $x\in X$.}\hspace*{-35pt}
\end{equation}
In particular, when $\mm$ is a probability measure,
\[
u_{2t}[y](x)\ge\exp \biggl(-\frac{{\mathsf d}^2(x,y)}{4 \rmI
_{2K}(t)} \biggr)
\qquad\mbox{for $\mm$-a.e. $x\in X$.}
\]
\end{corollary}

\begin{pf}
Simply take $f=u_t[y]$ in \eqref{eq62} and notice that $\tsP{t}f
(x)=u_{2t}[y](x)$ for $\mm$-a.e.
$x\in X$ by the semigroup property.
\end{pf}

In the next crucial result, we will show that \eqref{eqGquasi} yields
$\entv(\sfH_t\mu)<\infty$ for every measure $\mu\in\probt X$.

%
%
\begin{theorem}[(L\,logL regularization)]\label{thmLlogL}
Let $\mu\in\mathscr P_2(X)$ and let $f_t\in\break L^1(X,\mm)$, $t>0$, be
the densities of $\sfH_t\mu
\in\probt X$.
Then
%
\begin{equation}
\label{eqGpoint} \int_X f_t\log
f_t\,\d\mm\le \frac{1}{2\rmI_{2K}(t)} \biggl(r^2+\int
_X{\mathsf d}^2(x,x_0)\,\d\mu (x)
\biggr) -\log \bigl(\mm\bigl(B_r(x_0)\bigr) \bigr)
\hspace*{-25pt}
\end{equation}
for every $x_0\in X$ and $r, t>0$.
\end{theorem}

\begin{pf}
By approximation, it suffices to consider the case when $\mu=f\mm$
with $f\in L^2(X,\mm)$.
Let us fix $x_0\in X$ and $r>0$, set $\sfz=\mm(B_r(x_0))$ and $\nu
=\sfz^{-1}\mm\res B_r(x_0)$.
Notice first that we have the pointwise inequality
\begin{eqnarray*}
\tsP{t}f(z) \log\bigl(\tsP{t}f(z)\bigr) &=& \biggl(\int_X
u_t[z] \,\d\mu \biggr)\log \biggl(\int_X
u_t[z]\,\d\mu \biggr)
\\
&\le&\int_X u_t[z](y)\log
\bigl(u_t[z](y) \bigr)\,\d\mu(y).
\end{eqnarray*}
Since $\tsP{t} f=f_t$ $\mm$-a.e. in $X$, integrating with respect to
$\mm$ and using the symmetry property of $u_t$,
\eqref{eqGquasi}, and Jensen's inequality,
we get
\begin{eqnarray*}
\int_X f_t\log f_t \,\d\mm&\le&
\int_X \biggl(\int_X
u_t[y](z)\log u_t[y](z)\,\d\mm(z) \biggr)\,\d\mu(y)
\\
&=& \int_{X\times X} \biggl(\int_X
u_t[y](z)\log u_t[y](z)\,\d\mm(z) \biggr)\,\d\nu(x)\,\d
\mu(y)
\\
&\le& \int_{X\times X} \biggl(\log \bigl(u_{2t}[y](x)
\bigr)+ \frac{{\mathsf d}^2(x,y)}{4 \rmI_{2K}(t)} \biggr)\,\d\nu(x)\,\d\mu(y)
\\
&\le& \log \biggl(\int_{X}\int_X
u_{2t}[y](x)\,\d\nu(x)\,\d\mu(y) \biggr)
\\
&&{}+ \frac{1}{2\rmI_{2K}(t)} \biggl(r^2+\int_X{
\mathsf d}^2(x,x_0)\,\d\mu (x) \biggr)
\\
&\le& \log\frac{1}{\sfz} +\frac{1}{2\rmI_{2K}(t)} \biggl(r^2+\int
_X{\mathsf d} ^2(x,x_0)\,\d\mu(x)
\biggr),
\end{eqnarray*}
where we used the inequality
\[
\int_X u_{2t}[y](x)\,\d\nu(x)=
\frac{1}{\sfz}\int_{B_r(x_0)} u_{2t}[y](x)\,\d\mm(x)
\le\frac{1}{\sfz}\qquad{\mbox{for every }y\in X}.
\]
\upqed
\end{pf}

We conclude with a further regularization
and an integration by parts
formula for ${\DeltaE^{(1)}}$
in a special case.
Notice that thanks to the regularizing effect of $\sfH_t$ we
can extend the mollification $\frh^\eps$ of the semigroup in \eqref
{eqGG4} to measures $\mu\in\probt X$, that is, we set
%
\begin{equation}
\label{eq77} \mathfrak h^\eps\mu:=\frac{1}\eps\int
_0^\infty f_r \kappa(r/\eps )\,\d r,
\qquad f_r\mm= \sfH_r\mu\qquad\mbox{for }r>0,
\end{equation}
with $\kappa$ as in \eqref{eq65}, obtaining a map $\mathfrak h^\eps
\dvtx \probt X\to D({\DeltaE^{(1)}})$ {still
satisfying \eqref{eqbanff1} [via an elementary approximation based on the
characterization \eqref{eq2} of $D({\DeltaE^{(1)}})$].}

%
\begin{lemma}\label{lemintbyparts2}
Let $\tilde\mu\in\probt X$
and let $f=\frh^\eps\tilde\mu$ as in
\eqref{eq77}.
Then for every
$\eps, T>0$ there exists a constant $C(\eps,T)$
such that
%
\begin{equation}
\label{eq78} \mathsf F(\sP t f)\le C(\eps,T) \biggl(1+\int
_X V^2\,\d\mu \biggr)
\end{equation}
and, writing $\mu_t=\sP tf\mm$,
%
\begin{equation}
\label{eq78bis} \llvert \dot\mu_t\rrvert ^2\le C(\eps,T)
\biggl(1+\int_X V^2\,\d\mu \biggr)\qquad
\mbox{for $\Leb{1}$-a.e. $t\in[0,T]$.}
\end{equation}
Moreover, for every bounded and nondecreasing Lipschitz function
$\omega\dvtx [0,\infty)\to\R$
such that $\sup_r r\omega'(r)<\infty$,
we have
%
\begin{equation}
\int_X \omega(f){\DeltaE^{(1)}}f\,\d\mm+ 4\int
_X f\omega'(f) \Gamma(\sqrt f)\,\d\mm
\leq0.\label{eq79}
\end{equation}
\end{lemma}

\begin{pf}
Combining \eqref{eq66}, \eqref{eqGpoint}, the commutation
identity $\sP t\frh^\eps=\frh^\eps\sP t$, and the convexity
of $\mathsf F$
we get \eqref{eq78}. {We obtain immediately the Lipschitz estimate
\eqref{eq78bis} from \eqref{eq78}
and \eqref{eq63} when $f\in L^2(X,\mm)$. The general case follows by
a truncation
argument.}
Concerning \eqref{eq79}, if $\tilde\mu=\tilde f\mm$ with $\tilde
f\in L^2(X,\mm)$,
then $f\in L^2(X,\mm)$, ${\DeltaE^{(1)}}f=\DeltaE f\in L^2(X,\mm)$
and the
stated inequality
is an equality, by the chain rule
$\Gamma(f,\omega(f))=\omega'(f)\Gamma(f)=4f\omega'(f) \Gamma
(\sqrt f)$.
In the general case,
we approximate $\tilde\mu$ in $\probt X$ by a sequence of measures
$\tilde\mu_n=\tilde f_n\mm$ with $\tilde f_n\in L^2(X,\mm)$ and we
consider $f_n=\frh^\eps\mu_n$.
By \eqref{eqbanff1}, we obtain that ${\DeltaE^{(1)}}f_n\to{\DeltaE
^{(1)}}f$ in
$L^1(X,\mm)$
while, setting $\phi(s)=\int_0^s\sqrt{r^2\omega'(r^2)}\,\d r$,
the lower semicontinuity of
$g\mapsto\int\Gamma(g)\,\d\mm$ and
the strong convergence
of $\sqrt{f_n}$ to $\sqrt f$ in $L^2(X,\mm)$
give
\begin{eqnarray*}
4\int_X f\omega'(f) \Gamma(\sqrt f)\,\d
\mm&=&\int_X \Gamma \bigl(\phi(\sqrt f)\bigr)\,\d\mm\leq
\liminf_{n\to\infty}\int_X\Gamma \bigl({\phi(
\sqrt{f_n})}\bigr)\,\d\mm
\\
&=&\liminf_{n\to\infty}\int_X f_n
\omega'(f_n) \Gamma(\sqrt{f_n})\,\d\mm.
\end{eqnarray*}
\upqed
\end{pf}

Motivated by the regularity assumptions needed in the next
section, we give the following definition.

%
\begin{definition}[(Regular curve)]
\label{defregularcurve}
Let $\rho_s=f_s\mm\in\Probabilities{X}$, $s\in[0,1]$.
We say that $\rho$ is regular if:
\begin{longlist}[(b)]
\item[(a)] $\rho\in\AC2{[0,1]}{(\probt X, W_2)}$;
\item[(b)] $\operatorname{Ent}_{\mm}(\rho_s)$ is bounded;
\item[(c)] $f\in\rmC^1([0,1];L^1(X,\mm))$;
\item[(d)] There exists $\eta>0$ such that
for all $s\in[0,1]$ the function
$f_s$ is representable in the form $\frh^\eta\tilde f_s$
for some $\tilde f_s\in L^1(X,\mm)$; in addition
${\DeltaE^{(1)}}f\in\rmC([0,1];L^1(X,\mm))$ and
%
\begin{equation}
\label{eqGatwick} \sup \bigl\{\mathsf{F}(\sP t f_s)\dvtx s\in[0,1], t
\in[0,T] \bigr\}<\infty \qquad\forall T>0.
\end{equation}
\end{longlist}
\end{definition}

In particular, if $\rho_s=f_s\mm$ is a regular curve, for every $T>0$
there exist positive constants $M_T,E_T,F_T$ such that
\begin{eqnarray}
\int_X V^2\,\d{\sH{t}}
\rho_s\le M_T, \qquad \operatorname{Ent}_{\mm}(\sP t
\rho_s)\le E_T,\qquad  \mathsf F(\sP t f_s)\le
F_T, \nonumber\\
\eqntext{s\in[0,1], t\in[0,T].}
\end{eqnarray}

%
\begin{proposition}[(Approximation by regular curves)]\label{propapprox}
For all
$\rho\in\break \AC2{[0,\allowbreak 1]}{(\probt X, W_2)}$ there exist regular curves
$\rho^n$
such that $\rho^n_s\to\rho_s$ in $\probt{X}$ for all $s\in[0,1]$ and
%
\begin{equation}
\label{eqoptimalaction} \limsup_n\int_0^1
\bigl\llvert \dot\rho^n_s\bigr\rrvert ^2\,\d
s\leq\int_0^1\llvert \dot\rho _s
\rrvert ^2\,\d s.
\end{equation}
{Furthermore, we can build $\rho^n$ in such a way that
$\entv(\rho^n_0)\leq\entv(\rho_0)$.}
\end{proposition}

\begin{pf} First, we extend $\rho$ by continuity and with constant
values in $(-\infty,\allowbreak 0)\cup(1,\infty)$. Then we
define $\rho^{n,1}_s:= {\sH{\tau_n}}\rho_s$, with $\tau_n^{-1}\in
[n,2n]$. By the {continuity and} contractivity
properties of $\sH{t}$, we see that $\rho^{n,1}$ fulfils
(a) and (b) of the definition of regularity, the convergence
requirement of the Lemma and \eqref{eqoptimalaction}.
Indeed, obviously condition (a) is fulfilled, while we
gain $\rho^{n,1}_s\ll\mm$ and $\sup_s\entv(\rho^{n,1}_s)<\infty$
by Theorem~\ref{thmLlogL}.
{In addition, \eqref{eqextraextra} shows that $\entv(\rho
^{n,1}_0)\leq\entv(\rho_0)$.}

In order to achieve condition (c), we do a second regularization, by
averaging w.r.t. the $s$ variable:
precisely, denoting by $f^{n,1}_s$ the densities of $\rho^{n,1}_s$, we
set $\rho^{n,2}_s:=f^{n,2}_s\mm$, where
\[
f^{n,2}_s:=\int_{\R}f^{n,1}_{s-s'}
\chi_n\bigl(s'\bigr)\,\d s'
\]
and $\chi_n\in C^\infty_c(\R)$ are standard convolution kernels
{with support in $(0,\infty)$ convergent to the identity.}
By the convexity properties of squared Wasserstein distance and entropy,
we see that the properties (a), (b) are retained and that the action in
\eqref{eqoptimalaction} does not increase. In addition,
we clearly gain property (c) and, {since $\rho^{n,1}$ is constant in
$(-\infty,0]$, this regularization does not increase the
entropy at time $0$}.

In the last step, we mollify using the heat semigroup, setting $\rho
^{n}_s:=f^{n}_s\mm$, where
$f^{n}_s=\frh^{\eps_n} f^{n,2}_s$ and $\eps_n\downarrow0$.
By the same reasons used for $\rho^{n,2}$, property (a) is retained by
$\rho^n$ and the action in
\eqref{eqoptimalaction}, as well as the Entropy at any time $s$, do
not increase [because the support
of the kernel $\kappa$ is contained in $(0,\infty)$]. In addition,
(c) is retained as well
since $\frh^{\eps}$ is a continuous linear map from $L^1(X,\mm)$ to
{$L^1(X,\mm)$}.
{With this mollification, we gain property (d) from
\eqref{eqbanff1}}
and from \eqref{eq78}, which provides the sup bound \eqref
{eqGatwick} on Fisher information.
\end{pf}

\subsection{Action estimates}
\label{subsecaction}

This section contains the core of the arguments
leading to the equivalence Theorem~\ref{thmmain-identification}.
We refer to \cite{Daneri-Savare08} for
the underlying geometric ideas in a smooth Riemannian context
and the role of the Bochner identity.
Here, we had to circumvent many technical
difficulties related to regularity issues,
to the lack of ultracontractivity properties of the semigroup
$(\sP t)_{t\ge0}$ (i.e.,\vspace*{1pt} regularization from $L^1$ to
$L^\infty$),
and to the weak formulation of the Bakry--\'Emery condition.

Since we shall often consider regular curves $\rho\in\AC
2{[0,1]}{(\probt X, W_2)}$
of measures representable in the form
$\rho=f\mm$ with $f\in C^1([0,1];L^1(X,\mm))$, we shall denote by
$\dot f\in C([0,1];L^1(X,\mm))$ the functional
derivative in\break
$L^1(X,\mm)$, retaining the notation $\llvert \dot\rho_t\rrvert $ for the metric
derivative
w.r.t. $W_2$.

We begin with a simple estimate of the oscillation of $s\mapsto\int_X\varphi\,\d\rho_s$ along absolutely continuous or
$\rmC^1$ curves.

%
\begin{lemma}\label{legeneralplans}
For all $\rho\in\AC2{[0,1]}{(\probt X, W_2)}$, it holds
%
\begin{eqnarray}
\label{eqoscphi} \biggl\llvert \int_X\varphi\,\d
\rho_1-\int_X\varphi\,\d\rho_0
\biggr\rrvert \leq\int_0^1\llvert \dot
\rho_s\rrvert \biggl(\int_X \llvert \rmD
\varphi\rrvert ^2\, \d\rho_s \biggr)^{1/2}\,\d s
\nonumber
\\[-8pt]
\\[-8pt]
\eqntext{\mbox{for every }\varphi\in\cVu.}
\end{eqnarray}
If moreover $\rho=f\mm$ with $f\in\rmC^1 ([0,1];L^1(X,\mm
) )$, for all $\varphi\in\cVu$ it holds
%
\begin{equation}
\label{eqGconj} \qquad\biggl\llvert \int_X \dot
f_s \varphi\,\d\mm\biggr\rrvert \leq \llvert \dot\rho_s
\rrvert \biggl(\int_X \llvert \rmD\varphi\rrvert
^2\,\d\rho_s \biggr)^{1/2}\qquad \mbox{for $
\Leb1$-a.e. $s\in(0,1)$.}
\end{equation}
\end{lemma}

\begin{pf}
It is easy to check that \eqref{eqoscphi} can be obtained using the
representation of
$\rho_s$ given by Lisini's theorem \cite{Lisini07} (see \cite{AGS11a},
Lemma~5.15).

Choosing now a Lebesgue point $\bar s$ both for $s\mapsto|\dot\rho
_s|^2$ and
$\int_X\llvert \rmD\varphi\rrvert ^2\,\d\rho_s$, for all $a>0$ we can pass to
the limit as $h\downarrow0$ in the inequality
\[
\frac{1}{h} \biggl\llvert \int_X\varphi\,\d
\rho_{\bar s+h}-\int_X\varphi\,\d\rho_{\bar s}
\biggr\rrvert \leq\frac{1}{2h}\int_{\bar s}^{\bar s+h}
\biggl(a\llvert \dot\rho _s\rrvert ^2+\frac{1}a
\int_X \llvert \rmD{\varphi}\rrvert ^2\,\d
\rho_s \biggr)\,\d s
\]
and then minimize w.r.t. $a$, obtaining \eqref{eqGconj}.
\end{pf}

%
%
\begin{lemma}
\label{lest-velocity} Let $\rho=f\mm\in\AC2{[0,1]}{(\probt X,W_2)}$
be a regular curve according to Definition~\ref{defregularcurve},
and let $\vartheta\dvtx [0,1]\to[0,1]$ be a $\rmC^1$
function with $\vartheta(i)=i$, $i=0,1$.
Define
\[
\rho_{\st}:= \sfH_{st}\rho_{\vartheta(s)}=f_{\st}
\mm,\qquad s\in [0,t], t\ge0.
\]
Then, for every $t\ge0$, the curve $s\mapsto\rho_{\st}$ belongs to
$\AC2{[0,1]}{(\probt X, W_2)}$ and
$\mathsf F(f_{s,t})$ is uniformly bounded.
Moreover,
for any $\varphi\in\Lip_b(X)$ with
bounded support,
setting $\varphi_s:=Q_{s}\varphi$,
the map $s\mapsto\int_X \varphi_s\,\d\rho_{\st}$ is absolutely continuous
in $[0,1]$ and
%
\begin{eqnarray}
\label{eqGKuwada} \frac{\d}{\d s}\int_X
\varphi_s\,\d\rho_{{s,t}}&=& \dot\vartheta(s)\int
_X \dot f_{ \vartheta(s)} \sP{st}\varphi_s\,
\d\mm-\frac{1}2\int_X \llvert \rmD
\varphi_s\rrvert ^2\,\d\rho_{\st}
\nonumber
\\[-8pt]
\\[-8pt]
&&{} - t\int
_X \tilde\Gamma(f_{\st},\varphi_s)\,
\d\mm
\nonumber
\end{eqnarray}
for $\Leb1$-a.e. $s\in(0,1)$.
\end{lemma}

\begin{pf}
We only consider the case $t>0$ and we set
$\tilde\rho_{s}:=\rho_{\vartheta(s)}=\tilde f_s\mm$.
Notice that $f_{\st}= \sP{st}\tilde f_s$ and $\tilde\rho_s$
satisfies
the same assumptions than $\rho_s$.
Since $\sfH_{t}$ is a Wasserstein $K$-contraction
\begin{eqnarray*}
W_2(\rho_{s_0,t},\rho_{s_1,t})&\le&
\rme^{-Ks_0t}W_2(\tilde\rho_{s_0},
\sfH_{(s_1-s_0)t}\tilde\rho_{s_1})
\\
&\le& \rme^{-Ks_0t} \bigl(W_2(\tilde\rho_{s_0},\tilde
\rho_{s_1})+ W_2(\tilde\rho_{s_1},
\sfH_{(s_1-s_0)t}\tilde\rho_{s_1}) \bigr),
\end{eqnarray*}
and \eqref{eq78bis} and the regularity of $\rho$ give
\[
W_2(\tilde\rho_{s_1},\sfH_{(s_1-s_0)t}\tilde
\rho_{s_1}) \leq C(\rho,T) (s_1-s_0)t\qquad
\mbox{whenever $(s_1-s_0)t\leq T$.}
\]
We conclude that $s\mapsto\rho_{\st}$
belongs to $\AC2{[0,1]}{(\probt X,W_2)}$.
Moreover, using the splitting
\begin{eqnarray*}
&&\int_X \varphi_{s_1}\,\d\rho_{s_1,t} -
\int_X \varphi_{s_0}\,\d\rho_{s_0,t}
\\
&&\qquad= \int_X \varphi_{s_1}\,\d
\rho_{s_1,t}- \int_X \varphi_{s_0}\,\d
\rho_{s_1,t}+ \int_X \varphi_{s_0}\,\d
\rho_{s_1,t}- \int_X \varphi_{s_0}\,\d
\rho_{s_0,t}
\\
&&\qquad\le\llVert \varphi_{s_1}-\varphi_{s_0}\rrVert
_\infty+ \Lip(\varphi_{s_0})W_2(\tilde
\rho_{s_1},\tilde\rho_{s_0})
\end{eqnarray*}
we immediately see that also $s\mapsto\int_X\varphi_s\,\d\rho_{\st
}$ is
absolutely continuous. In order to compute its derivative, we
write
\begin{eqnarray*}
\int\varphi_{s+h}\,\d\rho_{s+h,t}-\int_X
\varphi_s \,\d\rho _{\st}&=& \int\varphi_{s+h}\,
\d\rho_{s+h,t}-\int_X \varphi_s \,\d
\rho _{s+h,t}
\\
&&{}+ \int\sP{(s+h)t}\varphi_s\,\d(\tilde\rho_{s+h}-\tilde
\rho_s)
\\
&&{}+ \int( \sP{ht}\varphi_s -\varphi_s)\, \d{\sH{st}}
\tilde\rho_{s}.
\end{eqnarray*}
Now, the Hopf--Lax formula \eqref{eqidentities} and the strong
convergence of
$f_{s+h,t}$ to $f_{\st}$ in $L^1(X,\mm)$ yield
\[
\lim_{h\downarrow0}\frac{1}{h} \biggl( \int
\varphi_{s+h}\,\d\rho _{s+h,t}-\int_X
\varphi_s \,\d\rho_{s+h,t} \biggr) = - \frac{1}2 \int
_X \llvert \rmD\varphi_s\rrvert ^2
\,\d\rho_{\st}.
\]
The differentiability of $\rho_s$ in $L^1(X,\mm)$ yields
\[
h^{-1}\int\sP{(s+h)t}\varphi\,\d(\tilde
\rho_{s+h}-\tilde\rho _s)\to \dot\vartheta(s)\int
_X \sP{st}\varphi\,\dot f_{ \vartheta(s)}\,\d \mm.
\]
Finally, the next lemma yields
\[
h^{-1}\int( \sP{ht}\varphi_s -
\varphi_s)\, \d\sP{st}\rho_{s} \to-t\int
_X \tilde\Gamma(f_{\st},\varphi_s)\,
\d\mm.
\]
\upqed
\end{pf}

%
\begin{lemma}\label{lemasinRiemannian}
For all $\varphi\in\cVu$ and all $\rho=f\mm\in
\prob{X}$ with
$\mathsf{F}(f)<\infty$ it holds
\[
\lim_{h\downarrow0}\int_X\frac{ \sP h\varphi-\varphi}{h}f\,
\d\mm =-\int_X\tilde\Gamma(f,\varphi)\,\d\mm.
\]
\end{lemma}

\begin{pf}
We argue as in \cite{AGS11b}, Lemma~4.2, proving first that
%
\begin{equation}
\label{eqgrosseto} \int_X\frac{ \sP h\varphi-\varphi}{h}f\,\d\mm= -\int
_0^1\int_X\tilde\Gamma(f,
\sP{rh}\varphi)\,\d\mm\,\d r.
\end{equation}
Notice first that, possibly approximating $\varphi$ with the functions
$\varphi^\eps:=\frh^\eps\varphi$
whose Laplacian is in $L^\infty(X,\mm)$, in the proof of \eqref{eqgrosseto}
we can assume with no loss of generality that $\DeltaE\varphi\in
L^\infty(X,\mm)$. Indeed, \eqref{eqFisherbound}
and the strong convergence in $\Gamma$ norm of $ \sP{rh}\varphi^\eps
$ to $ \sP{rh} \varphi$ ensure the dominated
convergence of the integrals in the right-hand sides, while the
convergence of the left-hand sides is obvious.

Assuming $\DeltaE\varphi\in L^\infty(X,\mm)$, since
\[
\int_X\frac{ \sP h\varphi-\varphi}{h}g\,\d\mm=\int
_0^1\int_X g\DeltaE
\sP{rh}\varphi\,\d\mm\,\d r
\]
for all $g\in L^2(X,\mm)$ we can consider the truncated functions
$g_N=\min\{f,N\}$ and pass to the limit as
$N\to\infty$ to get that $f$ satisfies the same identity. Since
$\DeltaE\sP{rh}\varphi= \sP{rh}\DeltaE\varphi\in L^\infty(X,\mm
)$ we can
use \eqref{eqintbyparts1} to obtain \eqref{eqgrosseto}.

Having established \eqref{eqgrosseto}, the statement follows using
once more \eqref{eqFisherbound} and
the strong convergence of $ \sP{rh}\varphi$ to $\varphi$ in {$\cV$}.
\end{pf}

Under the same assumptions of Lemma~\ref{lest-velocity},
the same computation leading to \eqref{eqGKuwada}
(actually with a simplification, due to the fact that $\varphi$ is
independent on $s$)
and \eqref{eqdefDelta1} give
\[
\frac{\d}{\d s}\int_X \varphi\,\d\rho_{\st}=
\int_X \bigl(\dot \vartheta(s) \sP{st}\dot
f_{ \vartheta(s)}+t \sP{st} {\DeltaE^{(1)}}f_{ \vartheta
(s)}\bigr)
\varphi\,\d\mm
\]
for $\Leb1$-a.e. $s\in(0,1)$
and all $\varphi\in\cVu$. Here we used the fact that
${\DeltaE^{(1)}}( P_r g)= P_r{\DeltaE^{(1)}}g$ whenever
$g\in D({\DeltaE^{(1)}})$.
Since ${\DeltaE^{(1)}}f\in\rmC([0,1];L^1(X,\mm))$,
the right-hand side is a continuous function of $s$, hence
%
\begin{eqnarray}
\label{eqGKuwadabis} \frac{\d}{\d s}\int_X \varphi\,\d
\rho_{\st}=\int_X \bigl(\dot
\vartheta_s \sP{st}\dot f_{ \vartheta(s)} +t \sP{st} {
\DeltaE^{(1)}}f_{ \vartheta(s)}\bigr)\varphi\,\d\mm
\nonumber
\\[-8pt]
\\[-8pt]
\eqntext{\mbox{for every }s\in(0,1).}
\end{eqnarray}
For $\eps>0$, let us now consider the regularized entropy functionals
\begin{eqnarray}
E_\eps(\rho):=\int_X
e_\eps(f)\,\d\mm
\nonumber
\\
\eqntext{\mbox{where } e_\eps'(r):=\log\bigl(\eps+r\land
\eps^{-1}\bigr)\in\operatorname{Lip}\bigl([0,\infty)\bigr),
e_\eps(0)=0.}
\end{eqnarray}
Since we will mainly consider functions $f$ with finite Fisher
information, we will also introduce the function
\[
p_\eps(r):=e_\eps'\bigl(r^2\bigr)-
\log\eps=\log\bigl(\eps+r^2\land\eps^{-1}\bigr) -\log\eps .
\]
Since $p_\eps$ is also Lipschitz and $p_\eps(0)=0$, we have
\[
f\in L^1_+(X,\mm),\qquad  \mathsf F(f)<\infty\quad\Rightarrow\quad
e_\eps'(f)-\log\eps=p_\eps(\sqrt f)\in\cV.
\]

%
\begin{lemma}[(Derivative of $E_\eps$)]
\label{leEderivative}
With the same notation of Lemma~\ref{lest-velocity}, if $\rho$ is
regular and $t>0$ we have for
$g^\eps_{s,t}:=p_\eps(\sqrt{f_{s,t}})$
%
\begin{eqnarray}
\label{eqGmain} %
&&E_\eps(\rho_{1,t})
-E_\eps(\rho_{0,t})
\nonumber
\\[-8pt]
\\[-8pt]
&&\qquad \le \int_0^1 \biggl( -t\int
_X f_{\st}\Gamma\bigl(g^\eps_{s,t}
\bigr)\,\d\mm+ \dot\vartheta _s\int_X \sP{st}
\bigl(g^\eps_{\st} \bigr)\dot f_{ \vartheta(s)}\,\d\mm
\biggr)\,\d s. %
\nonumber
\end{eqnarray}
\end{lemma}

\begin{pf} The weak differentiability of $s\mapsto f_{\st}$
(namely, in duality with functions in $\cVu$) given
in \eqref{eqGKuwadabis} can, thanks to the continuity assumption made
on ${\DeltaE^{(1)}}f_s$, turned into strong $L^1(X,\mm)$ differentiability,
so that
%
\begin{eqnarray}
\label{eqGmain2} \frac{\d}{\d s}f_{\st}= \dot\vartheta_s
\sP{st}\dot f_{ \vartheta(s)}+t \sP{st}\bigl({\DeltaE^{(1)}}f_{
\vartheta
(s)}
\bigr)
\nonumber
\\[-8pt]
\\[-8pt]
\eqntext{\mbox{in $L^1(X,\mm)$, for all $s\in(0,1)$.}}
\end{eqnarray}
Since $e_\eps$ is of class $\rmC^{1,1}$, it is easy to check that
this implies the absolute continuity of
$s\mapsto E_\eps(\rho_{\st})$. In addition, the mean value theorem gives
\[
\frac{\d}{\d s}E_\eps(\rho_{\st})=\lim
_{h\to0}\int_X e_\eps
'(f_{\st})\frac{f_{s+h,t}-f_{\st}}{h}\,\d\mm\qquad\forall s
\in(0,1).
\]
Notice also that Lemma~\ref{lemintbyparts2} with $f=f_{\st}$ and
$\omega=e_\eps'-\log\eps$ gives
[since $e_\eps'(f_{\st})-\log\eps=
p_\eps(\sqrt{f_{\st}})$ is nonnegative and integrable and $ \sP
{st}({\DeltaE^{(1)}}f_{ \vartheta(s)})={\DeltaE^{(1)}}( \sP{st}f_{
\vartheta
(s)})$ has null mean]
%
\begin{equation}
\label{eqintbypartsDelta1} \int_X \sP{st}\bigl({\DeltaE^{(1)}}f_{ \vartheta(s)}
\bigr) e_\eps'(f_{\st}) \,\d\mm\leq-4 \int
_X f_{\st} e_\eps''(f_{\st})
\Gamma(\sqrt{f_{\st}})\,\d\mm.
\end{equation}
Now we use \eqref{eqGmain2}, \eqref{eqintbypartsDelta1}
and conclude
\begin{eqnarray*}
\frac{\d} {\d s}E_\eps(\rho_{\st})&\leq& -t\int
_X 4f_{\st}e_\eps''(f_{\st})
\Gamma(\sqrt{f_{\st}})\,\d\mm
\\
&&{}+ \dot\vartheta_s \int_X
\bigl(e_\eps'(f_{\st})-\log\eps\bigr) \sP{\st}
\dot f_{ \vartheta(s)}\,\d\mm.
\end{eqnarray*}
On the other hand, since
$4r e_\eps''(r)\ge4r^2 (e_\eps''(r) )^2=
r (p_\eps'(\sqrt r) )^2$,
we get
\[
- 4f_{\st}e_\eps''(f_{\st})
\Gamma(\sqrt{f_{\st}})\leq - f_{\st} \bigl(p_\eps'(
\sqrt{f_{\st}}) \bigr)^2 \Gamma(\sqrt{f_{\st}}) =
-f_{\st}\Gamma\bigl(p_\eps(\sqrt{f_{\st}})\bigr)
\]
and an integration with respect to $s$ and the definition of $g^\eps
_{\st}$ yield \eqref{eqGmain}.
\end{pf}

%
\begin{theorem}[(Action and entropy estimate on regular curves)]\label
{thmaction}
Let $\rho_s=f_s\mm$ be a regular curve. Then,
setting $\rho_{1,t}= \sfH_{t}\rho_1$,
it holds
%
\begin{equation}
\label{eq32} W_2^2(\rho_0,
\rho_{1,t})+ 2t \entv(\rho_{1,t}) \le\rmR_K^2(t)
\int_0^1 \llvert \dot\rho_s\rrvert
^2\,\d s+ 2t\entv(\rho_0),
\end{equation}
where
\[
\rmR_K(t):=\frac{t}{\rmI_{K}(t)}= \frac{Kt}{\rme^{Kt}-1}\qquad\mbox{if }K
\neq0;\qquad  \rmR_0(t)\equiv1.
\]
\end{theorem}

\begin{pf}
Set
$\rho_{\st},f_{\st}$ as in Lemma~\ref{lest-velocity},
$p_\eps(r)=e_\eps'(r^2)-\log\eps$,
$g^\eps_{\st}=p_\eps(\sqrt{f_{\st}})$ as
in Lemma~\ref{leEderivative},
$q_\eps(r):=\sqrt r(2- \sqrt r p_\eps'(\sqrt r))$,
and
$\varphi_s:=Q_s\varphi$ for
a Lipschitz function $\varphi$ with bounded support.

Notice that by \eqref{eq69}
\[
\tilde\Gamma(f_{s,t},\varphi_s)= 2\sqrt{f_{\st}}
\Gamma(\sqrt{f_{\st}},\varphi_s) =f_{\st} \Gamma
\bigl(g^\eps_{\st},\varphi_s\bigr)+
q_\eps(f_{\st})\Gamma(\sqrt{f_{\st}},
\varphi_s).
\]
Applying \eqref{eqGKuwada}, \eqref{eqGmain}
in the weaker form
\[
t E_\eps(\rho_{1,t}) -t E_\eps(
\rho_{0,t})\le \int_0^1 \biggl( t\dot
\vartheta_s\int_X \sP{st}
\bigl(g^\eps_{\st} \bigr)\dot f_{ \vartheta(s)}\,\d\mm -
\frac{t^2}{2}\int_X f_{\st} \Gamma
\bigl(g^\eps_{s,t} \bigr)\,\d\mm \biggr)\,\d s %
\]
and eventually the Young inequality $2xy\leq ax^2+y^2/a$ in \eqref{eqGconj}
with $a:=\dot\vartheta_s\rme^{-2Ks t}$,
we obtain
\begin{eqnarray*}
&&\int_X \varphi_1\,\d\rho_{1,t}-
\int_X \varphi_0\,\d\rho_{0,t} +t
\bigl( E_\eps(\rho_{1,t})-E_\eps(
\rho_{0,t}) \bigr)
\\
&&\qquad\le \int_0^1 \biggl( \dot
\vartheta_s \int_X \dot f_{ \vartheta(s)}
\sP{st} \bigl(\varphi_s+t g^\eps_{\st} \bigr)\,\d
\mm -\frac{1}2\int_X \bigl(\llvert \rmD
\varphi_s\rrvert ^2+t^2 \Gamma
\bigl(g^\eps_{\st}\bigr) \bigr)\,\d\rho_{\st}
\\
&&\hspace*{59pt}\quad \qquad {}- t\int_X \Gamma\bigl(g^\eps_{\st},
\varphi_s\bigr) \,\d\rho_{\st}- t\int_X
q_\eps(f_{s,t}) \Gamma(\sqrt{f_{\st}},
\varphi_s) \,\d\mm \biggr)\,\d s
\\
&&\qquad\le \int_0^1 \biggl( \dot
\vartheta_s \int_X \dot f_{ \vartheta(s)}
\sP{st} \bigl(\varphi_s+t g^\eps_{\st} \bigr)\,\d
\mm -\frac{1}2\int_X \Gamma\bigl(
\varphi_s+t g^\eps_{\st} \bigr)\,\d
\rho_{\st}
\\
&&\hspace*{131pt}\quad\qquad {}- t\int_X q_\eps(f_{s,t})
\Gamma(\sqrt{f_{\st}},\varphi_s) \,\d\mm \biggr)\,\d s
\\
&&\qquad\le \int_0^1 \biggl( \dot
\vartheta_s  \int_X \dot f_{ \vartheta(s)}
\sP{st} \bigl(\varphi_s+t g^\eps_{\st} \bigr)\,\d
\mm
\\
&&\hspace*{19pt}\qquad \quad {} -\frac{1}2 \rme^{2Kst} \int_X
\Gamma\bigl(\sP{st} \bigl(\varphi_s+t g^\eps_{\st}
\bigr)\bigr)\,\d \rho_{ \vartheta(s)}
\\
&&\hspace*{40pt}\quad\qquad{}+ t\int_X \bigl\llvert q_\eps(f_{\st})
\bigr\rrvert \bigl\llvert \Gamma(\sqrt{f_{\st}},\varphi_s)
\bigr\rrvert \,\d\mm \biggr)\,\d s
\\
&&\qquad\le \int_0^1 \biggl(\frac{1}2{
(\dot\vartheta_s)^2\rme^{-2Kst}}\llvert \dot
\rho_s\rrvert ^2 + \frac{t}{8} \delta\mathsf{F}(
\rho_{\st})+\frac{t}{2\delta}\int_X
q_\eps^2 (f_{\st})\llvert \rmD
\varphi_s\rrvert ^2\,\d\mm \biggr)\,\d s.
\end{eqnarray*}
Now we pass first to the limit as $\eps\down0$,
observing that $p_\eps'(r)=2r(\eps+r^2)^{-1}\nchi_{r^2<\eps^{-1}}$ gives
\[
q^2_\eps(r)=4r \biggl(1-\frac
{r}{\eps+r}
\biggr)^2\nchi_{r<\eps^{-1}}\le4r,\qquad \lim_{\eps\down0}q^2_\eps(r)=0,
\]
and then as $\delta\down0$;
choosing
\[
\vartheta(s):=\frac{\rmI_{K}(st)}{\rmI_{K}(t)}\quad \mbox{so that} \quad \dot\vartheta(s)=
\rmR_K(t)\rme^{Kst},
\]
we obtain
\[
\int_X \varphi_1\,\d\rho_{1,t}-
\int_X \varphi_0\,\d\rho_{0,t} +t
\bigl( \operatorname{Ent}_{\mm}(\rho_{1,t})-
\operatorname{Ent}_{\mm}(\rho _{0,t}) \bigr)\le
\frac{1}2 \rmR_K^2(t) \int_0^1
\llvert \dot\rho_s\rrvert ^2\,\d s.
\]
Eventually we take the supremum with respect to $\varphi$, obtaining
\[
\frac{1}2 W_2^2(\rho_{1,t},
\rho_{0}) +t \bigl( \operatorname{Ent}_{\mm}(
\rho_{1,t})-\operatorname{Ent}_{\mm}(\rho _{0,t})
\bigr)\le \frac{1}2 \rmR_K^2(t) \int
_0^1 \llvert \dot\rho_s\rrvert
^2\,\d s.
\]
\upqed
\end{pf}

%
%
\begin{theorem}[{[$\BE K\infty$ is equivalent to $\RCD K\infty$]}]
\label{thmmain-identification}
If $(X,\tau,\mm,\cE)$ is a Riemannian Energy measure space
satisfying \eqref{eq80} relative to ${\mathsf d}_\cE$ and $\BE
K\infty$,
then $(X,{\mathsf d}_\cE,\mm)$ is a $\RCD K\infty$ space.

Conversely, if $(X,{\mathsf d},\mm)$ is a $\RCD K\infty$ space then,
denoting by
$\tau$ the topology induced by ${\mathsf d}$ and by
$\cE=2\C$ the Cheeger energy,
$(X,\tau,\mm,\cE)$ is a Riemannian Energy measure space
satisfying ${\mathsf d}_\cE={\mathsf d}$, \eqref{eq80}, and $\BE
K\infty$.
\end{theorem}

\begin{pf}
We have to show that \eqref{defEVIK} holds with $\sH{t}\rho$
precisely given by the dual semigroup. {By a standard density argument
(see \cite{AGS11b}, Proposition~2.21) suffices
to show this property for all $\rho\in\probt{X}$ with finite entropy.}
By the semigroup property {and \eqref{eqextraextra}},
it is sufficient to prove \eqref{defEVIK} at $t=0$ {for all $\nu\in
\probt{X}$ with finite entropy}. For any
$\rho\in\AC2{[0,1]}{(\probt X, W_2)}$ joining $\rho_0:=\nu$ to
$\rho_1:=\rho$ we find regular curves $\rho^n$ as in
Proposition~\ref{propapprox}
{with $\entv(\rho^n_0)\leq\entv(\rho_0)$} and apply the action estimate
\eqref{eq32} to the curves $\rho^n_{\st}=\sfH_{st}\rho^n_s$ to obtain
\[
W_2^2\bigl(\sfH_t
\rho^n_1,\rho^n_0\bigr)+2t\entv
\bigl(\sfH_t\rho^n_1\bigr) \le
\rmR_K^2(t)\int_0^1
\bigl\llvert \dot\rho^n_s\bigr\rrvert ^2\,\d
s+ 2t\entv\bigl(\rho^{_0}\bigr).
\]
We pass to the limit as $n\to\infty$ and use the lower
semicontinuity of $W_2$ and of the entropy,
to get
\[
W_2^2(\sfH_t\rho,\nu)+ 2t \entv({\sP t\rho})
\le\rmR_K^2(t)\int_0^1
\llvert \dot\rho_s\rrvert ^2\,\d s + 2t\entv(\nu).
\]
We can now minimize w.r.t. $\rho$ and use the fact
that
$(\probt{X},W_2)$ is a length space because $(X,\mathsf{d}_\cE)$ is
(this can be obtained starting from an optimal Kantorovich plan $\ppi
$, choosing in a $\ppi$-measurable way
a $\eps$-optimal geodesic with constant speed
as in the proof of Theorem~\ref{thmKuwadaequivalence}), getting
\[
W_2^2(\sfH_t \rho,\nu)+2t\entv({\sP t\rho})
\le\rmR_K^2(t)W_2^2(\rho,\nu) +
2t\entv(\nu).
\]
After dividing by $t>0$, letting $t\down0$ and using
$\rmR_K(t)=1-\frac{K}2t+o(t)$
we obtain \eqref{defEVIK}.

The converse implication, from $\RCD{K}{\infty}$ to
$\BE K\infty$ has been proved in
\cite{AGS11b}, Section~6.
\end{pf}

We conclude with an immediate application of the previous result to
metric measure spaces:
it follows by Theorem~\ref{thmequivalence}
and
Corollary~\ref{corBE=contraction}. Notice that for
the Cheeger energy condition (\ref{EDb}) and
upper-regularity are always true.

%
\begin{corollary}
\label{corRCD-criterion}
Let $(X,{\mathsf d},\mm)$ be a metric measure space satisfying
\textup{(MD$+$exp)}
with a quadratic Cheeger energy $\C$
defining the Dirichlet form $\cE=2\C$ as in
\eqref{eq88}.
$(X,{\mathsf d},\mm)$ is a $ \operatorname{RCD}(K,\infty)$-space
if (and only if) at least one of the following properties hold:
\begin{longlist}[(iii)]
\item[(i)] $(X,\mathsf{d})$ is a length space and $(\sP t)_{t\ge0}$ satisfies property \eqref{eq44}, that is,
for every function $f\in D(\C)$ with $\llvert \rmD f\rrvert _w\le1$
and every $t>0$,
%
\begin{equation}
\label{eq89} \quad \sP t f\in\Lip_b(X),\qquad \llvert \rmD\sP t f
\rrvert ^2\le\rme^{-2K t} \sP t \bigl(\llvert \rmD f\rrvert
_w^2 \bigr)\qquad \mm\mbox{-a.e. in $X$}.
\end{equation}
\item[(ii)] Conditions (\ref{EDa}), \eqref{eqwFeller}
(or $\cL=\cL_\rmC$), and
$\BE K\infty$ hold.
\item[(iii)] Condition (\ref{EDa}) holds and
$(\sH{t})_{t\ge0}$ satisfies the contraction property
\eqref{eq56} [or $(\sP t)_{t\ge0}$ satisfies the Lipschitz bound
\eqref{eq43}].
\end{longlist}
\end{corollary}

\section{Applications of the equivalence result}
\label{secapplications}

In this section, we present two applications of our equivalence result:
in one direction, we can use it to prove that the $\RCD{K}{\infty}$
condition is stable under tensorization, a property proved in
\cite{AGS11b} only under a nonbranching assumption on the base spaces.
We will also prove the same property for
Riemannian Energy measure spaces
satisfying the $\BE KN$ condition, obtaining in particular
the natural bound on the dimension of the product.

In the other direction, we shall prove a stability result
for Riemannian Energy measure spaces satisfying a uniform $\BE KN$
condition under
Sturm--Gromov--Hausdorff convergence.

\subsection{Tensorization}

Let $(X,{\mathsf d}_X,\mm_X)$, $(Y,{\mathsf d}_Y,\mm_Y)$ be
$\RCD K\infty$
metric measure spaces.

We may define a product space $(Z,{\mathsf d},\mm)$ by
%
\begin{eqnarray}\label{eqproduct}
Z&:=&X\times Y,\qquad {\mathsf d} \bigl((x,y),\bigl(x',y'
\bigr) \bigr):=\sqrt{{\mathsf d}_X^2
\bigl(x,x'\bigr)+{\mathsf d} _Y^2
\bigl(y,y'\bigr)},
\nonumber
\\[-8pt]
\\[-8pt]
\mm&:=&\mm_X\times\mm_Y.
\nonumber
\end{eqnarray}
Notice that also $(Z,{\mathsf d},\mm)$ satisfies
the quantitative $\sigma$-finiteness condition
\eqref{eq80}.

Denoting by $\cE^X,\cE^Y$ the Dirichlet forms associated to
the respective (quad\-ra\-tic) Cheeger energies
with domains $\cV^X,\cV^Y$,
we consider the Cartesian--Dirichlet form
%
\begin{equation}
\label{eq115} \qquad \cEZ(f):=\int_Y \cE^X
\bigl(f^y\bigr)\,\d\mm_Y(y)+\int_X
\cE^X\bigl(f^x\bigr)\,\d\mm_X(x), \qquad f
\in L^2(Z,\mmZ),
\end{equation}
where for every $f\in L^2(Z,\mmZ)$ and
$z=(x,y)\in Z$
we set $f^x=f(x,\cdot)$, $f^y(\cdot)=f(\cdot,y)$.
By \cite{AGS11b}, Theorem~6.18,
the proper domain $\cVZ$ of $\cEZ$ in $L^2(Z,\mmZ)$
is the Hilbert space
\begin{eqnarray*}
\cVZ&:=& \bigl\{f\in L^2(Z,\mmZ)\dvtx f^x\in
\cV^Y\mbox{ for $\mm_X$-a.e. $x\in X$},
\\
&&\hphantom{ \bigl\{}f^y\in\cV^X\mbox{ for $
\mm_Y$-a.e. $y\in Y$}, 
\bigl\llvert \rmD
f^x\bigr\rrvert _w(y), \bigl\llvert \rmD
f^y\bigr\rrvert _w(x) \in L^2(Z,\mmZ) \bigr
\}.
\end{eqnarray*}
Furthermore, $\frac{1}2\cEZ$ coincides with the Cheeger energy $\C$ in
$(Z,\sfdZ,\mmZ)$, and
%
\begin{eqnarray}
\label{eqten22} \llvert \rmD f\rrvert _w^2(x,y)=
\Gamma(f) (x,y)= \bigl\llvert \rmD f^x\bigr\rrvert
^2_w(y)+\bigl\llvert \rmD f^y\bigr\rrvert
^2_w(x)
\nonumber
\\[-8pt]
\\[-8pt]
\eqntext{\mbox{for $\mmZ$-a.e. $(x,y)\in Z$.}}
\end{eqnarray}
Even though the result in \cite{AGS11b}
is stated for metric measure spaces with finite measure, the
proof extends with no difficulty to the $\sigma$-finite case.
Also, it is worthwhile to mention that the curvature assumption
on the base spaces plays
almost no role in the proof, it is only used to build,
via the product semigroup, an operator with good regularization properties
(specifically from $L^\infty$ to $\rmC_b$); see
\cite{AGS11b}, Lemma~6.13.

It will be convenient, as in \cite{AGS11b} and in the previous
sections, to work with a pointwise
defined version of the semigroups in the base spaces, namely
\[
\sP t^Xu(x):=\int u\bigl(x'\bigr)\,\d
\sfH^X_t\delta_x\bigl(x'
\bigr),\qquad \sP t^Yv(y):=\int v\bigl(y'\bigr)\,\d
\sfH^Y_t\delta_y\bigl(y'\bigr)
\]
for $u\dvtx X\to\R$ and $v\dvtx Y\to\R$ bounded Borel, where $(\sH
{t}^X)_{t\ge0}$
and $(\sH{t}^Y)_{t\ge0}$
denote the Wasserstein semigroups on the base spaces.
These pointwise defined semigroups also provide the continuous versions
of \eqref{eqFeller} for the base spaces, see
\cite{AGS11b}, Theorem~6.1(iii).

Since the heat flows are linear, the tensorization \eqref{eqten22}
implies a corresponding tensorization of the heat flows,
namely for all $g\dvtx Z\to\R$ bounded and Borel, for $\mm$-a.e.
$(x,y)\in Z$ the following identities hold:
%
\begin{eqnarray}
\label{eqten1} \sP t g(z)&=&\int_X \sP t^Y g
\bigl(x',\cdot\bigr) (y)\,\d\sH{t}^X\delta
_x\bigl(x'\bigr),
\nonumber
\\[-8pt]
\\[-8pt]
\sP t g(z)&=&\int_Y\sP t^X g\bigl(
\cdot,y'\bigr) (x)\,\d\sH{t}^Y\delta_y
\bigl(y'\bigr).
\nonumber
\end{eqnarray}
With these ingredients at hand, we can now prove
the main tensorization properties.

%
\begin{theorem}
With the above notation, if $(X,{\mathsf d}_X,\mm_X)$ and $(Y,{\mathsf
d}_Y,\mm
_Y)$ are $\RCD{K}{\infty}$ spaces then
the space $(Z,{\mathsf d},\mm)$ is $\RCD{K}{\infty}$ as well.
\end{theorem}

\begin{pf} According to the characterization of $\RCD{K}{\infty}$
given in point (i) of Corollary~\ref{corRCD-criterion},
since the Cheeger energy in $Z$ satisfies \eqref{eq88}
by the above mentioned result of \cite{AGS11b},
it suffices to
show that the length space property
and \eqref{eq89}
are stable under tensorization.

\emph{Stability of the length space property.} This is
simple to check, one obtains an almost minimizing geodesic
$\gamma\dvtx [0,1]\to Z$ combining almost minimizing geodesics on the base
spaces with constant speed and parameterized on $[0,1]$.

\emph{Stability of \textup{\eqref{eq89}}.}
Let
us first notice that
$\sP t$ maps bounded and Borel functions into continuous ones, thanks of
any of the two
identities in \eqref{eqten1} and \eqref{eqFeller}.

Let
$f\in\Lip_b(Z)\cap L^2(Z,\mm)$. Keeping $y$ initially fixed,
the second identity in \eqref{eqten1} tells us that
$x\mapsto\sP t f(x,y)=(\sP t f)^y(x)$ is the mean w.r.t. $y'$,
weighted with $\sfH_t^Y\delta_y$,
of the functions $\sP t^Xf^{y'}(x)$. Hence, the convexity of the slope gives
\[
\bigl\llvert \rmD(\sP t f)^y\bigr\rrvert (x)\leq \int
_Y\bigl\llvert \rmD\sP t^X f^{y'}
\bigr\rrvert (x)\,\d\sfH_t^Y\delta_y
\bigl(y'\bigr),
\]
where gradients are understood with respect to the first variable. We
can thus
use the H\"older inequality to get
%
\begin{equation}
\label{eqten2} \bigl\llvert \rmD(\sP t f)^y\bigr\rrvert
^2(x)\leq \int_Y\bigl\llvert \rmD\sP
t^X f^{y'}\bigr\rrvert ^2(x)\,\d
\sfH_t^Y\delta_y\bigl(y'
\bigr).
\end{equation}
Now, for $\mm_Y$-a.e. $y'\in Y$ we apply \eqref{eq89} in the space
$X$ to the functions
$f^{y'}$ and use Fubini's theorem to get
%
\begin{equation}
\label{eqten3} \bigl\llvert \rmD\sP t^X f^{y'}\bigr
\rrvert ^2(x)\leq\rme^{-2Kt}P^X_t
\bigl\llvert \rmD f^{y'}\bigr\rrvert _w^2(x)
\qquad\mbox{for $\mm_Y$-a.e. $y'\in Y$.}
\end{equation}
Combining \eqref{eqten2} and \eqref{eqten3} and using once more
\eqref{eqten1} with $g(x,y)=\llvert \rmD f^y\rrvert _w^2(x)$ we get
\[
\bigl\llvert \rmD(\sP t f)^y\bigr\rrvert ^2(x)\leq
\rme^{-2Kt}\int_Y\sP t^X\bigl\llvert
\rmD f^{y'}\bigr\rrvert _w^2(x)\,\d\sH{t}
\delta_y\bigl(y'\bigr)= \rme^{-2Kt}\sP t\bigl
\llvert \rmD f^y\bigr\rrvert _w^2(x).
\]
Repeating a similar argument with the first identity in \eqref
{eqten1} and adding the two inequalities,
we obtain
\[
\bigl\llvert \rmD(\sP t f)^y\bigr\rrvert ^2(x)+\bigl
\llvert \rmD(\sP t f)^x\bigr\rrvert ^2(y)\leq
\rme^{-2Kt}\sP t\llvert \rmD f\rrvert _w^2(x,y).
\]
We conclude that \eqref{eq89} holds in $(Z,{\mathsf d},\mm)$ using the
calculus lemma \cite{AGS11b}, Lem\-ma~6.2, which provides the
information that the square root of
$\llvert \rmD(\sP t f)^y\rrvert ^2(x)+\llvert \rmD(\sP t f)^x\rrvert ^2(y)$ is an upper gradient
of $\sP t f$. It follows that
$\rme^{-Kt}\sqrt{\sP t\llvert \rmD f\rrvert _w^2}$ is an upper
gradient as well;
being continuous, it provides a pointwise upper
bound for the slope.
\end{pf}

Let us now consider the corresponding version of the
tensorization
theorem for Riemannian Energy measure spaces {with a finite upper bound
on the dimension.}

%
%
\begin{theorem}
\label{thmBEtensor}
Let $(X,\tau_X,\cE^X,\mm_X)$, $(Y,\tau_Y,\cE^Y,\mm_Y)$ be
Riemannian Energy measure spaces satisfying
the Bakry--\'Emery conditions
$\BE K{N_X}$ and $\BE K{N_Y}$ respectively, and
let us consider the cartesian Dirichlet form
$\cEZ$ defined by \eqref{eq115} on
$Z=X\times Y$ endowed with the product topology
$\tau=\tau_X\otimes\tau_Y$ and the product measure $\mmZ$ as in
\eqref{eqproduct}.

Then $(Z,\tau,\mmZ,\cEZ)$ is a Riemannian Energy measure
space, it
satisfies the Bakry--\'Emery condition $\BE K{N_X+N_Y}$ and
the induced distance ${\mathsf d}_\cE$ on $Z$ coincides with
the product distance defined in \eqref{eqproduct}.
\end{theorem}

\begin{pf}
It is not restrictive to assume that $N_X, N_Y<\infty$; by the
previous Theorem we already
know that $(Z,\tau,\mmZ,\cEZ)$ is a Riemannian Energy
measure space satisfying $\BE K\infty$,
whose induced distance ${\mathsf d}_\cE$ is given by \eqref{eqproduct};
we want to prove that
\eqref{eqBEKnu.6} holds with $\nu_Z:=\nu_X\nu_Y/(\nu_X+\nu_Y)$ where
$\nu_X:=N_X^{-1}$ and $\nu_Y:=N_Y^{-1}$.
We argue as in
\eqref{eqten3}, observing that for $\mm_Y$-a.e. $y'\in Y$
\eqref{eqBEKnu.6} and \eqref{eqG0} yield
\begin{eqnarray}
\bigl\llvert \rmD\sP t^X f^{y'}\bigr\rrvert
_{{w}}^2+ 2\nu_X \rmI_{2K,2}(t) \bigl(
\Delta_X \sP t^X f^{y'} \bigr)^2
\le \rme^{-2Kt} \sP t^X\bigl\llvert \rmD f^{y'}
\bigr\rrvert _w^2
\nonumber
\\
\eqntext{\mbox{for $\mm_X$-a.e. $x\in X$.}}
\end{eqnarray}
Integrating w.r.t. the measure $\sfH^Y_t \delta_y$
in $y'$ and recalling \eqref{eqten1}
and \eqref{eqten2},
we get
%
\begin{equation}
\label{eq94} \bigl\llvert \rmD\bigl(\sP t f
\bigr)^{y}\bigr\rrvert _{{w}}^2(x)+ 2
\nu_X \rmI_{2K,2}(t) \bigl(\Delta_X (\sP t
f)^y\bigr)^2(x) \le \rme^{-2Kt} \sP t \bigl(
\bigl\llvert \rmD f^{y}\bigr\rrvert _w^2
\bigr) (x)\hspace*{-20pt}
\end{equation}
$\mmZ\mbox{-a.e. in }Z$,
where we also used the H\"older inequality
\begin{eqnarray*}
\int_Y \bigl(\Delta_X \sP t^X
f^{y'} \bigr)^2(x) \,\d\sH{t}\delta_y
\bigl(y'\bigr)&\ge& \biggl(\int_Y
\Delta_X \sP t^X f^{y'}(x) \,\d\sH{t}
\delta_y\bigl(y'\bigr) \biggr)^2
\\
&=&\bigl(\Delta_X \sP t f^y\bigr)^2(x).
\end{eqnarray*}
By repeating a similar argument inverting the role of $X$ and $Y$, we
get
%
\begin{equation}
\label{eq95} \bigl\llvert \rmD\bigl(\sP t f
\bigr)^{x}\bigr\rrvert _w^2(y)+ 2
\nu_Y \rmI_{2K,2}(t) \bigl(\Delta_Y (\sP t
f)^x\bigr)^2(y) \le \rme^{-2Kt} \sP t \bigl(
\bigl\llvert \rmD f^{x}\bigr\rrvert _w^2
\bigr) (y)\hspace*{-20pt}
\end{equation}
$\mmZ\mbox{-a.e. in }Z$.
Adding \eqref{eq94} and \eqref{eq95}, and recalling
the elementary inequality
\[
\nu_X a^2+\nu_Yb^2\ge
\frac{\nu_X\nu_Y}{\nu_X+\nu_Y}(a+b)^2 \qquad\mbox{for every }a, b\ge0,
\]
we conclude thanks to the next simple lemma.
\end{pf}

%
\begin{lemma}
\label{lelapcart}
Assume that $f\in\cV$ satisfies
%
\begin{eqnarray}
\label{eq91} %
\bigl(f^y,f^x\bigr)\in D(
\Delta_X)\times D(\Delta_Y)
\nonumber
\\[-8pt]
\\[-8pt]
\eqntext{\mbox{for $\mmZ$-a.e. $(x,y)\in Z$}, \Delta_X
f^y, \Delta_Y f^x\in L^2(Z,
\mmZ).} %
\end{eqnarray}
Then $f\in D(\Delta_Z)$ and
$\Delta_Z f(x,y)=\Delta_X
f^y(x)+\Delta_Yf^x(y)$ for $\mmZ$-a.e. $(x,y)\in Z$.
\end{lemma}

\begin{pf}
If \eqref{eq91} holds, Fubini's theorem
and the very definition of $\Delta_X, \Delta_Y$
yield for every $\varphi\in\cV$
\begin{eqnarray*}
\cE(f,\varphi)&=& \int_Y \cE^X
\bigl(f^y,\varphi^y\bigr)\,\d\mm_Y(y)+ \int
_X \cE^Y\bigl(f^x,
\varphi^x\bigr)\,\d\mm_X(x)
\\
&=&-\int_Y \biggl(\int_X
\Delta_X f^y \varphi^y\,\d
\mm_X \biggr)\, \d\mm_Y- \int_X
\biggl(\int_Y \Delta_Y f^x
\varphi^x\,\d\mm_Y \biggr)\,\d\mm _X
\\
&=& -\int_Z \bigl(\Delta_X f^y+
\Delta_Y f^x \bigr) \varphi\,\d\mm .
\end{eqnarray*}
\upqed
\end{pf}

\subsection{Sturm--Gromov--Hausdorff convergence and stability of the $\mathrm{BE}(K,N)$ condition}

\mbox{}\break\indent\textit{Preliminaries.}
Here and in the following we adopt the notation
$\bar\N:=\N\cup\{\infty\}$.

Let us first recall an equivalent characterization of Sturm--Gromov--Haus\-dorff
(SGH)
convergence of a sequence $(X_n,{\mathsf d}_n,\mm_n)$, $n\in\bar\N$,
of metric measure spaces that is very well adapted to our aims;
for the sake of simplicity, we restrict here to the case
when $\mm_n\in\probt{X_n}$. The general case of $\sigma$-finite measures
satisfying \eqref{eq80} could be attacked
by the techniques developed in \cite{AGMS12}, assuming that
\eqref{eq80} holds uniformly along the sequence.
We refer to \cite{AGS11b}, Section~2.3, \cite{Sturm06I}, Section~3.1,
for other definitions and important properties
of SGH convergence.

%
\begin{definition}[(SGH-convergence)]
\label{defSGH}
Let $(X_n,{\mathsf d}_n,\mm_n)$, $n\in\bar\N$, be
complete and separable metric measure spaces with $\mm_n\in\probt{X_n}$
for every $n\in\bar\N$. We say that
$(X_n,{\mathsf d}_n,\mm_n)$ SGH-converge to $(X_\infty,{\mathsf
d}_\infty
,\mm
_\infty)$
as $n\to\infty$ if
there exist a complete and separable metric space
$(X,{\mathsf d})$ and isometries $\iota_n\dvtx  X_n\to X$, $n\in\bar\N$,
such that $W_2((\iota_n)_\sharp\mm_n,(\iota_\infty)_\sharp\mm
_\infty)\to0$ as $n\to\infty$.
\end{definition}

A more intrinsic approach would state SGH-convergence
for the equivalence classes of metric-measure spaces
induced by measure-preserving isometries:
according to this point of view, two metric-measure spaces
$(X_1,{\mathsf d}_1,\mm_1)$ and $(X_2,{\mathsf d}_2,\mm_2)$ are
isomorphic if there exists an isometry
$\iota\dvtx \supp(\mm_1)\to X_2$ such that $\mm_2=\iota_\sharp\mm_1$.

Here, we do not insist on this aspect, since
owing to the Definition~\ref{defSGH}
we will always consider an effective realization of
such a convergence provided by the space $(X,{\mathsf d})$ and
the system of isometries $(\iota_n)$, $n\in\bar\N$.
Moreover, by identifying
$X_n$ with $\iota_n(X_n)$ in $X$, and
$\mm_n$ with $(\iota_n)_\sharp\mm_n$ in $\probt X$,
it will not restrictive to assume that
%
\begin{equation}
\label{eq96} %
\begin{tabular}{@{\quad\qquad}p{320pt}@{\hspace*{-5pt}}}$\mm_n\in
\probt X$, $X_n\subset X$, ${\mathsf d}_n\equiv{\mathsf
d}$ for every $n\in\bar\N$, $\mm_n$ are converging to $
\mm_\infty$ in $\probt X$ as $n\to\infty$;\end{tabular} %
\end{equation}
one has just to take care that in general
$\mm_n,\mm$ could be not fully supported.

%
\begin{remark}
\label{reminvariance}
It is important to notice that, by construction, the Cheeger energy
is invariant by isometries: considering, for example, the
situation of Definition~\ref{defSGH}, if
$\iota_\infty\dvtx X_\infty\to X$ is an isometric imbedding
of $(X_\infty,{\mathsf d}_\infty)$ in a complete and separable
metric space
$(X,{\mathsf d})$, with $\tilde\mm_\infty:=(\iota_\infty)_\sharp
\mm_\infty$,
the Cheeger energy $\frac{1}2\tilde\cE_\infty$ associated to
$(X,{\mathsf d},\tilde\mm_\infty)$ in $L^2(X,\tilde\mm_\infty)$ satisfies
\[
\tilde\cE_\infty(f)=\cE_\infty(f\circ\iota_\infty)\qquad
\mbox{for every }f\in L^2(X,\tilde\mm_\infty).
\]
Since the composition with $\iota_\infty$ provides an order preserving
isomorphism
between $L^2(X,\tilde\mm_\infty)$ and $L^2(X_\infty,\mm_\infty)$,
it is immediate to check that
$(X_\infty,\tau_\infty,\break\mm_\infty,\cE_\infty)$ satisfies $\BE
KN$ if
and only if
$(X,\tau,\tilde\mm_\infty,\tilde\cE_\infty)$ satisfies $\BE KN$ as
well
(here $\tau$ is the topology induced by ${\mathsf d}$ in $X$).
\end{remark}

We will also need
a few results, strongly related to the theory of Young measures,
concerning
convergence
for sequences of functions defined
in $L^2$-spaces associated to different measures
(see, e.g., \cite{AGS08}, Section~5.{4}).
We first make precise this notion of convergence.

%
%
\begin{definition}
\label{defconvergenceAGS}
Let $(X,{\mathsf d})$ be a complete and separable metric space,
let $(\mm_n)\subset\probt X$, $n\in\bar\N$, be converging in
$\probt X$,\vspace*{1pt}
and consider
a sequence of vector valued
functions $\ff_n \in L^2(X,\mm_n;\R^k)$,
$n\in\bar\N, k\in\N$.
We say that $(\ff_n)$ converges to $\ff_\infty$ as $n\to\infty$ if
%
\begin{equation}
\label{eq119} (\ii\times\ff_n)_\sharp\mm_n
\to(\ii\times\ff_\infty)_\sharp \mm_\infty\qquad \mbox{in
}\probt{X\times\R^k}.
\end{equation}
\end{definition}

We will use three properties stated in the next lemma.

%
%
\begin{lemma}
\label{leitemized}
Let $(X,{\mathsf d})$, $(\ff_n)$ and
$(\mm_n)\subset\probt X$, $n\in\bar\N$, as in Definition~\ref{defconvergenceAGS} above.
{\renewcommand\theenumi{\roman{enumi}}
\renewcommand\labelenumi{(\theenumi)}
\begin{longlist}[(iii)]
\item\label{itemA} \eqref{eq119} is equivalent to the convergence
of each component $f^j_n$, $j=1,\ldots,k$, to $f^j$.
\item\label{itemB} In the scalar case $k=1$, if $f_n$ satisfy
\begin{eqnarray*}
\lim_{n\to\infty}\int_X f_n
\varphi\,\d\mm_n&=& \int_X f\varphi\,\d\mm
\qquad\mbox{for every $\varphi\in\rmC _b(X)$},
\\
\lim_{n\to\infty}\int_X f_n^2
\,\d\mm_n&=& \int_X f^2\,\d\mm,
\end{eqnarray*}
then $f_n$ converges to $f$ according to \eqref{eq119}.
The same conclusion holds if
$f_n\in L^1_+(X,\mm_n)\cap L^\infty(X,\mm_n)$ are uniformly bounded
probability
densities satisfying
\[
f_n\mm_n\weakto f\mm\qquad\mbox{in }\prob
X,\qquad \lim_{n\to\infty}\int_X
f_n\log f_n\,\d\mm_n= \int
_X f\log f\,\d\mm.
\]
\item\label{itemC} Finally, if $\rr\dvtx \R^k\to\R^h$ is a continuous
map with linear growth, and $\ff_n$ converge to $\ff$ according to
\eqref{eq119} then $\rr\circ\ff_n$ converge to $\rr\circ\ff$.
\end{longlist}
}
\end{lemma}

\begin{pf}
(\ref{itemA}) follows by disintegration and
the fact that a probability measure in $\R^k$ is a Dirac mass
if and only if its coordinate projections are Dirac masses;
see, for example, \cite{AGS08}, Lemma~5.3.2.

Property {(\ref{itemB})} is a consequence of the fact that,
for strictly convex functions, equality holds in Jensen's inequality
only when the measure is a Dirac mass. A~detailed argument
is presented in \cite{AGS08}, Theorem~5.4.4 (the
fact that the base space $X$ is a general metric space instead of
an Hilbert space is not relevant here).

The proof of {(\ref{itemC})} is straightforward.
\end{pf}

\textit{Stability of $\BE KN$ under SGH-convergence.}

%
%
\begin{theorem}
\label{thmBE-stability}
Let $(X_n,\tau_n,\cE_n,\mm_n)$,
$n\in\N$,
be
Riemannian energy
measure spaces
satisfying $\BE KN $
with $\mm_n\in\probt{X_n}$ and let us
suppose that, denoting by ${\mathsf d}_n$ the corresponding distances
${\mathsf d}_{\cE_n}$, $(X_n,{\mathsf d}_n,\mm_n)$
converge to $(X_\infty,{\mathsf d}_\infty,\mm_\infty)$
in the Sturm--Gromov--Hausdorff sense
of Definition~\ref{defSGH}.
If $\frac{1}2\cE_\infty$ is the Cheeger energy
in the limit space and $\tau_\infty$ the
topology induced by ${\mathsf d}_\infty$, then
$(X_\infty,\tau_\infty,\mm_\infty,\allowbreak \cE_\infty)$ is a
Riemannian Energy measure
space
satisfying $\BE KN $.
\end{theorem}

\begin{pf}
According to the Definition~\ref{defSGH}
of SGH-convergence and taking Remark~\ref{reminvariance} into
account, it is not restrictive to
assume that \eqref{eq96} holds, so that
all the spaces $X_n$ are subsets of a fixed
complete and separable metric space $(X,{\mathsf d})$,
${\mathsf d}_n$ are the restrictions of
${\mathsf d}$ on $X_n$,
the isometries $\iota_n$ are just the inclusions maps,
$\mm_n$ can be identified with $(\iota_n)_\sharp\mm_n$
and can be considered as measures in $\probt X$
converging to $\mm_\infty$, and the Cheeger energies $\cE_n$ are
Dirichlet forms on $L^2(X,\mm_n)$
satisfying the $\BE KN$ condition.

The case $N=\infty$ follows by the identification Theorem~\ref{thmmain-identification} and
\cite{AGS11b}, Theorem~6.10.
In particular,
the limit Cheeger energy associated to $(X,{\mathsf d},\mm_\infty)$
is a Dirichlet form that we call $\frac{1}2\cE_\infty$
and
the limit space endowed with the Cheeger energy
and the
topology $\tau$ induced by ${\mathsf d}$
is a
Riemannian Energy measure space.

We can thus consider the case $N<\infty$.
%
%
In order to show that $(X,\tau,\mm_\infty,\cE_\infty)$
satisfies $\BE KN$
we will prove that
the distributional characterization
\eqref{eqBEKnu.3} of $\BE KN $ holds
for every $f\in L^2(X,\mm_\infty)$ and
nonnegative $\varphi\in L^\infty(X,\mm_\infty)$.
By standard approximation,
it is also not restrictive to assume
$f\in L^\infty(X,\mm_\infty)$.
Our argument consists in passing to the limit in
the corresponding distributional inequality
written for suitable approximating sequences
in the spaces $(X,\tau,\mm_n,\cE_n)$.

Let us thus denote by $(\sP t^n)_{t\ge0}$ the Markov semigroups\vspace*{2pt}
in $L^2(X,\mm_n)$ with generators {
$\Delta_n:=\Delta_{\cE^n}$},
$n\in\bar\N$.
By \cite{AGS11b}, Lemma~6.12, and
Lemma~\ref{leitemized}(\ref{itemB}), for every
$f,\varphi\in L^\infty(X,\mm)$, $\varphi$ nonnegative,
we can find
sequences $f_n,\varphi_n\in L^\infty(X,\mm_n)$,
$\varphi_n$ nonnegative,
converging to $f,\varphi$ according to Definition~\ref{defconvergenceAGS}.
We can also suppose that $f_n,\varphi_n$ are uniformly bounded
by some constant $C>0$.

Applying Lemma~\ref{leextended-convergence} below,
we get that
for every $t\ge0$ and $s\in[0,t]$
$ \sP{t-s}^nf_n$ converge to $ \sP{t-s}^\infty f$
and $ \sP s^n\varphi_n$ converge to $ \sP s^\infty\varphi$
as $n\to\infty$ according to Definition~\ref{defconvergenceAGS}.

Applying Lemma~\ref{leitemized}(\ref{itemA}) to the function
$\ff_n:=( \sP{t-s}^nf_n,
\sP s^n\varphi_n)$ and choosing the bounded and continuous test function
$\psi(x,r_1,r_2)=r_1^2r_2\rmS_C(r_2)$,\break $(x,r_1,r_2)\in X\times
\R^2$ [with $\rmS_C$ defined as in \eqref{eq10tris}],
we obtain
\begin{eqnarray*}
\lim_{n\to\infty} \int_X \bigl(
\sP{t-s}^n f_n\bigr)^2 \sP s^n
\varphi_n\,\d\mm_n&=& \lim_{n\to\infty}\int
\psi\,\d(\ii\times\ff_n)_\sharp\mm_n= \int\psi
\,\d(\ii\times\ff)_\sharp\mm
\\
&=& \int_X \bigl( \sP{t-s}^\infty f
\bigr)^2 \sP s^\infty\varphi\,\d\mm.
\end{eqnarray*}
A similar argument yields
\[
\lim_{n\to\infty} \int_X \bigl(
\DeltaE^n P^n_{t-s} f_n
\bigr)^2 P^n_s\varphi_n\,\d
\mm_n= \int_X \bigl(\DeltaE
\sP{t-s}^\infty f\bigr)^2 \sP s^\infty\varphi\,\d\mm
\]
{for every }
$t>0, s\in[0,t)$.
We can thus pass to the limit in
the distributional inequality
\eqref{eqBEKnu.3} written for $f_n,\varphi_n$.
\end{pf}

%
%
\begin{lemma}
\label{leextended-convergence}
Let $(X,{\mathsf d})$ be a complete and separable metric space,
$\mm_n\in\probt X$, $n\in\bar\N$, be a converging sequence
such that $(X,{\mathsf d},\mm_n)$ is a $\RCD K\infty$ space
with Cheeger energy $\frac{1}2\cE_n$,
and let us denote by $(\sP t^n)_{t\ge0}$ the Markov semigroups
in $L^2(X,\mm_n)$ with generators
{$\Delta_n:=\Delta_{\cE^n}$},
$n\in\bar\N$.

If $f_n\in L^\infty(X,\mm_n)$
converge to $f_\infty\in L^\infty(X,\mm_\infty)$
according to Definition~\ref{defconvergenceAGS},
with uniformly bounded $L^\infty$ norm,
then
$\sP t^{n} f_n$
converge to
$\sP t^\infty f_\infty$
for every $t\ge0$
and $\Delta_n \sP t^n f_n$
converge to
$\Delta_\infty\sP t^\infty f_\infty$
as $n\to\infty$ for every $t>0$.
\end{lemma}

\begin{pf}
When $f_n$ are probability densities the convergence
of $ \sP t^n f_n$
follows
by applying Lemma~\ref{leitemized}(\ref{itemB}) and
the convergence results of
\cite{AGS11b}, Theorem~6.11
[which shows that $ \sP t^n f_n\mm_n$ converges
to $\sP t^\infty f_\infty\mm_\infty$ in $\probt X$] and
\cite{AGMS12} [which yields the convergence of the entropies
$\operatorname{Ent}_{\mm_n}(f_n\mm_n)\to
\operatorname{Ent}_{\mm}(f_\infty\mm_\infty)$].

The case $f_n\in L^1_+(X,\mm_n)$ can be easily reduced to the previous
one by a rescaling, since $\int_X f_n\,\d\mm_n\to\int_X f\,\d\mm$
by \eqref{eq119}
and $ (\sP t^n)_{t\ge0}$ is mass preserving.

The general case can be proved by decomposing each
$f_n$ into the difference $f_n^+-f_n^-$ of
its positive and negative part, observing that
$f_n^\pm$ converge to $f^\pm_\infty$ thanks to Lemma~\ref{leitemized}(\ref{itemC}).
Thus,
by (\ref{itemA}) it follows that $(\sP t^n f_n^+,\sP t^n f_n^-)$
converge to
$(\sP t^\infty f^+_\infty,\sP t^\infty f^-_\infty)$
and a further application of (\ref{itemC}) yields the convergence
result by the linearity
of the semigroups.

In order to prove the convergence of $\Delta_n \sP t^n f_n$
we still apply (\ref{itemB}) of Lem\-ma~\ref{leitemized}: recall that
$\sP t^n$ are analytic semigroups in $L^2(X,\mm_n)$,
$\Delta_n\sP t^n f_n=\frac\d{\dt} \sP t^nf_n$, and
the uniform estimates
(see, e.g., \cite{Stein70bis}, page~75, step 2)
%
\begin{equation}
\label{eq104} t^j\biggl\llVert \frac{\d^j}{\dt^j} \sP
t^nf_n\biggr\rrVert _{L^2(X,\mm_n)}\le
A_j\llVert f_n\rrVert _{L^2(X,\mm_n)}\qquad\mbox{for
every }t>0, n\in\N\hspace*{-20pt}
\end{equation}
hold with universal constants $A_j$ for every $t>0$. Since we just
proved that
for every $\varphi\in\rmC_b(X)$
the sequence of functions
$\zeta_n(t):=\int_X \sP t^n
f_n \varphi\,\d\mm_n$ converge pointwise
to the corresponding $\zeta$ as $n\to\infty$,
\eqref{eq104} yields that
\[
\lim_{n\to\infty}\zeta_n'(t)= \lim
_{n\to\infty}\int_X \Delta_n \sP
t^n f_n \varphi\,\d\mm_n=
\zeta'(t) =\int_X \Delta_\infty\sP
t^\infty f_\infty\varphi\,\d\mm
\]
for every $t>0$.
The same argument holds for
\[
t\mapsto\int_X \bigl(\Delta_n \sP
t^n f_n\bigr)^2\,\d\mm_n= {
\frac{1}4} \frac{\d^{2}}{\dt^{2}}\int_X \bigl( \sP
t^n f_n\bigr)^2\,\d\mm_n.
\]
\upqed
\end{pf}

\section*{Acknowledgements}
We would like to thank Michel Ledoux for useful discussions on
various aspects of the theory of Dirichlet forms and
$\Gamma$-calculus.


%

\printaddresses


\begin{thebibliography}{57}

\bibitem{Ambrosio-Colombo-DiMarino12}
%
\begin{bmisc}[author]
\bauthor{\bsnm{{Ambrosio}},~\bfnm{L.}\binits{L.}},
\bauthor{\bsnm{{Colombo}},~\bfnm{M.}\binits{M.}} \AND
\bauthor{\bsnm{{Di Marino}},~\bfnm{S.}\binits{S.}}
(\byear{2014}).
\bhowpublished{Sobolev spaces in metric measure spaces: Reflexivity
and lower semicontinuity of slope. \textit{Adv. Stud. Pure
Math}. To appear. Available at \arxivurl{arXiv:1212.3779}.}
\end{bmisc}
%
\bptok{imsref}%
\endbibitem

\bibitem{AGMR12}
%
\begin{bmisc}[author]
\bauthor{\bsnm{Ambrosio},~\bfnm{L.}\binits{L.}},
\bauthor{\bsnm{Gigli},~\bfnm{N.}\binits{N.}},
\bauthor{\bsnm{Mondino},~\bfnm{A.}\binits{A.}} \AND
\bauthor{\bsnm{Rajala},~\bfnm{T.}\binits{T.}}
(\byear{2014}).
\bhowpublished{Riemannian {R}icci curvature lower bounds in metric
measure spaces with $\sigma$-finite measure. \textit{Trans. Amer.
Math. Soc.} To appear. Available at \arxivurl{arXiv:1207.4924v2}.}
\end{bmisc}
%
\bptok{imsref}%
\endbibitem

\bibitem{AGS08}
%
\begin{bbook}[mr]
\bauthor{\bsnm{Ambrosio},~\bfnm{Luigi}\binits{L.}},
\bauthor{\bsnm{Gigli},~\bfnm{Nicola}\binits{N.}} \AND
\bauthor{\bsnm{Savar{\'e}},~\bfnm{Giuseppe}\binits{G.}}
(\byear{2008}).
\btitle{Gradient Flows in Metric Spaces and in the Space of
Probability Measures},
\bedition{2nd} ed.
\bpublisher{Birkh\"auser},
\blocation{Basel}.
\bid{mr={2401600}}
\end{bbook}
%
\bptok{imsref}%
\endbibitem

\bibitem{AGS11c}
%
\begin{barticle}[mr]
\bauthor{\bsnm{Ambrosio},~\bfnm{Luigi}\binits{L.}},
\bauthor{\bsnm{Gigli},~\bfnm{Nicola}\binits{N.}} \AND
\bauthor{\bsnm{Savar{\'e}},~\bfnm{Giuseppe}\binits{G.}}
(\byear{2013}).
\btitle{Density of {L}ipschitz functions and equivalence of weak
gradients in metric measure spaces}.
\bjournal{Rev. Mat. Iberoam.}
\bvolume{29}
\bpages{969--996}.
\bid{doi={10.4171/RMI/746}, issn={0213-2230}, mr={3090143}}
\end{barticle}
%
\bptok{imsref}%
\endbibitem

\bibitem{AGS11b}
\begin{barticle}[mr]
\bauthor{\bsnm{Ambrosio},~\bfnm{Luigi}\binits{L.}},
\bauthor{\bsnm{Gigli},~\bfnm{Nicola}\binits{N.}} \AND
\bauthor{\bsnm{Savar{\'e}},~\bfnm{Giuseppe}\binits{G.}}
(\byear{2014}).
\btitle{Metric measure spaces with {R}iemannian {R}icci curvature bounded from below}.
\bjournal{Duke Math. J.}
\bvolume{163}
\bpages{1405--1490}.
\bid{doi={10.1215/00127094-2681605}, issn={0012-7094}, mr={3205729}}
\end{barticle}
\bptok{imsref}%
\endbibitem

\bibitem{AGS11a}
%
\begin{barticle}[mr]
\bauthor{\bsnm{Ambrosio},~\bfnm{Luigi}\binits{L.}},
\bauthor{\bsnm{Gigli},~\bfnm{Nicola}\binits{N.}} \AND
\bauthor{\bsnm{Savar{\'e}},~\bfnm{Giuseppe}\binits{G.}}
(\byear{2014}).
\btitle{Calculus and heat flow in metric measure spaces and
applications to spaces with {R}icci bounds from below}.
\bjournal{Invent. Math.}
\bvolume{195}
\bpages{289--391}.
\bid{doi={10.1007/s00222-013-0456-1}, issn={0020-9910}, mr={3152751}}
\end{barticle}
%
\bptok{imsref}%
\endbibitem

\bibitem{Ambrosio-Savare-Zambotti09}
%
\begin{barticle}[mr]
\bauthor{\bsnm{Ambrosio},~\bfnm{Luigi}\binits{L.}},
\bauthor{\bsnm{Savar{\'e}},~\bfnm{Giuseppe}\binits{G.}} \AND
\bauthor{\bsnm{Zambotti},~\bfnm{Lorenzo}\binits{L.}}
(\byear{2009}).
\btitle{Existence and stability for {F}okker--{P}lanck equations with
log-concave reference measure}.
\bjournal{Probab. Theory Related Fields}
\bvolume{145}
\bpages{517--564}.
\bid{doi={10.1007/s00440-008-0177-3}, issn={0178-8051}, mr={2529438}}
\end{barticle}
%
\bptok{imsref}%
\endbibitem

\bibitem{Ane-et-al00}
%
\begin{bbook}[mr]
\bauthor{\bsnm{An{\'e}},~\bfnm{C{\'e}cile}\binits{C.}},
\bauthor{\bsnm{Blach{\`e}re},~\bfnm{S{\'e}bastien}\binits{S.}},
\bauthor{\bsnm{Chafa{\"{\i}}},~\bfnm{Djalil}\binits{D.}},
\bauthor{\bsnm{Foug{\`e}res},~\bfnm{Pierre}\binits{P.}},
\bauthor{\bsnm{Gentil},~\bfnm{Ivan}\binits{I.}},
\bauthor{\bsnm{Malrieu},~\bfnm{Florent}\binits{F.}},
\bauthor{\bsnm{Roberto},~\bfnm{Cyril}\binits{C.}} \AND
\bauthor{\bsnm{Scheffer},~\bfnm{Gr{\'e}gory}\binits{G.}}
(\byear{2000}).
\btitle{Sur les In\'egalit\'es de {S}obolev Logarithmiques}.
\bseries{Panoramas et Synth\`eses [Panoramas and Syntheses]}
\bvolume{10}.
\bpublisher{Soci\'et\'e Math\'ematique de France},
\blocation{Paris}.
\bid{mr={1845806}}
\end{bbook}
%
\bptok{imsref}%
\endbibitem

\bibitem{Arnaoudon-Driver-Thalmaier07}
%
\begin{barticle}[mr]
\bauthor{\bsnm{Arnaudon},~\bfnm{Marc}\binits{M.}},
\bauthor{\bsnm{Driver},~\bfnm{Bruce~K.}\binits{B.~K.}} \AND
\bauthor{\bsnm{Thalmaier},~\bfnm{Anton}\binits{A.}}
(\byear{2007}).
\btitle{Gradient estimates for positive harmonic functions by
stochastic analysis}.
\bjournal{Stochastic Process. Appl.}
\bvolume{117}
\bpages{202--220}.
\bid{doi={10.1016/j.spa.2006.07.002}, issn={0304-4149}, mr={2290193}}
\end{barticle}
%
\bptok{imsref}%
\endbibitem

\bibitem{Arnaudon-Plank-Thalmaier03}
%
\begin{barticle}[mr]
\bauthor{\bsnm{Arnaudon},~\bfnm{Marc}\binits{M.}},
\bauthor{\bsnm{Plank},~\bfnm{Holger}\binits{H.}} \AND
\bauthor{\bsnm{Thalmaier},~\bfnm{Anton}\binits{A.}}
(\byear{2003}).
\btitle{A {B}ismut type formula for the {H}essian of heat semigroups}.
\bjournal{C. R. Math. Acad. Sci. Paris}
\bvolume{336}
\bpages{661--666}.
\bid{doi={10.1016/S1631-073X(03)00123-7}, issn={1631-073X}, mr={1988128}}
\end{barticle}
%
\bptok{imsref}%
\endbibitem

\bibitem{Bakry06}
%
\begin{bincollection}[mr]
\bauthor{\bsnm{Bakry},~\bfnm{Dominique}\binits{D.}}
(\byear{2006}).
\btitle{Functional inequalities for {M}arkov semigroups}.
In \bbooktitle{Probability Measures on Groups: Recent Directions and Trends}
\bpages{91--147}.
\bpublisher{Tata Inst. Fund. Res.},
\blocation{Mumbai}.
\bid{mr={2213477}}
\end{bincollection}
%
\bptok{imsref}%
\endbibitem

\bibitem{Bakry-Emery84}
%
\begin{bincollection}[mr]
\bauthor{\bsnm{Bakry},~\bfnm{D.}\binits{D.}} \AND
\bauthor{\bsnm{{\'E}mery},~\bfnm{Michel}\binits{M.}}
(\byear{1985}).
\btitle{Diffusions hypercontractives}.
In \bbooktitle{S\'eminaire de Probabilit\'es, {XIX}, 1983/84}.
\bseries{Lecture Notes in Math.}
\bvolume{1123}
\bpages{177--206}.
\bpublisher{Springer},
\blocation{Berlin}.
\bid{doi={10.1007/BFb0075847}, mr={0889476}}
\end{bincollection}
%
\bptok{imsref}%
\endbibitem

\bibitem{Bakry-Gentil-Ledoux14}
%
\begin{bbook}[author]
\bauthor{\bsnm{Bakry},~\bfnm{D.}\binits{D.}},
\bauthor{\bsnm{Gentil},~\bfnm{I.}\binits{I.}} \AND
\bauthor{\bsnm{Ledoux},~\bfnm{M.}\binits{M.}}
(\byear{2014}).
\btitle{Analysis and Geometry of Markov Diffusion Operators}.
\bseries{Grundlehren der Mathematischen Wissenschaften}
\bvolume{348}.
\bpublisher{Springer},
\blocation{Cham}.
\end{bbook}
%
\bptok{imsref}%
\endbibitem

\bibitem{Bakry-Ledoux96}
%
\begin{barticle}[mr]
\bauthor{\bsnm{Bakry},~\bfnm{D.}\binits{D.}} \AND
\bauthor{\bsnm{Ledoux},~\bfnm{M.}\binits{M.}}
(\byear{1996}).
\btitle{L\'evy--{G}romov's isoperimetric inequality for an
infinite-dimensional diffusion generator}.
\bjournal{Invent. Math.}
\bvolume{123}
\bpages{259--281}.
\bid{doi={10.1007/s002220050026}, issn={0020-9910}, mr={1374200}}
\end{barticle}
%
\bptok{imsref}%
\endbibitem

\bibitem{Bakry-Ledoux06}
%
\begin{barticle}[mr]
\bauthor{\bsnm{Bakry},~\bfnm{Dominique}\binits{D.}} \AND
\bauthor{\bsnm{Ledoux},~\bfnm{Michel}\binits{M.}}
(\byear{2006}).
\btitle{A logarithmic {S}obolev form of the {L}i-{Y}au parabolic inequality}.
\bjournal{Rev. Mat. Iberoam.}
\bvolume{22}
\bpages{683--702}.
\bid{doi={10.4171/RMI/470}, issn={0213-2230}, mr={2294794}}
\end{barticle}
%
\bptok{imsref}%
\endbibitem

\bibitem{Biroli-Mosco95}
%
\begin{barticle}[mr]
\bauthor{\bsnm{Biroli},~\bfnm{M.}\binits{M.}} \AND
\bauthor{\bsnm{Mosco},~\bfnm{U.}\binits{U.}}
(\byear{1995}).
\btitle{A {S}aint-{V}enant type principle for {D}irichlet forms on
discontinuous media}.
\bjournal{Ann. Mat. Pura Appl. (4)}
\bvolume{169}
\bpages{125--181}.
\bid{doi={10.1007/BF01759352}, issn={0003-4622}, mr={1378473}}
\end{barticle}
%
\bptok{imsref}%
\endbibitem

\bibitem{Bogachev07}
%
\begin{bbook}[mr]
\bauthor{\bsnm{Bogachev},~\bfnm{V.~I.}\binits{V.~I.}}
(\byear{2007}).
\btitle{Measure Theory. {V}ol. {I}, {II}}.
\bpublisher{Springer},
\blocation{Berlin}.
\bid{doi={10.1007/978-3-540-34514-5}, mr={2267655}}
\end{bbook}
%
\bptok{imsref}%
\endbibitem

\bibitem{Bouleau-Hirsch91}
%
\begin{bbook}[author]
\bauthor{\bsnm{Bouleau},~\bfnm{N.}\binits{N.}} \AND
\bauthor{\bsnm{Hirsch},~\bfnm{F.}\binits{F.}}
(\byear{1991}).
\btitle{Dirichlet Forms and Analysis on {W}iener Sapces}.
\bseries{De Gruyter Studies in Mathematics}
\bvolume{14}.
\bpublisher{De Gruyter},
\blocation{Berlin}.
\end{bbook}
%
\bptok{imsref}%
\endbibitem

\bibitem{Burago-Burago-Ivanov01}
%
\begin{bbook}[mr]
\bauthor{\bsnm{Burago},~\bfnm{Dmitri}\binits{D.}},
\bauthor{\bsnm{Burago},~\bfnm{Yuri}\binits{Y.}} \AND
\bauthor{\bsnm{Ivanov},~\bfnm{Sergei}\binits{S.}}
(\byear{2001}).
\btitle{A Course in Metric Geometry}.
\bseries{Graduate Studies in Mathematics}
\bvolume{33}.
\bpublisher{Amer. Math. Soc.},
\blocation{Providence, RI}.
\bid{mr={1835418}}
\end{bbook}
%
\bptok{imsref}%
\endbibitem

\bibitem{Cheeger00}
%
\begin{barticle}[mr]
\bauthor{\bsnm{Cheeger},~\bfnm{J.}\binits{J.}}
(\byear{1999}).
\btitle{Differentiability of {L}ipschitz functions on metric measure spaces}.
\bjournal{Geom. Funct. Anal.}
\bvolume{9}
\bpages{428--517}.
\bid{doi={10.1007/s000390050094}, issn={1016-443X}, mr={1708448}}
\end{barticle}
%
\bptok{imsref}%
\endbibitem

\bibitem{Cordero-McCann-Schmuckenschlager01}
%
\begin{barticle}[mr]
\bauthor{\bsnm{Cordero-Erausquin},~\bfnm{Dario}\binits{D.}},
\bauthor{\bsnm{McCann},~\bfnm{Robert~J.}\binits{R.~J.}} \AND
\bauthor{\bsnm{Schmuckenschl{\"a}ger},~\bfnm{Michael}\binits{M.}}
(\byear{2001}).
\btitle{A~{R}iemannian interpolation inequality \`a la {B}orell,
{B}rascamp and {L}ieb}.
\bjournal{Invent. Math.}
\bvolume{146}
\bpages{219--257}.
\bid{doi={10.1007/s002220100160}, issn={0020-9910}, mr={1865396}}
\end{barticle}
%
\bptok{imsref}%
\endbibitem

\bibitem{Daneri-Savare08}
%
\begin{barticle}[mr]
\bauthor{\bsnm{Daneri},~\bfnm{Sara}\binits{S.}} \AND
\bauthor{\bsnm{Savar{\'e}},~\bfnm{Giuseppe}\binits{G.}}
(\byear{2008}).
\btitle{Eulerian calculus for the displacement convexity in the
{W}asserstein distance}.
\bjournal{SIAM J. Math. Anal.}
\bvolume{40}
\bpages{1104--1122}.
\bid{doi={10.1137/08071346X}, issn={0036-1410}, mr={2452882}}
\end{barticle}
%
\bptok{imsref}%
\endbibitem

\bibitem{Erbar10}
%
\begin{barticle}[mr]
\bauthor{\bsnm{Erbar},~\bfnm{Matthias}\binits{M.}}
(\byear{2010}).
\btitle{The heat equation on manifolds as a gradient flow in the
{W}asserstein space}.
\bjournal{Ann. Inst. Henri Poincar\'e Probab. Stat.}
\bvolume{46}
\bpages{1--23}.
\bid{doi={10.1214/08-AIHP306}, issn={0246-0203}, mr={2641767}}
\end{barticle}
%
\bptok{imsref}%
\endbibitem

\bibitem{Fukaya87}
%
\begin{barticle}[mr]
\bauthor{\bsnm{Fukaya},~\bfnm{Kenji}\binits{K.}}
(\byear{1987}).
\btitle{Collapsing of {R}iemannian manifolds and eigenvalues of
{L}aplace operator}.
\bjournal{Invent. Math.}
\bvolume{87}
\bpages{517--547}.
\bid{doi={10.1007/BF01389241}, issn={0020-9910}, mr={0874035}}
\end{barticle}
%
\bptok{imsref}%
\endbibitem

\bibitem{Fukushima-Oshima-Takeda11}
%
\begin{bbook}[mr]
\bauthor{\bsnm{Fukushima},~\bfnm{Masatoshi}\binits{M.}},
\bauthor{\bsnm{Oshima},~\bfnm{Yoichi}\binits{Y.}} \AND
\bauthor{\bsnm{Takeda},~\bfnm{Masayoshi}\binits{M.}}
(\byear{2011}).
\btitle{Dirichlet Forms and Symmetric {M}arkov Processes},
\bedition{extended} ed.
\bseries{de Gruyter Studies in Mathematics}
\bvolume{19}.
\bpublisher{de Gruyter},
\blocation{Berlin}.
\bid{mr={2778606}}
\end{bbook}
%
\bptok{imsref}%
\endbibitem

\bibitem{Gigli10}
%
\begin{barticle}[mr]
\bauthor{\bsnm{Gigli},~\bfnm{Nicola}\binits{N.}}
(\byear{2010}).
\btitle{On the heat flow on metric measure spaces: Existence,
uniqueness and stability}.
\bjournal{Calc. Var. Partial Differential Equations}
\bvolume{39}
\bpages{101--120}.
\bid{doi={10.1007/s00526-009-0303-9}, issn={0944-2669}, mr={2659681}}
\end{barticle}
%
\bptok{imsref}%
\endbibitem

\bibitem{Gigli12}
%
\begin{bmisc}[author]
\bauthor{\bsnm{Gigli},~\bfnm{N.}\binits{N.}}
(\byear{2014}).
\bhowpublished{On the differential structure of metric measure spaces
and applications. \textit{Mem. Amer. Math. Soc.} To appear. Available at
\arxivurl{arXiv:1205.6622}.}
\end{bmisc}
%
\bptok{imsref}%
\endbibitem

\bibitem{GigliKuwadaOhta10}
%
\begin{barticle}[mr]
\bauthor{\bsnm{Gigli},~\bfnm{Nicola}\binits{N.}},
\bauthor{\bsnm{Kuwada},~\bfnm{Kazumasa}\binits{K.}} \AND
\bauthor{\bsnm{Ohta},~\bfnm{Shin-Ichi}\binits{S.-I.}}
(\byear{2013}).
\btitle{Heat flow on {A}lexandrov spaces}.
\bjournal{Comm. Pure Appl. Math.}
\bvolume{66}
\bpages{307--331}.
\bid{doi={10.1002/cpa.21431}, issn={0010-3640}, mr={3008226}}
\end{barticle}
%
\bptok{imsref}%
\endbibitem

\bibitem{AGMS12}
%
\begin{bmisc}[author]
\bauthor{\bsnm{Gigli},~\bfnm{N.}\binits{N.}},
\bauthor{\bsnm{Mondino},~\bfnm{A.}\binits{A.}} \AND
\bauthor{\bsnm{Savar{\'e}},~\bfnm{G.}\binits{G.}}
(\byear{2013}).
\bhowpublished{Convergence of pointed non-compact metric measure
spaces and stability of {R}icci curvature bounds and heat flows.   Available at \arxivurl{arXiv:1311.4907}.}
\end{bmisc}
%
\bptok{imsref}%
\endbibitem

\bibitem{Hsu02}
%
\begin{bbook}[mr]
\bauthor{\bsnm{Hsu},~\bfnm{Elton~P.}\binits{E.~P.}}
(\byear{2002}).
\btitle{Stochastic Analysis on Manifolds}.
\bseries{Graduate Studies in Mathematics}
\bvolume{38}.
\bpublisher{Amer. Math. Soc.},
\blocation{Providence, RI}.
\bid{mr={1882015}}
\end{bbook}
%
\bptok{imsref}%
\endbibitem

\bibitem{Jordan-Kinderlehrer-Otto98}
%
\begin{barticle}[mr]
\bauthor{\bsnm{Jordan},~\bfnm{Richard}\binits{R.}},
\bauthor{\bsnm{Kinderlehrer},~\bfnm{David}\binits{D.}} \AND
\bauthor{\bsnm{Otto},~\bfnm{Felix}\binits{F.}}
(\byear{1998}).
\btitle{The variational formulation of the {F}okker--{P}lanck equation}.
\bjournal{SIAM J. Math. Anal.}
\bvolume{29}
\bpages{1--17}.
\bid{doi={10.1137/S0036141096303359}, issn={0036-1410}, mr={1617171}}
\end{barticle}
%
\bptok{imsref}%
\endbibitem

\bibitem{Kendall86}
%
\begin{barticle}[mr]
\bauthor{\bsnm{Kendall},~\bfnm{Wilfrid~S.}\binits{W.~S.}}
(\byear{1986}).
\btitle{Nonnegative {R}icci curvature and the {B}rownian coupling property}.
\bjournal{Stochastics}
\bvolume{19}
\bpages{111--129}.
\bid{doi={10.1080/17442508608833419}, issn={0090-9491}, mr={0864339}}
\end{barticle}
%
\bptok{imsref}%
\endbibitem

\bibitem{Koskela-Shanmugalingam-Zhou12}
\begin{barticle}[mr]
\bauthor{\bsnm{Koskela},~\bfnm{Pekka}\binits{P.}},
\bauthor{\bsnm{Shanmugalingam},~\bfnm{Nageswari}\binits{N.}} \AND
\bauthor{\bsnm{Zhou},~\bfnm{Yuan}\binits{Y.}}
(\byear{2014}).
\btitle{Geometry and analysis of {D}irichlet forms ({II})}.
\bjournal{J. Funct. Anal.}
\bvolume{267}
\bpages{2437--2477}.
\bid{doi={10.1016/j.jfa.2014.07.015}, issn={0022-1236}, mr={3250370}}
\bptnote{check year}%
\end{barticle}
\bptok{imsref}%
\endbibitem

\bibitem{Kuwada10}
%
\begin{barticle}[mr]
\bauthor{\bsnm{Kuwada},~\bfnm{Kazumasa}\binits{K.}}
(\byear{2010}).
\btitle{Duality on gradient estimates and {W}asserstein controls}.
\bjournal{J. Funct. Anal.}
\bvolume{258}
\bpages{3758--3774}.
\bid{doi={10.1016/j.jfa.2010.01.010}, issn={0022-1236}, mr={2606871}}
\end{barticle}
%
\bptok{imsref}%
\endbibitem

\bibitem{Kuwada-Sturm13}
%
\begin{barticle}[mr]
\bauthor{\bsnm{Kuwada},~\bfnm{Kazumasa}\binits{K.}} \AND
\bauthor{\bsnm{Sturm},~\bfnm{Karl-Theodor}\binits{K.-T.}}
(\byear{2013}).
\btitle{Monotonicity of time-dependent transportation costs and
coupling by reflection}.
\bjournal{Potential Anal.}
\bvolume{39}
\bpages{231--263}.
\bid{doi={10.1007/s11118-012-9327-4}, issn={0926-2601}, mr={3102986}}
\end{barticle}
%
\bptok{imsref}%
\endbibitem

\bibitem{Ledoux01}
%
\begin{bbook}[mr]
\bauthor{\bsnm{Ledoux},~\bfnm{Michel}\binits{M.}}
(\byear{2001}).
\btitle{The Concentration of Measure Phenomenon}.
\bseries{Mathematical Surveys and Monographs}
\bvolume{89}.
\bpublisher{Amer. Math. Soc.},
\blocation{Providence, RI}.
\bid{mr={1849347}}
\end{bbook}
%
\bptok{imsref}%
\endbibitem

\bibitem{Ledoux04}
%
\begin{bincollection}[mr]
\bauthor{\bsnm{Ledoux},~\bfnm{Michel}\binits{M.}}
(\byear{2004}).
\btitle{Spectral gap, logarithmic {S}obolev constant, and geometric bounds}.
In \bbooktitle{Surveys in Differential Geometry. {V}ol. {IX}}
\bpages{219--240}.
\bpublisher{Int. Press},
\blocation{Somerville, MA}.
\bid{mr={2195409}}
\end{bincollection}
%
\bptok{imsref}%
\endbibitem

\bibitem{Ledoux11}
%
\begin{bincollection}[mr]
\bauthor{\bsnm{Ledoux},~\bfnm{Michel}\binits{M.}}
(\byear{2011}).
\btitle{From concentration to isoperimetry: Semigroup proofs}.
In \bbooktitle{Concentration, Functional Inequalities and Isoperimetry}.
\bseries{Contemp. Math.}
\bvolume{545}
\bpages{155--166}.
\bpublisher{Amer. Math. Soc.},
\blocation{Providence, RI}.
\bid{doi={10.1090/conm/545/10770}, mr={2858471}}
\end{bincollection}
%
\bptok{imsref}%
\endbibitem

\bibitem{Lisini07}
%
\begin{barticle}[mr]
\bauthor{\bsnm{Lisini},~\bfnm{Stefano}\binits{S.}}
(\byear{2007}).
\btitle{Characterization of absolutely continuous curves in
{W}asserstein spaces}.
\bjournal{Calc. Var. Partial Differential Equations}
\bvolume{28}
\bpages{85--120}.
\bid{doi={10.1007/s00526-006-0032-2}, issn={0944-2669}, mr={2267755}}
\end{barticle}
%
\bptok{imsref}%
\endbibitem

\bibitem{Lott-Villani-Poincare}
%
\begin{barticle}[mr]
\bauthor{\bsnm{Lott},~\bfnm{John}\binits{J.}} \AND
\bauthor{\bsnm{Villani},~\bfnm{C{\'e}dric}\binits{C.}}
(\byear{2007}).
\btitle{Weak curvature conditions and functional inequalities}.
\bjournal{J. Funct. Anal.}
\bvolume{245}
\bpages{311--333}.
\bid{doi={10.1016/j.jfa.2006.10.018}, issn={0022-1236}, mr={2311627}}
\end{barticle}
%
\bptok{imsref}%
\endbibitem

\bibitem{Lott-Villani09}
%
\begin{barticle}[mr]
\bauthor{\bsnm{Lott},~\bfnm{John}\binits{J.}} \AND
\bauthor{\bsnm{Villani},~\bfnm{C{\'e}dric}\binits{C.}}
(\byear{2009}).
\btitle{Ricci curvature for metric-measure spaces via optimal transport}.
\bjournal{Ann. of Math. (2)}
\bvolume{169}
\bpages{903--991}.
\bid{doi={10.4007/annals.2009.169.903}, issn={0003-486X}, mr={2480619}}
\end{barticle}
%
\bptok{imsref}%
\endbibitem

\bibitem{Ma-Rockner92}
%
\begin{bbook}[author]
\bauthor{\bsnm{Ma},~\bfnm{Zhi-Ming}\binits{Z.-M.}} \AND
\bauthor{\bsnm{R{\"o}ckner},~\bfnm{M.}\binits{M.}}
(\byear{1992}).
\btitle{Introduction to the Theory of (Non-Symmetric) Dirichlet Forms}.
\bpublisher{Springer},
\blocation{New York}.
\end{bbook}
%
\bptok{imsref}%
\endbibitem

\bibitem{McCann97}
%
\begin{barticle}[mr]
\bauthor{\bsnm{McCann},~\bfnm{Robert~J.}\binits{R.~J.}}
(\byear{1997}).
\btitle{A convexity principle for interacting gases}.
\bjournal{Adv. Math.}
\bvolume{128}
\bpages{153--179}.
\bid{doi={10.1006/aima.1997.1634}, issn={0001-8708}, mr={1451422}}
\end{barticle}
%
\bptok{imsref}%
\endbibitem

\bibitem{Sturm-Ohta-CPAM}
%
\begin{barticle}[mr]
\bauthor{\bsnm{Ohta},~\bfnm{Shin-Ichi}\binits{S.-I.}} \AND
\bauthor{\bsnm{Sturm},~\bfnm{Karl-Theodor}\binits{K.-T.}}
(\byear{2009}).
\btitle{Heat flow on {F}insler manifolds}.
\bjournal{Comm. Pure Appl. Math.}
\bvolume{62}
\bpages{1386--1433}.
\bid{doi={10.1002/cpa.20273}, issn={0010-3640}, mr={2547978}}
\end{barticle}
%
\bptok{imsref}%
\endbibitem

\bibitem{Otto-Villani00}
%
\begin{barticle}[mr]
\bauthor{\bsnm{Otto},~\bfnm{F.}\binits{F.}} \AND
\bauthor{\bsnm{Villani},~\bfnm{C.}\binits{C.}}
(\byear{2000}).
\btitle{Generalization of an inequality by {T}alagrand and links with
the logarithmic {S}obolev inequality}.
\bjournal{J. Funct. Anal.}
\bvolume{173}
\bpages{361--400}.
\bid{doi={10.1006/jfan.1999.3557}, issn={0022-1236}, mr={1760620}}
\end{barticle}
%
\bptok{imsref}%
\endbibitem

\bibitem{Otto-Westdickenberg05}
%
\begin{barticle}[mr]
\bauthor{\bsnm{Otto},~\bfnm{Felix}\binits{F.}} \AND
\bauthor{\bsnm{Westdickenberg},~\bfnm{Michael}\binits{M.}}
(\byear{2005}).
\btitle{Eulerian calculus for the contraction in the {W}asserstein distance}.
\bjournal{SIAM J. Math. Anal.}
\bvolume{37}
\bpages{1227--1255 (electronic)}.
\bid{doi={10.1137/050622420}, issn={0036-1410}, mr={2192294}}
\end{barticle}
%
\bptok{imsref}%
\endbibitem

\bibitem{Pazy83}
%
\begin{bbook}[mr]
\bauthor{\bsnm{Pazy},~\bfnm{A.}\binits{A.}}
(\byear{1983}).
\btitle{Semigroups of Linear Operators and Applications to Partial
Differential Equations}.
\bseries{Applied Mathematical Sciences}
\bvolume{44}.
\bpublisher{Springer},
\blocation{New York}.
\bid{doi={10.1007/978-1-4612-5561-1}, mr={0710486}}
\end{bbook}
%
\bptok{imsref}%
\endbibitem

\bibitem{Rajala11}
%
\begin{barticle}[mr]
\bauthor{\bsnm{Rajala},~\bfnm{Tapio}\binits{T.}}
(\byear{2012}).
\btitle{Local {P}oincar\'e inequalities from stable curvature
conditions on metric spaces}.
\bjournal{Calc. Var. Partial Differential Equations}
\bvolume{44}
\bpages{477--494}.
\bid{doi={10.1007/s00526-011-0442-7}, issn={0944-2669}, mr={2915330}}
\end{barticle}
%
\bptok{imsref}%
\endbibitem

\bibitem{Stein70bis}
%
\begin{bbook}[mr]
\bauthor{\bsnm{Stein},~\bfnm{Elias~M.}\binits{E.~M.}}
(\byear{1970}).
\btitle{Topics in Harmonic Analysis Related to the
{L}ittlewood-{P}aley Theory.}
\bseries{Annals of Mathematics Studies}
\bvolume{63}.
\bpublisher{Princeton Univ. Press},
\blocation{Princeton, NJ.}
\bid{mr={0252961}}
\end{bbook}
%
\bptok{imsref}%
\endbibitem

\bibitem{Stollmann10}
%
\begin{barticle}[mr]
\bauthor{\bsnm{Stollmann},~\bfnm{Peter}\binits{P.}}
(\byear{2010}).
\btitle{A dual characterization of length spaces with application to
{D}irichlet metric spaces}.
\bjournal{Studia Math.}
\bvolume{198}
\bpages{221--233}.
\bid{doi={10.4064/sm198-3-2}, issn={0039-3223}, mr={2650987}}
\end{barticle}
%
\bptok{imsref}%
\endbibitem

\bibitem{Sturm95}
%
\begin{barticle}[mr]
\bauthor{\bsnm{Sturm},~\bfnm{Karl-Theodor}\binits{K.-T.}}
(\byear{1995}).
\btitle{Analysis on local {D}irichlet spaces. {II}. {U}pper {G}aussian
estimates for the fundamental solutions of parabolic equations}.
\bjournal{Osaka J. Math.}
\bvolume{32}
\bpages{275--312}.
\bid{issn={0030-6126}, mr={1355744}}
\end{barticle}
%
\bptok{imsref}%
\endbibitem

\bibitem{Sturm98}
%
\begin{bincollection}[mr]
\bauthor{\bsnm{Sturm},~\bfnm{Karl-Theodor}\binits{K.-T.}}
(\byear{1998}).
\btitle{The geometric aspect of {D}irichlet forms}.
In \bbooktitle{New Directions in {D}irichlet Forms}.
\bseries{AMS/IP Stud. Adv. Math.}
\bvolume{8}
\bpages{233--277}.
\bpublisher{Amer. Math. Soc.},
\blocation{Providence, RI}.
\bid{mr={1652282}}
\end{bincollection}
%
\bptok{imsref}%
\endbibitem


\bibitem{Sturm06II}
%
\begin{barticle}[mr]
\bauthor{\bsnm{Sturm},~\bfnm{Karl-Theodor}\binits{K.-T.}}
(\byear{2006}).
\btitle{On the geometry of metric measure spaces. {I}}.
\bjournal{Acta Math.}
\bvolume{196}
\bpages{65--131}.
\bid{doi={10.1007/s11511-006-0002-8}, issn={0001-5962}, mr={2237206}}
\end{barticle}
%
\bptok{imsref}%
\endbibitem

\bibitem{Sturm06I}
%
\begin{barticle}[mr]
\bauthor{\bsnm{Sturm},~\bfnm{Karl-Theodor}\binits{K.-T.}}
(\byear{2006}).
\btitle{On the geometry of metric measure spaces. {II}}.
\bjournal{Acta Math.}
\bvolume{196}
\bpages{133--177}.
\bid{doi={10.1007/s11511-006-0003-7}, issn={0001-5962}, mr={2237207}}
\end{barticle}
%
\bptok{imsref}%
\endbibitem


\bibitem{Villani09}
%
\begin{bbook}[mr]
\bauthor{\bsnm{Villani},~\bfnm{C{\'e}dric}\binits{C.}}
(\byear{2009}).
\btitle{Optimal Transport: Old and New}.
\bseries{Grundlehren der Mathematischen Wissenschaften [Fundamental
Principles of Mathematical Sciences]}
\bvolume{338}.
\bpublisher{Springer},
\blocation{Berlin}.
\bid{doi={10.1007/978-3-540-71050-9}, mr={2459454}}
\end{bbook}
%
\bptok{imsref}%
\endbibitem

\bibitem{Sturm-VonRenesse05}
%
\begin{barticle}[mr]
\bauthor{\bparticle{von} \bsnm{Renesse},~\bfnm{Max-K.}\binits
{M.-K.}} \AND
\bauthor{\bsnm{Sturm},~\bfnm{Karl-Theodor}\binits{K.-T.}}
(\byear{2005}).
\btitle{Transport inequalities, gradient estimates, entropy, and
{R}icci curvature}.
\bjournal{Comm. Pure Appl. Math.}
\bvolume{58}
\bpages{923--940}.
\bid{doi={10.1002/cpa.20060}, issn={0010-3640}, mr={2142879}}
\end{barticle}
%
\bptok{imsref}%
\endbibitem

\bibitem{Wang11}
%
\begin{barticle}[mr]
\bauthor{\bsnm{Wang},~\bfnm{Feng-Yu}\binits{F.-Y.}}
(\byear{2011}).
\btitle{Equivalent semigroup properties for the curvature-dimension condition}.
\bjournal{Bull. Sci. Math.}
\bvolume{135}
\bpages{803--815}.
\bid{doi={10.1016/j.bulsci.2011.07.005}, issn={0007-4497}, mr={2838102}}
\end{barticle}
%
\bptok{imsref}%
\endbibitem

\end{thebibliography}
\end{document}